\documentclass[english,a4paper,11pt]{article}
\usepackage[a4paper,hmargin=1.0in,vmargin=1.0in]{geometry}
\usepackage[round]{natbib}
\usepackage[usenames,dvipsnames]{xcolor}
\usepackage[linktocpage=true,pagebackref=true,colorlinks,allcolors=blue,bookmarks=true,bookmarksopen,bookmarksnumbered]{hyperref}
\usepackage{sectsty}
\usepackage[english]{babel}
\usepackage[utf8]{inputenc}
\usepackage{amsfonts}
\usepackage{amssymb}
\usepackage{amsmath}
\usepackage{stmaryrd}
\usepackage{comment}
\usepackage{setspace}
\usepackage{bbm}
\usepackage{mathtools}
\usepackage{enumitem}
\usepackage{amsthm}
\usepackage{algorithm}
\usepackage{algorithmic}
\usepackage{tikz}
\usepackage{comment}
\usepackage{thmtools}
\usepackage{thm-restate}
\usepackage{booktabs}
\usepackage{array}
\usepackage{cleveref}

\newcommand{\N}{\mathbb{N}}

\newcommand{\E}{\mathbb{E}}

\newcommand{\obj}{{\cal E}}
\newcommand{\Teps}{\epsilon}
\newcommand{\stat}{\mathrm{Static}}
\newcommand{\thr}{{\cal B}} 
\newcommand{\emm}{{\cal M}}
\newcommand{\bl}{\boldsymbol{\ell}}
\newcommand{\Nc}{\mathcal{N}}

\newcommand{\opt}{{\sf OPT}}
\newcommand{\ord}{\mathrm{WO}}
\newcommand{\dpp}{\mathrm{DP}}
\newcommand{\DMLA}{\sf{Dynamic-MLA}}
\newcommand{\SMLA}{\sf{Static-MLA}}
\renewcommand{\P}{\mathbb{P}}
\newcommand{\eps}{\epsilon}
\newcommand{\ost}{{I}}

\newcommand{\ayy}{I}

\newcommand{\fil}{{\sf FILL }}
\newcommand{\const}{{\cal L }}
\newcommand{\fiil}{$\eps$-{\sf FILL }}
\newcommand{\comm}{-}

\newcommand{\sumdyn}{C^{\dpp}}
\newcommand{\sumstat}{C}
\newcommand{\ex}{{\cal X}}
\newcommand{\MyAbove}[2]{\genfrac{}{}{0pt}{}{#1}{#2}}
\def\argmax{\mathop{\rm arg\,max}}
\newcommand{\blue}[1]{{#1}}

\newcommand{\red}[1]{{{#1}}}

\newcommand{\stkout}[1]{\ifmmode\text{\sout{\ensuremath{#1}}}\else\sout{#1}\fi}

\newtheoremstyle{DStheorem}
  {\topsep}
  {\topsep}
  {\itshape}
  {0pt}
  {\scshape}
  {.}
  { }
  {\thmname{#1}\thmnumber{ #2}\thmnote{ (#3)}}
\theoremstyle{DStheorem}
\newtheorem{theorem}{Theorem}[section]
\newtheorem{lemma}[theorem]{Lemma}
\newtheorem{claim}[theorem]{Claim}

\let\oldproofname=\proofname
\renewcommand{\proofname}{\rm\sc{\oldproofname}}

\sectionfont{\large} \subsectionfont{\normalsize}
\allowdisplaybreaks
\makeindex

\newcounter{dannycounter}

\newcounter{omarcounter}

\newcounter{marouanecounter}

\begin{document}

\begin{titlepage}

\title{Maximum Load Assortment Optimization: \\
Approximation Algorithms and Adaptivity Gaps}
\author{%
Omar El Housni\thanks{School of Operations Research and Information Engineering, Cornell Tech, Cornell University. Email: {\tt \{oe46,mi262\}@cornell.edu}.}
\and
Marouane Ibn Brahim\footnotemark[1]
\and
Danny Segev\thanks{Department of Statistics and Operations Research, School of Mathematical Sciences, Tel Aviv University, Tel Aviv 69978, Israel. Email: {\tt segevdanny@tauex.tau.ac.il}. Supported by Israel Science Foundation grant 1407/20.}}
\date{}
\maketitle

\setcounter{page}{200}
\thispagestyle{empty}

\begin{abstract}
Motivated by modern-day applications such as Attended Home Delivery and Preference-based Group Scheduling, where decision makers wish to steer a large number of customers toward choosing the exact same alternative, we introduce a novel class of assortment optimization problems, referred to as {\em Maximum Load Assortment Optimization}. In such settings, given a universe of substitutable products, we are facing a stream of customers, each choosing between either selecting a product out of an offered assortment or opting to leave without making a selection. Assuming that these decisions are  governed by the Multinomial Logit choice model, we define the random {\em load} of any underlying product as the total number of customers who select it. Our objective is to offer an assortment of products to each customer so that the expected maximum load across all products is maximized.  

We consider both static and dynamic formulations of the maximum load assortment optimization problem. In the static setting, a single offer set is carried throughout the entire process of customer arrivals, whereas in the dynamic setting, the decision maker offers a personalized assortment to each customer, based on the entire information available at that time. As can only be expected, both formulations present a wide range of computational challenges and analytical questions. The main contribution of this paper resides in proposing efficient algorithmic approaches for computing near-optimal static and dynamic assortment policies. In particular, we develop a polynomial-time approximation scheme (PTAS) for the static problem formulation. Additionally, we demonstrate that an elegant policy utilizing weight-ordered assortments yields a $1/2$-approximation. Concurrently, we prove that such policies are sufficiently strong to provide a $1/4$-approximation with respect to the dynamic formulation, establishing a constant-factor bound on its adaptivity gap.  Finally, we design an adaptive policy whose expected maximum load is within factor $1-\eps$ of optimal, admitting a quasi-polynomial time implementation.
\end{abstract}


\end{titlepage}

\setcounter{page}{200}
\thispagestyle{empty}
\tableofcontents

\newpage
\setcounter{page}{1}
\section{Introduction}
    
Assortment optimization forms one of the most fundamental problems in revenue management, arising in a wide spectrum of application domains such as retailing and online advertising. At a high level, in such settings, the decision maker wishes to decide on a subset of products, picked out of a given universe, that will be offered to arriving customers in order to optimize a certain objective function. At least traditionally, each customer either chooses a single product from the offered assortment or decides to leave without making any purchase, with choice probabilities that are captured by a {\em discrete choice model}. The vast majority of assortment optimization models are guided by having either revenue maximization or sales maximization as their objective function. In the former case, each underlying product is associated with a fixed selling price, and the goal is to identify an assortment that maximizes the expected revenue due to a single representative customer, where the price of each product within this assortment is weighted by its corresponding choice probability. Problems of this form arise, for instance, when an online retailer displays a subset of products from a large universe in order to maximize the expected revenue. 
On the other hand, in sales maximization, our goal is to determine an assortment that maximizes the expected market share, given by the probability that a customer would purchase a product from the offered set. For instance, publishers such as Google Ads or Microsoft Ads may wish to select a subset of online ads to display, aiming to maximize the probability that customers will click on one of these ads. For a comprehensive overview of classical assortment optimization problems and their applications, we refer the reader to related surveys and books \citep{kok2009assortment,phillips2021pricing,gallego2019revenue}.

\paragraph{Informal model description.} In this paper, we introduce and study a new class of assortment optimization problems where, informally speaking, our goal is to identify, either statically or adaptively, assortments that would steer a large number of customers towards choosing the exact same product. Deferring the formal model formulation to be discussed in Section~\ref{sec:form}, given a universe of substitutable products, we are facing a \blue{finite} stream of customers, each choosing between either selecting a product out of an offered assortment or opting to leave without making a selection. Assuming that these decisions are governed by the Multinomial Logit (MNL) choice model, we define the random {\em load} of any underlying product as the total number of customers who select it along the arrival sequence. This way, the number of customers who choose the most selected product corresponds to the maximum load across all products. Our objective is to offer an assortment of products to each customer so that the expected maximum load across all products is maximized. We refer to problem formulations along these lines as {\em Maximum Load Assortment Optimization}. Specifically, we consider both the static formulation (\ref{SMLA}), where a single offer set should be kept unchanged for the entire sequence of customer arrivals, and the dynamic setting (\ref{DMLA}), in which the decision maker offers a personalized assortment to each customer, based on the entire information available at that time, taking into account the choices of all previously-arriving customers. 
As we proceed to show next, the above-mentioned objective function is motivated  by real-life applications in e-commerce such as Attended Home Delivery (AHD), where customers have the option to select their delivery slot, as well as in scheduling platforms, where users select a common time slot to meet.

\paragraph{Attended Home Delivery.}  Online supermarket chains such as WholeFoods, FreshDirect, and AmazonFresh provide customers with various delivery time slots to choose from, based on their individual preferences. Similarly, e-retailers such as Amazon, Wayfair, and Walmart allow their customers to select an appealing delivery time among the available options. The dominant business model in the grocery delivery sector is known as Attended Home Delivery (AHD) \citep{manerba2018attended}. This model entails the customer's presence during the delivery process, necessitating an agreement on a specific time slot between the e-grocer and the customer.
To optimize delivery costs, e-commerce platforms are interested in {\em packing} as many customers as possible from the same geographical area  into the same time slot. In their survey on this topic, \citet[Sec.~2.2]{wassmuth2023demand} highlight the importance of managing customer demand: ``{\em   Demand management aims to manage the resulting trade-offs between captured demand (revenue) and assembly and delivery efficiency (costs)}''.

In the context of AHD, there are two primary strategies for managing customer demand: offering and pricing. In the offering strategy, decision makers determine which delivery time slots will be presented to customers and which slots will be hidden
\citep{casazza2016optimizing,truden2022computational,van2024machine,wassmuth2023demand}. \blue{For example, \cite{mackert2019choice} introduces a dynamic time slot management framework under the generalized attraction model \citep{gallego2015general}, which starts by approximating the opportunity cost of a given customer request, and then employs this approximation to formulate a non-linear integer program and its linearization, in order to determine the time slot assortment}.  On the other hand, in the pricing strategy, prices are assigned to each delivery slot in order to influence customers choices
\citep{campbell2006incentive}. \blue{For instance, \cite{yang2017approximate} propose a dynamic programming framework for the delivery time slot pricing problem, and show in MNL-based simulation that these pricing policies can improve profitability by over $2\%$, when compared to simple fixed price policies.} 
In this paper, we focus on exploiting the offering strategy to steer customers towards selecting the same time slot. To this end, while booking their delivery times, the platform can guide customers  by strategically determining the assortment of time slots to offer. This objective seamlessly aligns with our framework, where each delivery time slot can be viewed as a product, meaning that the ``load'' of each product represents the number of customers who select its corresponding time slot. Consequently, our aim is to determine an assortment of time slots that maximizes the expected maximum load across all available time slots. \blue{Interestingly, \cite{amorim2024customer} have recently investigated the effect of time slot management in the context of AHD. In particular, their MNL-based study concludes that retailers with the ability to tailor their time slots offering to specific customer segments enjoy a $9\%$ increase in shipping revenue. These conclusions further emphasize the importance of introducing and studying frameworks which align with time slot management, and highlight its significant managerial implications.}

\paragraph{Preference-based Group Scheduling.} 

When scheduling a group meeting, the overarching goal is to identify the most suitable time slot from a given set of options, i.e., one that accommodates the maximum number of attendees. Platforms such as Doodle and When2meet often rely on users selecting a preferred time slot from the available choices. However, the individual choice made by each user is influenced by the available options, due to substitution effects. Therefore, to maximize the likelihood of users selecting the same time slot, decision makers can  carefully curate the assortment of offered time slots. This approach is known as preference-based group scheduling \citep{brzozowski2006grouptime, berry2007preference}. That said, to our knowledge, previous studies have not approached such questions from the perspective of assortment optimization, nor have they utilized discrete choice models to deal with customer preferences. Our framework effectively captures this scenario, by viewing each time slot as a product, again implying that the load of each product represents the number of users who select that time slot.  Thus, our objective would be to determine an assortment of time slots that maximizes the expected maximum load among all  available options.

\subsection{Fundamental challenges}\label{subsec:challenges}

As readers would quickly find out by examining our model formulations, whether one considers static or adaptive settings, coming up with efficient algorithmic approaches that can be rigorously analyzed appears to be a very challenging goal. To better understand where some hurdles are emerging from, we should bear in mind the conceptual trade-off between offering an extensive set of products versus a more focused set, due to two competing effects. On one hand, providing a wide array of products grants customers more choices, reducing their likelihood to leave the market without making a selection, and potentially increasing the maximum load. On the other hand, offering too many products may disperse customer demand across all available choices. As our objective is to guide customers towards selecting the same product, this dispersion can potentially diminish the maximum load.

Let us proceed by briefly highlighting some fundamental challenges in addressing both problem formulations. In the static setting, the first and foremost challenge revolves around the highly non-linear nature of the objective function. Unlike revenue or sales maximization, we are considering a novel objective function, appearing to be  very different from classical settings in the assortment optimization literature. Among other missing pieces, we are not aware of any integer programming formulations or linear relaxations for the problem in question. In fact, even computing the expected maximum load for a given static assortment is very much unclear at first glance. In the dynamic setting, the state space of every conceivable dynamic program describing this problem is exponential in size. Therefore, by directly solving natural dynamic programs, we would not end up with efficient algorithmic approaches. Moreover, as we explain in the sequel, the Bellman equations associated with such dynamic programs include optimizing over the seemingly-unstructured collection of all relevant assortments, which  generally poses a complex challenge by itself. On top of these obstacles, we will discuss additional challenges in subsequent sections, as soon as they can be better digested.


\subsection{Main contributions}

The primary contribution of this paper resides in developing a unified optimization framework with provably near-optimal performance guarantees for both formulations of the maximum load assortment optimization problem. 
In the static setting, we first present a polynomial-time evaluation oracle to compute the expected maximum load of a given assortment. Then, by uncovering well-hidden structural properties of the objective function, we provide an elegant constant-factor approximation for \ref{SMLA}. As our main result for the static setting, we present a polynomial-time approximation scheme (PTAS). 
For the dynamic formulation, by developing novel coupling arguments in this context, we first establish a constant-factor bound on its adaptivity gap. Moreover, we devise a $(1-\eps)$-approximate adaptive policy that can be computed in quasi-polynomial time. We proceed by providing refined details on our main contributions.

\paragraph{{\bf Static maximum load assortment optimization.}}
\begin{itemize}
    \item {\em Polynomial-time evaluation oracle for the expected maximum load.}  The first challenge in the static formulation consists of the seemingly-simple question of evaluating the objective function of \ref{SMLA} for a given assortment, i.e., computing its expected maximum load. In fact, even though the latter admits a closed-form expression, it requires summing over exponentially-many terms that arise from the Multinomial distribution. Our first contribution is to design a polynomial-time evaluation oracle for computing the expected maximum load of a given static assortment. Our algorithm, whose specifics are given in Section~\ref{compute}, builds on the work of \cite{frey2009algorithm} who designed polynomial-time procedures to evaluate {\em rectangular probabilities} for the Multinomial distribution. In essence, we show  that the expected maximum load function can be computed through polynomially-many external calls to evaluate rectangular probabilities. 

    \item {\em $1/2$-approximation via preference-weight-ordered assortments.} Prior to presenting our main result regarding \ref{SMLA}, we propose in Section \ref{subsec:halfapprox} an elegant and easy-to-implement way to obtain a $1/2$-approximation, utilizing {\em preference-weight-ordered assortments}. In a nutshell, such assortments  prioritize products with higher preference weights. Specifically, when a product is included in a preference-weight-ordered assortment, all products with higher preference weights are included as well. Interestingly, we prove that there exists a preference-weight-ordered assortment whose expected maximum load is within a factor $1/2$ of the optimum. Our policy then examines all such assortments, of which there are only linearly-many, picking the best via our previously-mentioned evaluation oracle for their expected maximum load. As a side note, avid readers may be familiar with the notion of ``revenue-ordered'' assortments, which has been explored and exploited in early literature. Most notably, in revenue maximization under the MNL model, optimal assortments are known to be revenue-ordered \citep{talluri2004revenue}. That said, beyond the natural resemblance through a certain parametric order, the analysis of preference-weight-ordered assortments  turns out to be entirely different and requires new analytical ideas.

    \item {\em Polynomial-time approximation scheme.} Our main technical contribution with respect to the static formulation resides in developing a polynomial-time approximation scheme (PTAS) for this setting, whose specifics are provided in Section~\ref{subsec:PTAS}. Namely, for any fixed $\eps>0$, our algorithm constructs in polynomial time an assortment whose expected maximum load is within factor $1-\eps$ of the optimum. To derive this result, we prove the existence of a polynomially-sized family of highly-structured assortments via  efficient enumeration ideas. We refer to these assortments as being block-based, showing that at least one such assortment yields a $(1-\eps)$-approximation. Finally, using our polynomial-time evaluation oracle, we enumerate over all block-based assortments, and pick the best one. We should note that, despite our best efforts in studying the computational complexity of \ref{SMLA}, we still do not know whether this problem is NP-hard or not. This difficulty mainly arises due to the nature of our objective function, as we are unaware of NP-hard problems with similar structure that would serve as candidates for potential reductions. Hence, attaining complexity lower bounds that will match our algorithmic guarantees remains an intriguing open question, further discussed in Section~\ref{sec:conclusion}.

    \blue{
    \item {\em Numerical analysis.} Complementing the aforementioned theoretical contributions, we conduct a series of numerical experiments to examine how optimal assortments behave with respect to the model primitives. Our analysis exhibits a notable tendency for the optimal static assortment to decrease in size as the number of arriving customers increases, as well as when the preference weights increase. \red{Specifically, our experiments show that offering the whole universe of products may become optimal for instances with small preference weights. This observation is particularly significant in instances with a smaller number of customers.} These results are reported in Section~\ref{sec:numerics}.}

\end{itemize}

\paragraph{{\bf Dynamic maximum load assortment optimization.}}
\begin{itemize}
    \item {\em Adaptivity gap.} Our first line of investigation examines questions related to the adaptivity gap of the maximum load assortment optimization problem.  In this context, the adaptivity gap is defined as the maximal ratio between the objective values of \ref{DMLA} and \ref{SMLA} over all possible instances.  This measure  quantifies the value of introducing adaptivity, quantifying the improvement gained by employing a dynamic policy instead of a static one. In Section \ref{sec:adaptivitygap}, we prove the existence of a static policy, utilizing preference-weight-ordered assortments, whose expected maximum load is within factor $1/4$ of the adaptive optimum, implying that the adaptivity gap is surprisingly bounded by $4$. This result immediately translates to a polynomial-time $1/4$-approximation for \ref{DMLA}. Moreover, when all products have the same preference weight, we improve the adaptivity gap to $2$. In the opposite direction, we present a family of instances demonstrating that the adaptivity gap of \ref{DMLA} is at least $4/3$. \blue{Additionally, in Appendix~\ref{apx:justif}, we present numerical experiments studying how the  adaptivity gap behaves under different parameteric regimes. This numerical section allows readers to gain a more concrete understanding of the inherent gap between the optimal static and dynamic objectives. Concurrently, these experiments motivate us to construct the family of instances yielding the aforementioned $4/3$ lower bound on the adaptivity gap.}

    \item {\em $(1-\eps)$-approximate adaptive policy in quasi-polynomial time.} Our cornerstone technical contribution in relation to \ref{DMLA} resides is devising a quasi-polynomial time adaptive policy whose expected maximum load is within factor $1-\eps$ of the optimum. This policy, whose finer details are discussed in Section~\ref{sec:dynamic}, builds upon two key ideas. Firstly, rather than attempting to solve an exponentially-sized natural dynamic program, we demonstrate that its state space can be shrunk to a quasi-polynomial scale while only sacrificing an $O(\eps)$-factor in optimality. More specifically, we observe that once a sufficiently large load is attained, the expected marginal gain from offering any further assortments becomes negligible in comparison  to the already-attained maximum load. Consequently, we can effectively terminate the arrival process (i.e., offer the empty assortment from this point on), which significantly reduces the state space size. Secondly, to compute an optimal action for the resulting recursive equations, an assortment-like optimization problem needs to be solved at each stage. We argue that this problem can be reformulated as an unconstrained revenue maximization question under the Multinomial Logit model, which can indeed be solved in polynomial time. \blue{It is worth mentioning that quasi-polynomial time approximation schemes have gained popularity in various domains including assortment optimization, network design, scheduling, computational game theory, and graph algorithms. While presenting an exhaustive overview of such results would be impractical, we refer the reader to selected papers \citep{chekuri2002approximation, arora2003approximation,lipton2003playing, bansal2006quasi,remy2009quasi,chan2011qptas,   adamaszek2014quasi,das2015quasipolynomial, mustafa2015quasi,desir2021mallows,AouadS23}, which highlight the ubiquitous nature of quasi-polynomial time approximation schemes.}
\end{itemize}

\subsection{Related literature}

In what follows, we discuss three lines of research that are directly relevant to our work. Firstly, we discuss the Multinomial Logit model, which stands as one of the most widespread choice models in both theoretical and practical domains.  Secondly, we provide a concise overview of the assortment optimization literature, emphasizing how our model fits within this body of research. Lastly, we discuss several classic  balls and bins problems, highlighting the relevant connections and similarities between these problems and our own setting.

\paragraph{The Multinomial Logit model.} The Multinomial Logit model (MNL) is widely regarded as the predominant choice model employed by the revenue management community to capture customer behavior when selecting from a given assortment. This model was initially introduced by \cite{luce1959individual}, with subsequent works by \cite{mcfadden1973conditional} and \cite{hausman1984specification} further refining its specification. Informally, the MNL models assigns a preference weight to each product. Then, each product is chosen with probability proportional to its preference weight, thereby capturing the substitution effect that occurs between various alternatives within any given assortment. The model's simplicity in calculating choice probabilities, its predictive power, and computational tractability have all contributed to its widespread adoption and extensive study in various domains. Some of these directions are evidenced by research works such as those of \cite{mahajan2001stocking}, \cite{talluri2004revenue}, \cite{rusmevichientong2014assortment}, \cite{sumida2021revenue}, \cite{gao2021assortment}, \cite{bai2022coordinated}, and \cite{el2021joint} to mention a few. For a comprehensive understanding and further references, we refer the reader to Chapter~4 of the book by \cite{gallego2019revenue}.


\paragraph{Assortment optimization.} Assortment optimization represents a long-standing research domain within revenue management, seeking to address fundamental questions regarding the selection of offer sets for customers under various choice models. Here, the typical goal is to optimize performance metrics such as revenue, market share, and engagement. Over the past decades, this field has witnessed substantial growth, resulting in an extensive literature encompassing different algorithmic developments under various choice models, such as Multinomial Logit \citep{talluri2004revenue,rusmevichientong2014assortment,aouad2021assortment}, Markov Chain \citep{blanchet2016markov,feldman2017revenue}, Nested Logit \citep{davis2014assortment,gallego2014constrained}, 
and non-parametric choice models \citep{farias2013nonparametric, aouad2018approximability}.
For a comprehensive study and further references,
we refer the reader to related surveys and books \citep{kok2009assortment,phillips2021pricing,gallego2019revenue}.
As previously mentioned, it is important to note that our work diverges from the classic assortment optimization literature in terms of the objective function we optimize.  To the best of our knowledge, this paper is the first study to investigate the maximum load objective function from an assortment optimization perspective.



    
\paragraph{Balls and bins.} In its most general setting, the literature on balls and bins explores the outcomes of randomly placing $m$ balls into $n$ bins. This topic finds numerous applications, with load balancing and hashing being arguably the most commonly known ones \citep{mitzenmacher2017probability, mirrokni2018consistent}. 
Relating such questions to our setting, each customer can be viewed as a ball, whereas each product can be represented by a bin. The probability of a particular ball falling into a specific bin corresponds to the likelihood of a particular customer selecting a particular product. In their seminal work, \cite{raab1998balls} provide a comprehensive analysis of the maximum number of balls in any bin, offering precise upper and lower bounds that hold asymptotically. Specifically, for $n$ balls and $n$ bins with equal probabilities, the expected maximum load is $(1+o(1)) \cdot \frac{\log n}{\log\log n}$ with high probability. In a different direction, 
\cite{azar1994balanced} proved a significant drop in the maximum load to $\frac{\log \log n}{\log 2}+O(1)$ with high probability, when each ball is placed in the least loaded out of two randomly chosen bins. It is important to point out that the literature on balls and bins primarily focuses on load balancing rather than on load maximizing applications, where one actually wishes to over-pack bins by actively selecting which bins to make use of, given the constant presence of an outside option. That being said, we find that some well-known results still offer preliminary insights. However, to our knowledge, none of these results are directly relevant to statically or adaptively making assortment decisions in order to optimize the maximum load. In other words, our paper is the first to study balls-and-bins-like problems within the context of choice modeling and assortment optimization.



   
\section{Problem Formulation}\label{sec:form}

\paragraph{The MNL choice model.} In what follows, we begin by explaining how the Multinomial Logit choice model is formally defined. To this end, let $\Nc = \{1, 2, \ldots, n\}$ be the universe of products at our disposal, where each product $i\in \Nc$ is associated with a {\em preference weight} $v_i>0$. In addition, the option of not selecting any of these products will be symbolically represented as product $0$, referred to as the no-purchase or no-selection option, with a preference weight of $v_0 = 1$. While the precise meaning of these parameters will be explained below, we mention in passing that the preference weight assigned to each product reflects its level of attractiveness, meaning that higher preference weights would indicate a greater level of popularity.

With these conventions, an assortment (or an offer set) is simply a subset of products $S \subseteq \Nc$. For convenience, we make use of $S_+=S \cup \{0\}$ to denote the inclusion of the no-purchase option within this assortment. \blue{We define the weight $v(S)$ of an assortment $S$ simply as the sum of the preference weights of its products, namely, $v(S) = \sum_{i\in S}v_i$}. Now, when any given assortment $S \subseteq \Nc$ is offered to an arriving customer, the MNL model prescribes a probability of $\phi_i(S) = \frac{ v_i }{ v(S_+) } = \frac{v_i}{1+\sum_{j\in S} v_j}$ for picking product $i \in S$ as the one to be purchased. Alternatively, this customer may decide to avoid selecting any of these products (i.e., picking the no-purchase option), which happens with the complementary probability, $\phi_0(S) = \frac{1}{1+\sum_{j\in S}v_j}$.


\paragraph{Stream of customers.} Next, we introduce the static formulation of the maximum load assortment problem, followed by the presentation of its dynamic counterpart. In both formulations, we will be facing a \blue{finite} stream of $T$ customers, arriving one after the other, and we therefore refer to these customers by their arrival indices, $1, \ldots, T$. We assume that the choice of each customer among any offered assortment is governed by the aforementioned Multinomial Logit model, meaning in particular that their purchasing decisions are mutually independent.

\subsection{Static Maximum Load Assortment  Optimization \texorpdfstring{\eqref{SMLA}}{}    }\label{subsec:SMLA}

In the static setting, we will be operating under the restriction that all  customers should be offered the exact same assortment of products throughout the arrival sequence. Specifically, consider an assortment $S\subseteq \Nc$. For any product $i\in S_+$ and for any customer $t \in [T]$, we define a  Bernoulli random variable  $X_{it}(S)$ to indicate whether customer $t$ selects product $i$ or not. Since customer $t$ chooses this product with probability $ \phi_i(S)$, we  have $\P(X_{it}(S)=1) = \phi_i(S)$. As such, $\sum_{i\in S+}X_{it}(S)=1$, reflecting the fact that each customer chooses exactly one product from the assortment $S$ or decides not to select any product at all. In addition, since  customers' decisions are independent, the indicators $\{ X_{it}(S) \}_{t \in [T]}$ of different customers are mutually independent.
 
Given an offered assortment $S$, we define the {\em load} of product $i\in\Nc$ as the total number of customers who select this product. This random quantity will be designated by $L_i(S)$, noting that it can be expressed as $L_i(S) =  \sum_{t=1}^TX_{it}(S)$. We use $L_0(S)=\sum_{t=1}^TX_{0t}(S)$ to denote the no-purchase load upon offering the assortment $S$, i.e., the number of customers who did not select any product. Finally, $M(S)$ will stand for the maximum load over all products, i.e., $M(S) = \max_{i\in S}L_i(S)$. Hence, the number of customers who choose the most selected product corresponds to the maximum load across all products. Our optimization problem  consists in computing an assortment that maximizes the expected maximum load. We refer to this problem as {\em Static Maximum Load Assortment}  (\ref{SMLA}), compactly formulated as follows: \begin{equation}\label{SMLA}
    \max_{S\subseteq\Nc} \;\E \left(
    M(S)
    \right). \tag{\SMLA}
    \end{equation}

\paragraph{Closed-form expression for the maximum load.} Given this formulation, we first observe that efficiently computing the expected maximum load $\E(M(S))$ of a given assortment $S$ is a non-trivial task. Prior to developing an efficient algorithm for this purpose, let us start by deriving a supposedly straightforward closed-form expression. Consider an assortment $S \subseteq \Nc$ and suppose without loss of generality that $S = \{1,\ldots,k\}$  for some $k\leq n$. For each product $i \in S_+$, its corresponding random load $L_i(S)$ clearly follows a Binomial distribution of mass parameter $T$ and probability of success $\phi_i(S)$. However, these random variables are correlated, since $\sum_{i \in S_+} L_i(S)=T$. In fact, the load vector ${\mathbf L}(S)\coloneqq(L_0(S),\ldots,L_k(S))$ is a random vector that follows a Multinomial distribution. In particular, for every $\bl\coloneqq(\ell_0,\ldots,\ell_{k})\in\N^{|S_+|}$ with $\sum_{i\in S_+}\ell_i= T$, we have $\P({\mathbf L}(S)=\bl)= \binom{ T }{ \ell_0,\ldots,\ell_k } \cdot \prod_{i\in S_+}(\phi_i(S))^{\ell_i}$,
where $\binom{ T }{ \ell_0,\ldots,\ell_k }\coloneqq\frac{T!}{\ell_0!\cdots \ell_k!}$ is the Multinomial coefficient. As such, a direct expression for the expected maximum load is given by:    
\begin{equation}\label{eq:multinomial}
\E \left(
    M(S)
    \right)=\sum_{ \MyAbove{{\bl}\in\N^{k+1} :}{\sum_{i\in S_+}\ell_i=T}}\P\left({\mathbf L}(S)={\bl}\right)\cdot \max_{i\in S}\ell_i.
\end{equation}
However, in this representation, we sum over an exponential number of terms, $\binom{ T+k}{ k}$, which makes this computation intractable. In Section \ref{compute}, we provide a polynomial-time algorithm to compute the expected maximum load for any given assortment. 

\subsection{{Dynamic Maximum Load Assortment Optimization (\ref{DMLA})}}\label{subsec:DMLA}

In the dynamic setting, customers arrive one after the other, allowing the decision maker to tailor the assortment offered to each customer based on the choices observed for previously-arriving customers. In particular, at each time period, we have access to the current load vector, which provides  the number of customers who have selected each product up to that point. Based on this information, we wish to determine a personalized assortment that will be offered to the next arriving customer. As such, the solution concept in this setting corresponds to an adaptive policy, captured by 
a function that takes as input the current system state (i.e., the number of customers remaining and the current load vector), and returns an assortment to offer to the next customer. The objective is to propose an adaptive policy that maximizes the expected maximum load over all products upon termination of the arrival stream.  


    

    

\paragraph{Dynamic programming representation.} To formalize the dynamic setting, we take the view of a dynamic program that determines the actions taken by an optimal policy, i.e.,  the personalized assortments that will be offered to arriving customers. For this purpose, we consider a planning horizon consisting of $T$ periods, each with a single customer arrival.
To describe the system state at the beginning of any time period, we introduce the state variable $\bl = (\ell_i : i \in \mathcal{N})$, where $\ell_i$ represents the number of customers who have selected product $i$ up to that point. For each time period $t = 1, \ldots, T$, we use $\emm_t(\bl)$ to denote the optimal expected maximum load when there are $t$ customers remaining in the planning horizon, and the system's state at the beginning of this period is characterized by the load vector $\bl$. By employing ${\bf e}_i \in \mathbb{R}_+^n$ to represent the $i$-th unit vector, we can compute the value functions $\{ \emm_t \}_{ t \in [T]}$ via the following dynamic program:    \begin{equation}\begin{array}{ll} \label{DMLA}\tag{\DMLA}
    		 \emm_{t}(\mathbf{\bl})=\displaystyle\max_{S\subseteq \Nc}\left(\emm_{t-1}( \bl)\cdot\phi_0(S)+\sum_{i\in S}\emm_{t-1}(\mathbf{\bl}+\mathbf{e}_i)\cdot \phi_i(S)\right) 
    \end{array}\end{equation}
with the boundary condition $\emm_0(\bl)=\max_{i \in [n]}\ell_i$. To better understand the recursive equation above, note that when the current load vector is $\bl$ and we offer the assortment $S$, the first possible outcome is that the currently arriving customer will choose the no-selection option, with probability $ \phi_0(S)$, in which case the load vector remains unchanged. The second outcome corresponds to choosing one of the products $i \in S$, with probability $ \phi_i(S)$; here, the load vector $\bl$ is updated to $\bl +\mathbf{e}_i$. Clearly, the optimal expected maximum load in the entire horizon is given by  $\emm_{T}(\bf 0)$.

It is imperative to mention that this formulation should be viewed as being an explicit characterization of optimal adaptive policies rather than as an efficiently-implementable algorithm, due to being defined over a state space of exponential size, $\Omega( T^n )$. Moreover, a careful inspection of \ref{DMLA} shows that each of the value functions $\emm_t(\cdot)$ is recursively obtained by only considering $\emm_{t-1}(\cdot)$-related terms. However, these equations by themselves ask us to solve an assortment-like optimization problem. Quite surprisingly, in \red{Appendix \ref{apx:proofquasitime}}, we show that this inner problem can be reformulated as an unconstrained revenue maximization question under the Multinomial Logit model, which can be solved in polynomial time.

\section{The Static Setting: Approximation Algorithms}\label{sec:SMLA}

In this section, we present our main algorithmic results for \ref{SMLA}, eventually showing that this setting can be efficiently approximated within any degree of accuracy. Toward this objective, in Section  \ref{compute}, 
we first provide a polynomial time evaluation oracle for computing the expected maximum load of a given assortment. In Section \ref{subsec:PTASlemmas}, we establish a number of structural lemmas that will be useful in analyzing our algorithmic framework. Using these claims, we show in Section~\ref{subsec:halfapprox}  that an elegant policy based on preference-weight-ordered assortments yields a $1/2$-approximation for \ref{SMLA}. In Section \ref{subsec:PTAS}, we present our main contribution for the static formulation, showing that it admits a polynomial-time approximation scheme (PTAS).
\blue{Finally, in Section \ref{subsec:regime}, we study the special case where the number of customers $T$ is very large and characterize optimal assortments  in this regime.}
\subsection{Polynomial-time evaluation oracle}\label{compute}

As previously mentioned, one of the basic challenges in addressing \ref{SMLA} resides in simply evaluating the objective function of a given assortment. As discussed in Section~\ref{subsec:SMLA}, 
computing the expected maximum load via representation \eqref{eq:multinomial} requires summing over exponentially-many terms, which is clearly not a tractable approach. Our first contribution is to provide a polynomial time algorithm for computing the  expected maximum load function.

\begin{theorem}\label{thm:oracle}
The expected maximum load of any assortment can be computed in $O(n^2T^3)$ time.
\end{theorem}

In a nutshell, our algorithm builds upon the work of  \cite{frey2009algorithm} and \cite{lebrun2013efficient}, who designed polynomial-time procedures to evaluate {\em rectangular probabilities} for the Multinomial distribution. Specifically, let $\mathbf L$ be a random vector that follows a Multinomial distribution. A rectangular probability is the probability of a so-called  {\em rectangular event}, of the form $\{ \mathbf{a} \leq \mathbf{L} \leq \mathbf{b} \}$, where $\mathbf a$ and $\mathbf b$ are integer vectors.

\blue{
\noindent {\bf Overview of Frey's algorithm. }
\cite{frey2009algorithm} proposed a sophisticated approach to compute rectangular probabilities in polynomial time. Let us present a brief overview of his ideas. Recall that $\mathbf L$ is a random $k$-dimensional vector that follows a Multinomial distribution, where $T$ and $(p_1,\ldots,p_k)$ are the number of trials and the probability vector, respectively. To compute $\P(\mathbf{a\leq L\leq b})$, we start by summing over all possible realizations $\mathbf a\leq \bl \leq \mathbf b$ of the load vector $\mathbf L$:
$$
    \P\left(\mathbf{a\leq L\leq b}\right) = \sum_{\ell_1=a_1}^{b_1}\ldots\sum_{\ell_k = a_k}^{b_k}\mathbbm{1}\left(\left\|\mathbf{\ell}\right\|_1=T\right)\cdot \frac{T!}{\ell_1!\cdots \ell_k!}\prod_{i=1}^kp_i^{\ell_i} ,
$$
where $\|\mathbf{\ell}\|_1 = \sum_{i=1}^k\ell_i$. Then, noticing that $\ell_i\geq a_i$ for all $i=1,\ldots,k$ in the summations above, we can perform the following factorization:
\begin{align*}
    \P\left(\mathbf{a\leq L\leq b}\right) &= \frac{T!\cdot p_1^{a_1}\cdots p_k^{a_k}}{a_1!\cdots a_k!}\sum_{\ell_1=a_1}^{b_1}\ldots\sum_{\ell_k = a_k}^{b_k} \mathbbm{1}\left(\left\|\mathbf{\ell}\right\|_1=T\right)\cdot\frac{1}{\frac{\ell_1!}{a_1!}\ldots \frac{\ell_k!}{a_k!}}\prod_{i=1}^kp_i^{\ell_i-a_i}\\
    & = \frac{T!\cdot p_1^{a_1}\cdots p_k^{a_k}}{a_1!\cdots a_k!}\sum_{\ell_1=a_1}^{b_1}\ldots\sum_{\ell_k = a_k}^{b_k}\mathbbm{1}\left(\left\|\mathbf{\ell}\right\|_1=T\right)\cdot
    \prod_{i=1}^k \prod_{j=1}^{\ell_i-a_i}\frac{p_i}{a_i+j}
\end{align*}
To proceed from this point on, notice that each summand in the last expression is the product of $T - \sum_{i=1}^ka_i$ terms. To compute each summand, a naive algorithm would simply start with $p_1/(a_1+1)$, then proceed to multiply it by $p_1/(a_1+2)$ or $p_2/(a_2+1)$, depending on whether $\ell_1-a_1>1$ or not, and so on for each of the summands. However, Frey's algorithm proposes a way to gather all summands that are multiplied by the same factor $p_i/(a_i+j)$ at the same time step, and multiply their sum by  $p_i/(a_i+j)$, which allows us to only perform a single multiplication for those summands at that time step, rather than a separate multiplication for each. Finally, using recursion, this algorithm computes the final values at the last step, and returns their sum. In what follows, we show how Frey's algorithm is employed to design our polynomial time evaluation oracle.
}

\paragraph{Preliminaries.} Consider an arbitrarily-structured assortment $S\subseteq\Nc$ and  suppose without loss of generality that $S=\{1,\ldots,k\}$. We remind the reader that the random variable $L_i(S)$ stands for the load of product $i\in S$, with $\mathbf{L}(S)=(L_1(S),\ldots,L_k(S))$ being the overall load vector. Additionally, $M(S)$ is the random variable that refers to the maximum load across the products in $S$, i.e., $M(S) = \max_{i \in S} L_i(S)$.  As argued in Section \ref{subsec:SMLA}, 
the load vector $\mathbf{L}(S)$ follows a Multinomial distribution. In what follows, we explain how to compute $\E(M(S))$ using only a polynomial number of externals calls to evaluate rectangular probabilities. 
To this end, noting that
\begin{equation}\label{eq:mean} \E\left(M(S)\right) = \sum\limits_{\ell=1}^{T}\ell\cdot \P(M(S) = \ell),
\end{equation}
it suffices to show how to efficiently compute each of the terms $\P(M(S) = \ell)$. In turn, we write each event $\{M(S)=\ell\}$ as a partition into $O( n )$ rectangular events
 with respect to the random vector ${\mathbf L}(S)$. Specifically, for $1\leq \ell\leq T$ and $1\leq j\leq k$, we define the event $F_{\ell j}(S)$ as:        \begin{align}\label{eq:rect}
        F_{\ell j}(S) =& \left[\bigwedge_{i=1}^{j-1}\left\{L_i(S)<\ell\right\}\right]\;\boldsymbol\land\; \left\{L_j(S)=\ell\right\} \;\boldsymbol\land\;\left[\bigwedge_{i=j+1}^k\left\{L_i(S)\leq \ell\right\}\right]  = \left\{\mathbf{a}_{\ell j}\leq\mathbf L(S)\leq \mathbf{b}_{\ell j}\right\},
        \end{align}
where ${\mathbf{a}_{\ell j}=\ell\cdot \mathbf{e}_j}$ and ${\mathbf{b}_{\ell j}=(\ell-1)\cdot\sum_{i=1}^{j-1}\mathbf{e}_i+\ell\cdot\sum_{i=j}^{k}\mathbf{e}_i}$. Here,  $F_{\ell j}(S)$ corresponds to the event where the maximum load is equal to $\ell$, and product $j$ is the minimal-index product that attains this load. The above expression implies that $F_{\ell j}(S)$ is a rectangular event, meaning that its probability can be computed using Frey's algorithm.

\paragraph{Computing ${\E(M(S))}$.}
The next lemma shows how to utilize these rectangular events to compute $\E(M(S))$ for any assortment $S\subseteq\Nc$.

\begin{lemma}\label{lem:partition}
For any assortment $S=\{1,\ldots,k\}\subseteq \Nc$, 
we have $\E(M(S)) = \sum_{\ell=1}^T[\ell\cdot\sum_{j=1}^k \P(F_{\ell j}(S))]$.
\end{lemma}
\begin{proof}
For convenience, we denote the random variables $M(S)$ and $L_j(S)$ simply by $M$ and $L_j$; similarly, the events $F_{\ell j}(S)$ will be replaced by $F_{\ell j}$. Fixing some $\ell=0,\ldots,T$, we will show that $(F_{\ell j})_{j=1,\ldots ,k}$ is a partition of the event $\{M=\ell\}$, 
i.e., the union of the events $\left(F_{\ell j}\right)_{j=1,\ldots ,k}$ is precisely $\{M=\ell\}$ and these events are mutually exclusive. Consequently, ${\P\left(M=\ell\right)=\sum_{j=1}^k\P(F_{\ell j})}$, and replacing this expression in Equation \eqref{eq:mean} yields the desired result.

First, we show the events $(F_{\ell j})_{j=1,\ldots ,k}$ are mutually exclusive, i.e., for all $j_1 \neq j_2$, we have $F_{\ell j_1} \cap F_{\ell j_2}= \emptyset$. To verify this claim, suppose without loss of generality that $j_1<j_2$. In the event $F_{\ell j_2}$, we have by definition $L_{j_1}<\ell$ since $j_1<j_2$. In particular, $L_{j_1} \neq \ell$ which implies that $\mathbf{L}\notin F_{\ell j_1}$. Hence, $F_{\ell j_1} \cap F_{\ell j_2}= \emptyset$. Second, let us show that $\bigvee_{j=1}^k F_{\ell j} = \{M=\ell\}$. First, by definition, the maximum load in any event $F_{\ell j}$ is exactly $\ell$, and therefore $\bigvee_{j=1}^k F_{\ell j} \subseteq \{M=\ell\}$. In the opposite direction, suppose that $M=\ell$. Then, at least one product has a load of $\ell$ and all other products have a load of at most $\ell$. Let $j_{\min}$ be the lowest-index product with a load of exactly $\ell$, i.e., $j_{\min} = \min\{i=1,\ldots,k \mid L_i=\ell\}$. By definition of $j_{\min}$, we have $L_i<\ell$ for all $i=1,\ldots,j_{\min}-1$ and $L_{j_{\min}} = \ell$. Also, since $M=\ell$ by supposition, $L_i\leq \ell$ for all $i=j_{\min}+1,\ldots,k$, meaning that $\{M=\ell\}\subseteq F_{\ell j_{\min}}\subseteq \bigvee_{j=1}^k F_{\ell j} $.
\end{proof}

            

\paragraph{Concluding the proof of Theorem~\ref{thm:oracle}.} We consider the running time incurred by computing $\E(M(S))$ via Lemma \ref{lem:partition}. For each of the events $\{M(S)=\ell\}$, we have to compute $O(n)$ rectangular probabilities. Since Frey's algorithm admits an $O(nT^2)$-time implementation, we arrive at $O(n^2T^2)$ operations per event. In turn, Equation \eqref{eq:mean} involves $O(T)$ such events, amounting to an overall running time of $O(n^2T^3)$, precisely as stated in Theorem~\ref{thm:oracle}.

\subsection{Structural lemmas}\label{subsec:PTASlemmas}

In what follows, we shed light on a number of structural claims that will be useful in presenting our algorithmic framework. In Lemmas \ref{lem:merge} and \ref{lem:charging}, we introduce two operations, referred to as {\em Merge} and {\em Transfer}, showing that their application to any assortment does not decrease the expected maximum load. In Lemmas \ref{lem:probs} and \ref{lem:weights}, we show that minor alterations of the instance parameters (choice probabilities or preference weights) yield a correspondingly small deviation from the expected maximum load.       
Before stating these lemmas, let us introduce the following  definition. We say that a product $i$ is {\em lighter} (resp.\ {\em heavier}) than a product $j$, if $v_i\leq v_j$ (resp.\ $v_i\geq v_j$), emphasizing that we have weak inequalities in both definitions.


\paragraph{Operation 1: Merge.} Consider an assortment $S$ and let $i$ and $j$ be two products in this assortment, with respective preference weights $v_i$ and $v_j$. The operation of merging products $i$ and $j$ consists of replacing both products with a single new product, whose preference weight is $v_i+v_j$. 
In the next lemma, whose proof is presented in Appendix \ref{apx:merge}, we show that the merge operation cannot decrease the expected maximum load.

\begin{lemma}\label{lem:merge}
Consider an assortment $S\subseteq \Nc$ and  let $\widetilde{S}$ be the assortment resulting from merging any two products of $S$. Then, $\E(M(\widetilde{S})) \geq	\E(M(S))$.
\end{lemma}

Roughly speaking, the main idea behind proving this result argues that, when we merge products $i$ and $j$, simple coupling arguments show that the maximum load of $\widetilde{S}$ is stochastically larger than that of $S$. In fact, the load of the merged product is \blue{equal in distribution} to the sum of the loads of products $i$ and $j$. Moreover, since the total sum of MNL preferences weights remains unchanged after a merge, the choice probabilities of all remaining products in the assortment $S$ are also unchanged, and consequently, their load is similar to its pre-merge counterpart. Hence, merging can only increase the maximum load across all products. 

\paragraph{Operation 2: Transfer.} Consider an assortment $S$ and let $i$ and $j$ be two  products in this assortment with respective preference weights $v_i\geq v_j$. For any $\delta \in [0, v_j]$, the operation of $\delta$-weight transfer from product $j$ to product $i$ consists of: (1)~Replacing product $i$ with a new product of preference weight $v_i+\delta$; and (2)~replacing product $j$ with a new product of preference weight $v_j-\delta$. Notably, we always transfer weight from a lighter product to a heavier product. In the next lemma, whose proof is provided in Appendix \ref{apx:charging}, we show that the transfer operation cannot decrease the expected maximum load.

\begin{lemma}\label{lem:charging}
Consider an assortment $S\subseteq \Nc$ and let $\widetilde{S}_{\delta}$ be the assortment resulting from a $\delta$-weight transfer. Then,
$\E(M(\widetilde{S}_{\delta})) \geq	\E(M(S))$.
\end{lemma}

In the proof of this result, we analytically study the function $\delta \mapsto \E(M(\widetilde S_{\delta}))$, showing that it  is non-decreasing. In fact, when $\delta$ increases, the choice probability of product $i$ increases, the choice probability of product $j$ decreases, and all other choice probabilities  remain unchanged. Therefore, at least intuitively, by increasing $\delta$, we virtually transfer part of the load of product $j$ to product $i$. Since product $i$ is heavier than product $j$, it is more likely that product $i$ has a higher  load, and therefore, transferring part of the load of product $j$ to product $i$ can only help increase the expected maximum load. 

\paragraph{Sensitivity of the expected maximum load function.} In the following, we study the effect of small changes in the instance parameters (choice probabilities or preference weights), on the expected maximum load of a given assortment. The next lemma shows that, given a Multinomial vector, where $0$ is the no-selection option and $1,\ldots,m$ are the product options, slight changes in the choice probabilities translate to small changes in the expected maximum load. 

\begin{lemma}\label{lem:probs}
Let ${\bf Y}=(Y_0,Y_1,\ldots,Y_m)$ and ${\bf W}=(W_0,W_1,\ldots,W_m)$ be Multinomial vectors, with parameters $(T,p^Y_0,\ldots,p^Y_m)$ and  $(T,p^W_0,\ldots,p^W_m)$, respectively. Then, when $p^W_i\geq (1-\epsilon)p^Y_i$ for all $i \in \{1,\ldots,m\}$, we have $\E(\max_{i=1,\ldots,m}W_i) \geq (1-\epsilon) \cdot \E(\max_{i=1,\ldots,m}Y_i)$.
\end{lemma}

To prove this result, we construct a coupling between the random variables $\mathbf{Y}$ and $\mathbf{W}$, where every customer that selects some option $i$ in $W_i$ also selects the same option in $Y_i$ with probability at least $1-\eps$. Consequently, when $i$ is a deterministic option, it is relatively straightforward to claim that the load of this option only suffers an $\eps$-fraction loss (in expectation). However, in our case, the option $i$ itself is a random variable, corresponding to the product that attains the maximum load. Based on an elegant conditioning argument, we show that a claim of similar spirit can be extended to the latter setting. The full proof of this lemma appears in Appendix \ref{apx:probs}.

Similarly, in Lemma \ref{lem:weights}, we show that with respect to any assortment, small changes in the preference weights of its products translate to small changes in the expected maximum load. The proof of this result can be found in Appendix \ref{apx:weights}.
	
\begin{lemma}\label{lem:weights}
Let $S^+=\{1^+,\ldots,m^+\}$ and $S^-=\{1^-,\ldots,m^-\}$ be a pair of assortments, and let $v_i^+$ (resp.\ $v_i^-$) be the preference weight of product $i^+$ (resp.\ $i^-$), for all $i\in [m]$. When $(1-\epsilon)v_{i}^+\leq v_{i}^-\leq v_{i}^+$ for all $i \in [m]$, we have $\E(M(S^-)) \geq (1-\epsilon)\cdot\E(M(S^+))$.
	\end{lemma}

\subsection{\texorpdfstring{$\boldsymbol{1/2}$}{}-approximation via preference-weight-ordered assortments}\label{subsec:halfapprox}

\paragraph{Main result.} Roughly speaking, preference-weight-ordered assortments prioritize products with higher preference weight. Formally, assuming without loss of generality that $v_1\geq \cdots \geq v_n$, we say that an assortment $S$ is preference-weight-ordered when it forms a prefix of this sequence, i.e., $S=\{1,2,\ldots,j\}$ for some $ 1 \leq j \leq n$. In what follows, we show that there exists a weight-ordered-assortment whose expected maximum load is within factor $2$ of optimal. Since there are only $n$ such assortments, and since we can compute the expected maximum load for each of these assortments by employing our evaluation oracle (see Section~\ref{compute}), the latter claim yields a polynomial time $1/2$-approximation for \ref{SMLA}, as stated in the following theorem.

\begin{theorem}\label{thm:halfapprox}
There is a preference-weight-ordered assortment that forms a $1/2$-approximation to \ref{SMLA}. Moreover, we can compute such an assortment in polynomial time.
\end{theorem}

In order to establish this result, the overall idea is to consider an optimal assortment $S^*$ for \ref{SMLA}, that may not necessarily be preference-weight-ordered. We will then sequentially modify $S^*$  to obtain a weight-ordered assortment. We prove that, due to these modifications, the loss incurred does not exceed $1/2$ of the objective function. In other words, letting $S$ be the resulting preference-weight-ordered assortment, we would claim that ${\E(M(S))\geq \E(M(S^*))}/2$. 


\paragraph{Outline of analysis.} To prove Theorem \ref{thm:halfapprox}, using the Merge and Transfer operations presented in Section \ref{subsec:PTASlemmas}, we first show that any sufficiently heavy assortment can be replaced by a preference-weight-ordered assortment, plus a so-called virtual product (i.e., not present in the universe $\Nc$), without decreasing the expected maximum load. Moreover, we argue that the preference weight of the latter product is upper-bounded by every  preference weight in the weight-ordered assortment.

\begin{lemma}\label{lem:virtual}
Let $S\subseteq \Nc$ be an assortment with $v(S)\geq v_1$. Then, there exists a weight-ordered assortment $\widetilde S\subseteq \Nc$ and a virtual product $k$ whose preference weight is at most $\min_{i\in \widetilde S}v_i$ such that
$\E(M(\widetilde S\cup\{k\})) \geq \E(M(S))$.
\end{lemma}

To arrive at Lemma \ref{lem:virtual}, we begin with an assortment $S$ and execute a sequence of Merge and Transfer operations to generate the assortment $\widetilde S$ along with the virtual product $k$, both with the desired structure.
The  proof of this claim can be found in Appendix~\ref{apx:virtual}.

Recalling that product $k$ is not part of the universe $\Nc$, the next lemma shows that this virtual product can be removed, while losing a factor of at most $1/(|S|+1)$ in the objective function. 
    
\begin{lemma}\label{lem:drop}
Let $S\subseteq  \Nc$ be any non-empty assortment, and let $k\notin S$ be a product with 
$v_k \leq  \min_{i\in S}v_i$. Then, $\E(M(S)) \geq  \frac{|S|}{|S|+1} \cdot \E(M(S\cup \{k\}))$.
\end{lemma}

In particular, when $S \neq \emptyset$, the lemma above shows that, by removing the virtual product $k$ from $S\cup \{k\}$, we lose a factor of at most $1/2$ in the objective function. The proof of this result is based on our structural results regarding on the sensitivity of the expected maximum load function, along the lines of Lemma~\ref{lem:weights}. Its complete details are given in Appendix~\ref{apx:drop}. 

\paragraph{Concluding the proof of Theorem~\ref{thm:halfapprox}.} Let $S^*$ be an optimal assortment for \ref{SMLA}. First, we observe that $v(S^*)\geq v_1$. Indeed, suppose by contradiction that $v(S^*)< v_1$. In this case, on the one hand, when offering the assortment  $S^*$, the total number of customers who select an option in $S^*$ (i.e., do not select the no-purchase option) is a Binomial random variable with $T$ trials and success probability $v(S^*)/(1+v(S^*))$. Since the maximum load when offering $S^*$ is trivially upper-bounded by the total number of purchases, it follows that $\E(M(S^*))\leq Tv(S^*)/(1+v(S^*))$. On the other hand, when offering the single-product assortment $\{1\}$, the expected maximum load is given by $\E(M(\{1\})) = Tv_1/(1+v_1)$. Since the function $x\mapsto x/(1+x)$ is increasing over $[0,+\infty)$, and since $v(S^*)<v_1$, we have $\E(M(S^*))<\E(M(\{1\}))$, contradicting the optimality of $S^*$. 

Now, given that $v(S^*)\geq v_1$, the conditions of Lemma~\ref{lem:virtual} are met, and therefore, there exists a weight-ordered assortment $\widetilde S$ and a virtual product $k$ whose preference weight is at most $ \min_{i\in \widetilde S}v_i$, such that $\E(\widetilde S\cup\{k\}) \geq \E(M(S^*))$. By definition, weight-ordered assortments are non-empty, meaning that according to Lemma \ref{lem:drop}, we have $\E(M(\widetilde S)) \geq \frac{ 1 }{ 2 } \cdot \E(M(\widetilde S\cup\{k\}))$. Putting both inequalities together, 
\[ \E(M(\widetilde S))\geq \frac12\cdot \E(M(\widetilde S\cup \{k\}))\geq \frac12\cdot\E(M(S^*)). \]

\subsection{Polynomial-time approximation scheme}\label{subsec:PTAS}

Our main technical contribution for \ref{SMLA} consists in designing a polynomial-time approximation scheme (PTAS). In other words, for any fixed $\epsilon\in (0,1)$, we propose a polynomial-time algorithm for identifying an assortment whose expected maximum load is within factor $1 - \eps$ of optimal. This result is formalized in the next theorem. 
	
\begin{theorem}\label{thm:algo}
For any $\eps \in (0,1)$, \ref{SMLA} can be approximated within a factor of $1-\eps$ of optimal. The running time of our algorithm is $O(T^{O(1)} \cdot n^{O(\frac 1\eps \log \frac 1\eps)})$.
\end{theorem}

\paragraph{Block-based assortments.} In what follows, we introduce a family of highly-structured assortments, which will be referred to as being ``block-based''. As explained below, these assortments are defined in three steps, starting from the block of products with the highest preference weight and gradually moving to blocks with lower weights. Without loss of generality, we assume that $1/\eps$ takes an integer value, and that products are indexed in non-increasing order of preference weights, i.e., $v_1 \geq \cdots \geq v_n$. With these conventions, an assortment $S \subseteq \Nc$ is said to be block-based either if its cardinality is at most $1/\eps$, or when it can be written as $S = S_1 \cup S_2 \cup S_3$, where the latter sets are structured as follows: 
\begin{itemize}
	\item {\em Block 1:} The first set, $S_1$, is an arbitrary collection $1/\epsilon$ products. These products will form the subset of the heaviest products in this assortment.
 
	\item {\em Block 2:} Let \blue{$a$} be the highest-index product in $S_1$. The second subset of products in our assortment, $S_2$, is a contiguous block of products, starting from \blue{$a+1$}. In other words, \blue{$S_2=\{a+1,a+2,\ldots,b\}$}, for some \blue{$b\leq n$}.

    \item {\em Further blocks:} Let \blue{$c = b+1$}, where \blue{$b$} is the highest-index product in $S_2$. We create a multiplicative grid across \blue{$[\epsilon\cdot v_c, v_c]$} as follows. The class $C_1$ consists of all products whose weight falls within \blue{$[(1-\eps) \cdot v_c, v_c]$}. Then, the class $C_2$ consists of all products with weights in $\blue{[(1-\eps)^2\cdot v_c,(1-\eps)\cdot v_c)}$. So on and so forth, until we hit the lower bound \blue{$\eps\cdot v_c$}. Letting $C_1, \ldots, C_L$ be the resulting classes, one can easily verify that the number of classes is $L = O( \frac{ 1 }{ \eps } \log \frac{ 1 }{ \eps } )$. Now, for each class $C_\ell$, we select a number $N_\ell$ of products to be included in the assortment, and then simply include the $N_\ell$ products with largest indices from this class. For example, if a certain class is $\{8,9,10,11\}$, then we include $\emptyset$ when $N_{\ell} = 0$, $\{11\}$ when $N_{\ell} = 1$, up to $\{11,10,9,8\}$ when $N_{\ell} = 4$. We will refer to the union of these sets over all classes as $S_3$.
\end{itemize}

We proceed by explaining how to explicitly construct the entire family of blocked-based assortments in $O(n^{O( \frac1\eps \log( \frac1\eps ) )})$ time. First, there are $O(n^{O(1/\eps)})$ options to construct an assortment whose cardinality is strictly smaller than $1/\eps$. Let us now construct the assortments whose cardinality is at least $1/\eps$. In order to create the first block, $S_1$, it is easy to see that there are $O(n^{ O(1/\eps)})$ options. Since the second block $S_2$ is contiguous, there are at most $O(n)$ options here. For the remaining blocks, we create $L$ classes, and for each of these classes, we simply choose the number of products $N_{\ell}$ to be included. Therefore, there are $O(n^L) = O(n^{ O( \frac{ 1 }{ \eps } \log \frac{ 1 }{ \eps } ) })$ options to construct $S_3$.

\paragraph{The performance of block-based assortments.} Our algorithmic approach consists of enumerating over all block-based assortments. Since the evaluation oracle provided in Section~\ref{compute} can be implemented in $O(n^2T^3)$ time, our overall algorithm indeed has a running time of $O(T^{O(1)} \cdot n^{ O( \frac{ 1 }{ \eps } \log \frac{ 1 }{ \eps } )})$, as stated in Theorem~\ref{thm:algo}. The next result proves that at least one of the assortments we are enumerating over yields a $(1-\eps)$-approximation of \ref{SMLA}.

\begin{theorem} \label{thm:block_based_good}
Letting $S^*$ be an optimal assortment for \ref{SMLA}, there exists a block-based assortment $S$ for which $\E(M(S))\geq (1-\epsilon)\cdot \E(M(S^*))$.
\end{theorem}

\paragraph{Definitions and notation.} In order to establish this theorem, we begin by introducing a number of useful definitions and their surrounding notation:
\begin{itemize}
    \item For any assortment $S$, and for any $j=1,\ldots,|S|$, we define $\ost_{j}(S)$ as the index of the $j$-th heaviest product in $S$. That is, if $S=\{p_1,p_2, \ldots, p_{|S|} \}$ where $v_{p_1}\geq \cdots \geq  v_{p_{|S|}} $, then $\ost_{j}(S)= p_j$.
    
    \item For any assortment $S$ with $|S|> 1/\eps$, we define its {\em $\eps$-hole} as the product with index $h_{\eps}(S)$ where $h_{\eps}(S) = \min\{j \notin S : j \geq \ost_{1/\eps}(S)\}$.
    Put simply, the $\eps$-hole refers to the heaviest product in the universe $\Nc$ that is not part of the assortment $S$, but is lighter than the $1/\eps$ heaviest product of $S$.
    
    \item Finally, we say that an assortment $S$ is {\em $\eps$-restricted} either when $|S|\leq 1/\eps$, or when $|S|> 1/\eps$ and the weight of each product in $S$ is larger than a fraction $\eps$ of the weight of its $\eps$-hole, i.e., ${v_i\geq \eps\cdot v_{h_{\eps}(S)}}$ for every $i \in S$. It is worth noting that, by definition, all block-based assortments are $\eps$-restricted.
\end{itemize}

\paragraph{Analysis.} Given these definitions, the proof of Theorem~\ref{thm:block_based_good} consists of two steps. In Lemma~\ref{lem:thmstep1}, we prove that for any sufficiently large assortment $S$, there exists an $\eps$-restricted assortment $\widehat S$ and a virtual product $k$, such that the expected maximum load of $\widehat S \cup \{ k \}$ is at least as large as that of $S$. The proof of this lemma makes use of the Merge and Transfer operations introduced in Section \ref{subsec:PTASlemmas}, transforming the assortment $S$ into the union of an $\eps$-restricted assortment and a virtual product. The detailed proof is included in Appendix~\ref{apx:thmstep1}.

\begin{lemma}\label{lem:thmstep1}
    Let $S\subseteq \Nc$ be an assortment with $|S| > 1/\eps$. Then, there exists an $\eps$-restricted assortment $\widehat S\subseteq \Nc$ with $|\widehat S| > 1/\eps$, and a virtual product $k$ with $v_k \leq v_{h_{\eps}(\widehat S)}$, such that $\E(M(\widehat S\cup\{k\})) \geq \E(M(S))$.
\end{lemma}

In the following lemma, whose proof is presented in Appendix~\ref{apx:thmstep2}, we show that the assortment $\widehat S\cup\{k\}$ obtained in Lemma~\ref{lem:thmstep1} can be transformed into a block-based assortment, losing at most a factor $\eps$ in its objective value. 

\begin{lemma}\label{lem:thmstep2}
Let $\widehat S\subseteq \Nc$ be an $\eps$-restricted assortment with $|\widehat S|> 1/\eps$, and let $k$ be a virtual product with $v_k \leq v_{h_{\eps}(\widehat S)}$. Then, there exists a block-based assortment $\widetilde S\subseteq \Nc$ such that $    \E(M(\widetilde S)) \geq (1-\eps)\cdot \E(M(\widehat S\cup\{k\}))$.
\end{lemma}

To conclude the proof of Theorem \ref{thm:block_based_good}, let $S^*$ be an optimal assortment for \ref{SMLA}. When $|S^*|\leq 1/\eps$, we know that $S^*$ is a block-based assortment by definition, and it remains to consider the opposite case, where $|S^*|> 1/\eps$. By Lemma~\ref{lem:thmstep1}, there exists an $\eps$-restricted assortment $\widehat S$ with $|\widehat S| > 1/\eps$ and a virtual product $k$ with $v_k \leq v_{h_{\eps}(\widehat S)}$ such that $\E(M(\widehat S\cup \{k\}))\geq \E(M(S^*))$. In turn, since $\widehat S$ and $k$ satisfy the conditions of Lemma~\ref{lem:thmstep2}, there exists a block-based assortment $\widetilde S$ such that $\E(M(\widetilde S)) \geq (1-\eps)\cdot \E(M(\widehat S\cup\{k\})).$ Combining these two inequalities yields $\E(M(\widetilde S)) \geq (1-\eps)\cdot \E(M(S^*))$, as desired.

\blue{
\subsection{The many-customers regime: optimal assortment} \label{subsec:regime}
In this section, we show that when the number of customers $T$ is sufficiently large, offering the single product with the highest preference weight is optimal for \ref{SMLA}. As mentioned in Section \ref{subsec:challenges}, optimal assortments balance a trade-off between offering a large or a small number of products. In particular, offering more products decreases the choice probability of the no-purchase option, thereby capturing more customers. However, this decision comes at the cost of cannibalizing the demand between products due to the underlying substitution effect in the MNL model. That is, when we offer more products, the demand will not be concentrated in a single product as desired in the maximum load assortment problem. For large values of $T$, the effect of cannibalization is accentuated, since the load of each product is more and more concentrated around its mean, by virtue of Chernoff-type bounds. Consequently, lighter products' loads are highly unlikely to surpass those of heavier products, and offering them only contributes to cannibalizing the demand. This intuition suggests that large values of $T$ favor smaller-sized assortments. In the following lemma, whose proof is presented in Appendix \ref{apx:largeT}, we show that there exists a threshold $\Bar{T}$, depending on the problem parameters, above which offering only the product with the highest preference weight is optimal. Recall that products are assumed to be indexed in weakly decreasing order of their preference weights, meaning that product $1$ is the heaviest in $\Nc$.
    
    \begin{lemma}\label{lem:largeT}
        There exists a threshold $\bar T$ such that $\E(M(\{1\})) = \max_{ S \subseteq \Nc } \E(M(S))$ when $T\geq \bar T$.
    \end{lemma}
}

\section{The Dynamic Setting: Constant-Factor Adaptivity Gaps}\label{sec:adaptivitygap}

In this section, we examine how well an optimal static assortment could perform in comparison to an optimal adaptive policy. Specifically, we study the adaptivity gap of maximum load optimization, namely, the worst-possible ratio between the expected maximum load of an optimal adaptive policy and that of an optimal static policy, over all problem instances. \blue{
Specifically, any instance $I$ is characterized by its number of customers $T$, number of products $n$, and their preference weights. Letting $\cal I$ be the set of all possible instances, the adaptivity gap is formally defined as\begin{equation*}
    \max_{I\in {\cal I}} \frac{\opt^{\dpp}_I}{\opt^{\stat}_I},
\end{equation*}
where $\opt_I^{\stat}$ and $\opt_I^{\dpp}$ respectively denote the expected maximum load of an optimal static policy and an optimal dynamic policy for the instance $I$.
}
Quite surprisingly, we establish an adaptivity gap of at most $4$, showing that statically offering a weight-ordered assortment guarantees a $1/4$-approximation to \ref{DMLA}. Moreover, we show that this gap reduces to at most $2$ when all products have identical preference weights.

\paragraph{Outline.} In Section \ref{subsec:notationresult}, we provide some useful notation and describe our main adaptivity gap results in greater detail. In Section \ref{subsec:auxlemmas}, we present several auxiliary claims that will be helpful in the subsequent analysis. 
Then, we prove an adaptivity gap of at most $4$ for general instances in Section \ref{subsec:proofadaptivity}, deferring the improved finding for the identical-weight setting to Appendix \ref{apx:eqv}. Finally, we construct a \ref{DMLA} instance, demonstrating that the adaptivity gap of this problem is at least $4/3$.

\subsection{Notation and main results}\label{subsec:notationresult}

\paragraph{Notation.} Let us start by introducing some helpful notation and definitions. \blue{First, in the remainder of this section, we fix a single instance, consisting of $n$ products represented by the universe $\Nc$, their preference weights, and the number of customers $T$.} In what follows, we generalize the notion of preference-weight-ordered assortments to any universe of products $U\subseteq\Nc$, still prioritizing those with higher preference weights. Formally, suppose that $U=\{i_1,\ldots,i_k\}\subseteq \Nc$ and that $v_{i_1}\geq \cdots\geq v_{i_k}$. We say that the assortment $S\subseteq U$ is  preference-weight-ordered  in $U$ when $S = \{i_1,\ldots,i_m\}$ for some $1 \leq m\leq k$. With this definition, let $\opt^{\ord}(U)$ be the optimal expected maximum load achievable by a static preference-weight-ordered assortment in $U$. In other words, 
\[ \opt^{\ord}(U) = \max_{m=1,\ldots,k} \E\left(M\left(\{i_1,\ldots,i_m\}\right)\right). \]
In addition, we define $\opt^{\dpp}(U)$ as the expected maximum load of an optimal dynamic policy, using only products in $U$.

\paragraph{Main results.} Quite surprisingly, we show that by statically offering a weight-ordered assortment, one can attain an expected maximum load of at least $1/4$ of the optimal expected maximum load of \ref{DMLA}. The proof of this result appears in Section \ref{subsec:proofadaptivity}.  
  
\begin{theorem}\label{thm:adaptivitygap}
There exists a static weight-ordered assortment that provides a $1/4$-approximation to \ref{DMLA}, i.e., $\opt^{\ord}(\Nc)\geq \frac{1}{4} \cdot \opt^{\dpp}(\Nc)$.
\end{theorem}   

It is worth noting that $\opt^{\dpp}(\Nc)$ represents the expected maximum load of an optimal dynamic policy for \ref{DMLA}, where all products in $\Nc$ are considered. Additionally, $ \opt^{\ord}(\Nc)$ denotes the expected maximum load of an optimal weight-ordered static assortment in $\Nc$, which is clearly upper-bounded by the expected maximum load of an optimal assortment for \ref{SMLA}. Therefore, the adaptivity gap of this setting is at most $4$. Moreover, as explained in Section \ref{compute}, we can compute  $ \opt^{\ord}(\Nc)$ in polynomial time, meaning that the above theorem yields a $1/4$-approximation for \ref{DMLA}.

When all products are associated with identical preference weights, we derive an improved adaptivity gap of $2$, as stated in the next theorem, whose proof is 
provided in Appendix \ref{apx:eqv}.

\begin{theorem}\label{thm:eqv}
Suppose that all products have identical preference weights. Then, there exists a static weight-ordered  assortment that provides a $1/2$-approximation to \ref{DMLA}, i.e., $\opt^{\ord}(\Nc)\geq \frac{1}{2} \cdot \opt^{\dpp}(\Nc)$.
\end{theorem}


\subsection{Auxiliary claims}\label{subsec:auxlemmas}

\paragraph{Upper-bounding ${\opt^{\dpp}(U)}$.} For a fixed universe of products $U \subseteq \Nc$, recall that $\opt^{\dpp}(U)$ represents the expected maximum load attained by an optimal dynamic policy with respect to the universe $U$. In Lemma \ref{lem:multinom} below, whose proof is given in Appendix \ref{apx:multinom}, we provide an upper bound on $\opt^{\dpp}(U)$ which will serve as an initial step towards bounding the expected maximum load of an optimal dynamic policy. Specifically, we consider a random Multinomial vector $(L_1,\ldots,L_k)$, and establish a condition on its vector of probabilities and on the preference weights of products in $U$, ensuring that the expected maximal component of this Multinomial vector exceeds $\opt^{\dpp}(U)$.

\begin{lemma}\label{lem:multinom}
Let $(L_1,\ldots,L_k)$ be a random Multinomial vector with $T$ trials and probability vector $(p_1,\ldots,p_k)$. Let  $U\subseteq\Nc$ be a set of products with $\min_{i=1,\ldots,k}p_i\geq \max_{i\in U} \frac{v_i}{1+v_i}$. Then, $\E(\max(L_1,\ldots,L_k)) \geq \opt^{\dpp}(U)$.
\end{lemma}

To interpret the condition stated above, consider an optimal dynamic policy for the maximum load assortment problem with respect to the universe $U$. Whenever this policy offers an assortment $S$ to some customer $t \in [T]$, the MNL choice probability of each product $i\in S$ is $\frac{ v_i }{ 1+v(S) } \leq \frac{ v_i }{ 1+v_i }$. Therefore, $\min_{i=1,\ldots,k}p_i\geq \max_{i\in U} \frac{v_i}{1+v_i}$ can be viewed as a condition where, regardless of the offered assortment, the MNL choice probabilities of all products in $U$ are upper-bounded by $\min_{i=1,\ldots,k}p_i$. Under this condition, we prove that the expected maximal component of $(L_1,\ldots,L_k)$ is an upper bound on the  expected maximum load of an optimal dynamic policy.
 
\paragraph{Consequences of offering larger subsets.} Consider an adaptive policy $A$ for the maximum load assortment problem with respect to the universe $U$.  We denote by $\obj^A(U)$ the expected  maximum load achieved by this policy. Additionally, for each  $t=1,\ldots,T$, we make use of $S^A_t$ to designate the subset of $U$ offered by $A$ to customer $t$. This subset is clearly random, since it depends on the random selections made by previously arriving customers.
    
Now,  consider two adaptive policies, $A$ and $B$, such that for any $t \in [T]$, the assortment offered by policy $A$ to customer $t$ is almost surely a subset of the one offered by policy $B$ to this customer. However, we assume that the difference in total preference weight between these two assortment is almost surely upper-bounded by some $\epsilon \geq 0$. The next lemma, whose proof is included in Appendix \ref{apx:augment}, gives a lower bound on the ratio between the expected maximum loads of the two policies as a function of $\epsilon$. 

\begin{lemma}\label{lem:augment}
Let $A$ and $B$ be two adaptive policies with respect to the universe $U$. 
For every $t \in [T]$, suppose that $S^A_t\subseteq S^B_t $ and $v( S^B_t \setminus S^A_t) \leq \epsilon  $ almost surely. Then, $\obj^B(U)\geq \frac1{1+\epsilon}\cdot\obj^A(U)$.
\end{lemma}
	
\paragraph{Subadditivity.} Finally, we prove that $\opt^{\dpp}(\cdot)$ is a subadditive function, as formally stated below.

\begin{lemma} \label{lem:subadd}
For any $U_1,U_2\subseteq \Nc$, we have $\opt^{\dpp}(U_1\cup U_2) \leq \opt^{\dpp}(U_1) + \opt^{\dpp}(U_2)$.
\end{lemma}

The proof of this result relies on a coupling argument involving three dynamic policies: (1)~A policy that offers only products from the universe $U_1$; (2)~A policy offering only products from $U_2$; and (3)~An optimal policy that offers products from the combined universe, $U_1\cup U_2$. Within this coupling, we demonstrate that for at least one of the policies (1) and (2), its maximum load is almost surely at least as large as that of policy (3) with respect to the constructed coupling. The complete proof is provided in Appendix \ref{apx:subadd}.
 
\subsection{Proof of Theorem \ref{thm:adaptivitygap} }\label{subsec:proofadaptivity}

\paragraph{The easy regime: ${v(\Nc)< 1}$.} In this case, recalling that $v_1\geq \cdots\geq v_n$, we simply make use of the static policy $A$ where all products in $\Nc$ are offered to every customer $t \in [T]$. We will show that the expected maximum load of this policy is at least $\opt^{\dpp}(\Nc)/2$. To this end, focusing on an optimal dynamic policy $A^*$, let $S_t^{ A^* }$ be the random assortment it offers to customer $t$. We have $S_t^{ A^* } \subseteq \Nc$ and $v(\Nc \setminus S_t^{ A^* } )\leq 1$, since $v(\Nc)<1$ by the case hypothesis. 
Therefore, Lemma~\ref{lem:augment} implies that the expected maximum load $\obj^A(\Nc)$ of our static policy is at least $\obj^{A^*}(\Nc) / 2$. In other words, $\E (M(\Nc)) \geq \opt^{\dpp}(\Nc)/2$. Clearly, $\Nc$ is also preference-weight ordered, meaning that $\opt^{\ord}(\Nc)\geq \opt^{\dpp}(\Nc)/2$.

\paragraph{Overview of the difficult regime: ${v(\Nc) \geq 1}$.} 
Let $k$ be the minimal integer for which $\sum_{i=1}^kv_i\geq 1$, and consider the assortment $U=\{1,\ldots,k\}$. In the following, we argue that by statically offering this assortment to all customers, the expected maximum load is at least $\opt^{\dpp}(\Nc)/4$. In other words, $\E(M(U))\geq \opt^{\dpp}(\Nc)/4$. Since $U$ is weight-ordered, the latter bound would imply that $\opt^{\ord}(\Nc)\geq \opt^{\dpp}(\Nc)/4.$

For this purpose, by Lemma \ref{lem:subadd}, we know that 
$\opt^{\dpp}(\cdot)$ is a subadditive function, meaning in particular that
\begin{equation}\label{eq:subadd}
\opt^{\dpp}(\Nc) \leq \opt^{\dpp}(U)+\opt^{\dpp}(\Nc\setminus U).
\end{equation}
Now, let $(\widehat L_1,\ldots,\widehat L_k)$ be a Multinomial vector with $T$ trials and probability vector  $( p_1,\ldots,p_k)$, where $p_i = \frac{ v_i }{ v(U) }$ for every $i=1,\ldots,k$. In the next two lemmas, whose proofs appear in the sequel, we show that both $\opt^{\dpp}(U)$ and $\opt^{\dpp}(\Nc\setminus U)$ are upper bounded by $\E(\widehat M)$, where $\widehat M = \max_{i=1,\ldots,k}\widehat L_i$.

\begin{lemma}\label{lem:1}       
$\opt^{\dpp}(U) \leq \E(\widehat M)$.
\end{lemma}

\begin{lemma}\label{lem:2}
$\opt^{\dpp}(\Nc\setminus U) \leq \E(\widehat M)$.
\end{lemma}

On the other hand, we argue that the expected maximum load achieved by the static assortment $U$ is at least $\E(\widehat M)/2$.

\begin{lemma}\label{lem:3}
$\E (M(U)) \geq \E(\widehat M)/2$.
\end{lemma}

We are now ready to complete the proof of Theorem \ref{thm:adaptivitygap}. To this end, noting that the assortment $U$ is preference-weight-ordered, we know by Lemma~\ref{lem:3} that $\opt^{\ord}(\Nc)\geq \E(\widehat M)/2$. Consequently,
$$ \opt^{\ord}(\Nc)  \geq  \frac{ 1 }{ 2 } \cdot \E(\widehat M) \geq  \frac{ 1 }{ 4 } \cdot \left( \opt^{\dpp}(U) + \opt^{\dpp}(\Nc\setminus U) \right) 
 \geq \frac{ 1 }{ 4 } \cdot \opt^{\dpp}(\Nc),
$$
where the second inequality follows from Lemmas~\ref{lem:1}  and~\ref{lem:2}, and the third inequality is precisely~\eqref{eq:subadd}.

\paragraph{Proof of Lemma~\ref{lem:1}.} 
Let $(\widetilde L_1,\ldots,\widetilde L_k)$ be the random loads of the products when employing an optimal adaptive policy for the universe $U$, and let $\widetilde M=\max_{i\in U}\widetilde L_i$. By definition, $\opt^{\dpp}(U) = \E(\widetilde M)$, and our objective is to prove that $\E(\widetilde M) \leq \E(\widehat M)$.

For this purpose, we begin by observing that, for every $i\in [k]$,        \begin{equation}\label{eq:obs}
\frac{v_i}{v(U)} > \frac{v_i}{1+v_k}\geq \frac{v_i}{1+v_i},
\end{equation}
where the first and second inequalities hold respectively since $\sum_{j=1}^{k-1} v_j < 1$, by definition of $k$, and since $v_1 \geq \cdots\geq v_k$.
Looking into Equation~\eqref{eq:obs}, its left-hand side is exactly the  probability of component $i$, in the process of generating the random Multinomial vector $(\widehat L_1,\ldots,\widehat L_k)$. The right-hand side is the choice probability of product $i$ when the single-product assortment $\{i\}$ is offered, which is an upper bound on the choice probability of product $i$ at any step of generating $(\widetilde L_1,\ldots,\widetilde L_k)$. In other words, letting $S_t^{ A^* } \subseteq U$ be the random assortment offered by the optimal dynamic policy $A^*$ to customer $t \in [T]$, for every product $i \in [k]$, we have
\begin{equation}\label{eq:obs2}
\phi_i(S_t^{ A^* }) \leq \frac{v_i}{1+v_i} \leq  \frac{v_i}{v(U)}.
\end{equation}        
This observation implies the existence of a simple way to couple ${(\widetilde L_1,\ldots,\widetilde L_k)}$ and ${(\widehat L_1,\ldots,\widehat L_k)}$ such that ${\widetilde L_i\leq \widehat L_i}$, for every product $i \in [k]$. As a result, $\E(\widetilde M) \leq \E(\widehat M)$.


\paragraph{Proof of Lemma~\ref{lem:2}.} The key idea is to notice that $\frac{v_i}{v(U)} \geq \frac{v_i}{1+v_i} \geq \frac{v_j}{1+v_j}$, for every pair of products $1 \leq i \leq k$ and $k+1 \leq j \leq n$. Here, the first inequality is precisely Equation~\eqref{eq:obs}, and the second inequality holds since $v_1 \geq \cdots\geq v_n$. Therefore, $ \min_{i=1,\ldots,k}\frac{v_i}{v(U)}\geq\max_{j=k+1,\ldots,n} \frac{v_j}{1+v_j}$.
Recall that $v_i/v(U)$ is the probability of picking option $i$ in the Multinomial vector ${(\widehat L_1,\ldots,\widehat L_k)}$. Therefore, by applying Lemma \ref{lem:multinom}, we have $\E(\max{(\widehat L_1,\ldots,\widehat L_k)})\geq \opt^{\dpp}(\Nc\setminus U)$.
Finally, $\E(\widehat M)\geq \opt^{\dpp}(\Nc\setminus U)$, by definition of $\widehat M$.

\paragraph{Proof of Lemma~\ref{lem:3}.} Consider the static policy where we offer the weight-ordered assortment $U$ to every customer, and let $(L_0, L_1, \ldots, L_k)$ be the load vector associated with this policy. The latter vector follows a Multinomial distribution, where the choice probability of each product $i\in U$ is $\frac{ v_i }{ 1+v(U) }$ and the no-purchase option has probability $\frac{ 1 }{ 1+v(U) }$. On the other hand, consider the random vector $(\widehat L_1,\ldots,\widehat L_k)$, recalling that it has been defined as being Multinomial with $T$ trials and probabilities $( p_1,\ldots,p_k)$, where $p_i = \frac{ v_i }{ v(U) }$ for every $i=1,\ldots,k$. We complement this vector with $\widehat L_0$ that has a probability of $0$. 

We proceed by applying Lemma~\ref{lem:probs} to the Multinomial vectors $(L_0, L_1, \ldots, L_k)$ and $(\widehat L_0, \widehat L_1,\ldots,\widehat L_k)$. Specifically, since $v(U)\geq 1$, we have $\frac{v_i}{1+v(U)}\geq \frac12\cdot\frac{v_i}{v(U)}$,
implying that the conditions of this lemma are met with $\eps = 1/2$. It follows that ${\E(M(U))\geq \E(\widehat M)/2}$.

\subsection{Lower bound on the adaptivity gap}\label{subsec:lowbd}

While Theorem \ref{thm:adaptivitygap} attains an upper bound of $4$ on the adaptivity gap, we proceed to consider the opposite direction and provide a lower bound on this measure. In particular, we construct an instance of the maximum load assortment problem, demonstrating that the adaptivity gap of this setting is at least $4/3$. \blue{The upcoming  construction is motivated in Appendix \ref{apx:justif}, where we present a  numerical analysis of the adaptivity gap with respect to various problem parameters, in order to identify a regime where the largest adaptivity gaps are attained. In particular, we observe that high adaptivity gaps are attained when $T=2$ with uniform preference weights, which leads to our focus on such instances in the construction below. 
}

\begin{lemma}
The adaptivity gap of \ref{DMLA} is at least $4/3$.  
\end{lemma}

\paragraph{Instance.} In what follows, we consider an instance defined over a universe of $n$ products, each with a preference weight of $1$. In addition, the number of customers is $T=2$. Let us compare the optimal adaptive policy against the performance of an optimal static assortment.

\paragraph{The optimal dynamic policy.} Since products have identical preference weights, it is easy to verify that the optimal dynamic policy starts by offering the whole universe of products to the first customer. With probability $n/(1+n)$, she selects some product, say $i$. In this event, the optimal policy will offer the assortment $\{i\}$ to the second customer, as adding any other product can only cannibalize product $i$. In the complementary event, the first customer selects the no-purchase option, and in this case, the optimal policy offers the whole universe of products to the second customer. Therefore, by conditioning on the choice of the first customer, the expected maximum load of the optimal dynamic policy is given by
\begin{equation} \label{eqn:bad_inst_opt}
\opt^{\dpp}(\Nc) = \frac{n}{1+n}\cdot \left(1+\frac{1}{2}\right) + \frac{1}{1+n}\cdot \frac{n}{1+n} = \frac{ 3 }{ 2 } \cdot \left( 1 - \frac{ 1 }{ n+1 } \right) + \frac{ n }{ (1+n)^2 } . 
\end{equation}

\paragraph{The optimal static assortment.} For every $k=1,\ldots,n$, we compute the expected maximum load achieved by statically offering an assortment $S_k$ consisting of $k$ products. For this assortment, the maximum load is $2$ if and only if the same product is selected by both customers, which happens with probability $\frac{ k }{(1+k)^2}$. Similarly, the maximum load is $0$ if and only if the no-purchase option is selected by both customers, which happens with probability $\frac{1} { (1+k)^2 }$. It follows that the maximum load is $1$ with probability $\frac{ k^2 + k }{ (1+k)^2 }$, and a simple calculation shows that 
\begin{equation} \label{eqn:bad_inst_static}
\E(M(S_k)) = \frac{k^2 + 3k}{(1+k)^2}.
\end{equation}
To bound the latter expectation, elementary calculus arguments show that the function $x \mapsto \frac{x^2 + 3x}{(1+x)^2}$ attains its maximum value over $[0,\infty)$ at $x = 3$, and therefore $\max_{k \in [n]} \E(M(S_k)) \leq 9/8$

\paragraph{Lower bound on the adaptivity gap.} By combining equations~\eqref{eqn:bad_inst_opt} and~\eqref{eqn:bad_inst_static}, we obtain an adaptivity gap of at least
\[ \lim_{n \to \infty} \frac{ \opt^{\dpp}(\Nc) }{ \max_{k \in [n]} \E(M(S_k)) } \geq \frac{ 8 }{ 9 } \cdot \lim_{n \to \infty} \left( \frac{ 3 }{ 2 } \cdot \left( 1 - \frac{ 1 }{ n+1 } \right) + \frac{ n }{ (1+n)^2 } \right) = \frac{ 4 }{ 3 } . \]
As a side note, due to considering the case of identical preference weights, by Theorem~\ref{thm:eqv}, we know that the adaptivity gap in this case is upper bounded by $2$. 
It is important to note that the choice of the instance presented above is not arbitrary. Rather, it results from numerically optimizing the number of customer $T$ and their (uniform) preference weights to obtain the highest lower bound on the adaptivity gap.

\red{
\subsection{Numerical insights on the adaptivity gap}

Beyond the theoretical results presented in Sections \ref{subsec:notationresult} and \ref{subsec:lowbd}, several key numerical insights can be drawn from the numerical analysis we conduct in Appendix \ref{apx:justif}. These insights offer a clearer understanding of the scenarios in which dynamic policies outperform static ones, highlighting the potential advantages of using more complex adaptive strategies compared to static approaches.

The first insight, observable in Table \ref{fig:tableT} of Appendix \ref{apx:justif}, suggests that an increase in the customer base diminishes the advantage of employing a dynamic policy over a static one. This trend is reflected in the reduction of the adaptivity gap as the value of $T$ grows. To better grasp the underlying reasons, it is essential to identify which problem instances effectively leverage adaptiveness, and which do not. Specifically, in cases with a large number of customers, the flexibility to utilize adaptive policies is limited. These instances often require early commitment to specific products or subsets of products, making the different sample paths of the dynamic process resemble a static policy.

Secondly, when examining instances with a smaller customer base, Table \ref{fig:tablevariance} reveals that a larger adaptivity gap is observed in cases with a greater number of products. The underlying intuition here is that expanding the product set $\Nc$ enhances the flexibility of a dynamic policy, allowing it to initially offer larger assortments and thereby increasing the likelihood of capturing customer demand before making a final commitment. In contrast, a static policy is constrained by its commitment to a fixed assortment from the outset. Expanding the product universe may have little to no impact on the actual assortment offered under a static policy.
}

\section{The Dynamic Setting: Quasi-Polynomial \texorpdfstring{$\boldsymbol{(1-\eps)}$}{(1-eps)}-Approximate Policy}\label{sec:dynamic}

In this section, we shift our focus towards designing a truly near-optimal policy for \ref{DMLA}. Specifically, for any $\eps > 0$, we propose a $(1-\epsilon)$-approximate adaptive policy, admitting a quasi-polynomial time implementation. Our approach involves exploring a carefully selected class of policies with distinct properties, allowing one to  dramatically reduce the search space of seemingly-intractable dynamic programming ideas. 
\blue{Formally, we say that an algorithm terminates in quasi-polynomial time when, for every instance $I$, its running time is $O(|I|^{{\rm polylog}|I|})$, where $|I|$ stands for the input size in its binary representation.}
We start by presenting our main result in Section \ref{subsec:maindynamic}, stating the existence of a near-optimal dynamic policy that can be implemented in quasi-polynomial time. Subsequently, Section \ref{subsec:highr} will provide a number of auxiliary lemmas and observations. In Section \ref{subsec:policy}, we present the specifics of our algorithmic approach \red{by formally constructing a near-optimal dynamic policy}.

\subsection{Main result}\label{subsec:maindynamic}

\paragraph{Main result.} As previously mentioned, our primary result consists of a quasi-polynomial time approximation scheme (QPTAS) for \ref{DMLA}. The next theorem describes this finding in greater detail, noting that $O_{\eps}( \cdot )$ simply hides polynomial dependencies in $1/\eps$. 

\begin{theorem}\label{thm:quasitime}
For any $\eps > 0$, we can compute an adaptive policy for \ref{DMLA} whose expected maximum load is within  factor $1 - \eps$ of optimal. This policy can be implemented in $O( n^{ O_{\eps}( \log^3 n) } )$.
\end{theorem}

It is worth mentioning that the term ``implementation'' in this specific context encompasses two important aspects. Firstly, it includes any preprocessing steps undertaken prior to the beginning of the customer arrival process. Secondly, it contains the additional procedures required to compute a personalized assortment that will be offered to each arriving customer. 
\blue{As stated above, the running time of our algorithm is $O(n^{O_{\epsilon}(\log^3 n)})$, meaning that it indeed qualifies as a QPTAS. This running time is not as efficient as a PTAS, in which the exponent would have been dependent only on $1/\eps$.}

\paragraph{Technical overview.} The fundamental challenge in addressing the dynamic setting arises from the exponential size of its dynamic programming state space (see Section~\ref{subsec:DMLA}). Thus, our initial focus lies in modifying the original instance, with the intent of arriving at a dramatically scaled-down state space. This alteration involves transforming the original universe of products $\Nc$ into a modified universe, where product weights are slightly altered. Additionally, we confine our exploration to a specific class of policies that truncates the arrival sequence once a predetermined threshold on the maximum load is reached, incurring at most an $\eps$-fraction loss in the objective function. While the idea of altering the space of products $\Nc$ helps mitigate the search space issue, it introduces a new source of complexity, due to the dissimilarity between the products in the new universe and our initial universe. Consequently, our second step consists of recovering a policy with respect to the original universe, while essentially preserving the expected maximum load.

\subsection{Useful claims}\label{subsec:highr}

In this section, we introduce several auxiliary claims that will be helpful in designing our near-optimal policy as well as in its analysis. For convenience, we assume without loss of generality that $T\geq 2$ and $n\geq 2$. Indeed, when $T = 1$, it is optimal to offer the whole universe of products. Similarly, the setting of $n=1$ corresponds to having a single product, in which case it is optimal to offer this product to every customer.

\paragraph{Two parametric regimes.} Let us start by explaining how any given instance can be classified into two possible regimes, referred to as high-weight and low-weight. For this purpose, let $\alpha = v_{\max}/(1+v_{\max})$, which is precisely the choice probability of the heaviest product by a single customer, when it is the only one offered. Consequently, $T\alpha$ represents the expected load of this product when it is being statically offered to all customers. Then, the high-weight regime captures problem instances where $T\alpha\geq 12\ln(nT)/\eps^3$ or $v_{\max}\geq 1/\eps$, in which case we will show that a very simple static policy achieves a $(1-\eps)$-approximation. As explained later on, the core difficulty lies within the low-weight regime, where $T\alpha< 12\ln(nT)/\eps^3$ and \blue{$v_{\max}<1/\eps$}, which will be the main focus of this section. 

\paragraph{High weight regime: $\boldsymbol{{T\alpha\geq 12\ln(nT)/\eps^3}}$ or $\boldsymbol{{v_{\max}\geq 1/\eps}}$.}  In this case, we prove that statically offering the heaviest product to all customers provides a $(1-\eps)$-approximation to \ref{DMLA}, as formally stated below. The proof of this result appears in Appendix \ref{apx:high}. 

\begin{lemma}\label{lem:high}
When $T\alpha \geq \frac{12\ln(nT)}{\epsilon^3}$ or $v_{\max}\geq 1/\eps$, the static policy that offers the heaviest product guarantees a $(1-\epsilon)$-approximation for  \ref{DMLA}.
\end{lemma}

Roughly speaking, the high-weight regime allows us to efficiently employ concentration bounds. These bounds demonstrate that, for any given policy, the event where its (random) maximum load exceeds $T\alpha$ by a non-negligible factor is highly improbable. Consequently, with the right choice of parameters, we will show that the optimal expected maximum load is upper-bounded by $(1+\eps)\cdot T\alpha$. As such, Lemma~\ref{lem:high} will follow by recalling that $T\alpha$ represents the expected maximum load of statically offering the heaviest product to all customers.


\paragraph{Low weight regime: $\boldsymbol{{T\alpha< 12\ln(nT)/\eps^3}}$ and $\boldsymbol{{v_{\max}< 1/\eps}}$.}
In this case, we establish a polylogarithmic upper bound on the optimal expected maximum load, which will be utilized in Section~\ref{subsec:policy} to prove the near-optimality of a specific class of policies. Intuitively, under the low weight regime, products are not associated with high enough weights to prompt frequent selections of the same product. By formalizing this intuition, we show that within the low weight regime, the expected maximum load is polylogarithmically bounded. The proof of this result is included in Appendix~\ref{apx:low}.

\begin{lemma}\label{lem:low}
When $T\alpha< \frac{12\ln(nT)}{\epsilon^3}$ and $v_{\max}< 1/\eps$, we have $\opt^{\dpp}(\Nc) \leq  \frac{300\ln^2(nT)}{\eps^6}$.
\end{lemma}

\subsection{Constructing our policy}\label{subsec:policy}

According to Lemma \ref{lem:high}, statically offering the heaviest product to all customers achieves a $(1-\eps)$-approximation of the optimum under the high-weight regime. Therefore, in the remainder of this section, we focus on the low-weight regime. For convenience, let $\thr =  \frac{300\ln^2(nT)}{\eps^6}$ 
be the upper bound we obtained in Lemma \ref{lem:low} on the expected maximum load $\opt^{\dpp}(\Nc)$ of an optimal dynamic policy. The primary idea behind our adaptive policy is to only explore policies that:
\begin{itemize}
    \item Stop offering products as soon as the maximum load reaches a value of $\thr/\Teps$, assuming without loss of generality that the latter term is an integer. Namely, the empty assortment is offered to all remaining customers once we hit this threshold.

    \item Avoid offering products of ``tiny'' preference weight, upper-bounded by $\Teps^2 v_{\max}/n$.
\end{itemize}
We refer to such adaptive policies as {\em truncated} policies. The motivation behind this restriction is that it allows us to considerably shrink the search space, and in particular, to compute a near-optimal policy in quasi-polynomial time. An important question that remains to be answered is obviously centered around the performance guarantee of such policies. Our analysis will argue that, based on truncated policies, we indeed construct a near-optimal dynamic policy. 

\paragraph{Step 1: Dropping light products.} We start by considering a new universe of products $U$, where we drop all products whose preference weight is at most $\Teps^2\cdot v_{\max}/n$, i.e., $U= \{i\in \Nc\,\mid\, v_i>\Teps^2\cdot v_{\max}/n\}$. Clearly, any $U$-policy is also an $\Nc$-policy, with the restriction of not offering any products whose preference weight is at most $\Teps^2\cdot v_{\max}/n$. We assume without loss of generality that $U=\{1,\ldots,k\}$ for some $k\leq n$, and let $v_{\min}$ be the smallest weight among the products in $U$. By construction, $v_{\max}/v_{\min} \leq n/\Teps^2$.

\paragraph{Step 2: Creating weight classes.} We create a new universe of products $\widetilde U$ by modifying  $U$ as follows. First, we partition the interval $[v_{\min},v_{\max}]$ geometrically by powers of $1+\Teps$, into a collection of buckets $I_1,I_2,\ldots,I_J$, where $J = \lfloor\frac{\log(v_{\max}/v_{\min})}{\log(1+\Teps)}\rfloor = O_{\eps}(\log n)$. Formally, 
\[ I_j = \left[ v_{\min}\cdot(1+\Teps)^j,v_{\min}\cdot(1+\Teps)^{j+1} \right), \]
for $j=0,\ldots,J$. Now, we associate a product $\widetilde i$ to each product $i\in U$, whose weight  is the left endpoint of the bucket containing $v_i$. In other words, the universe of products $\widetilde U=\{\widetilde 1,\ldots,\widetilde k\}$ is created such that, for every product $i=1,\ldots,k$, we determine the interval $I_j$ where $v_i$ resides, and then set ${v}_{\widetilde i} = (1+\Teps)^j \cdot v_{\min}$. Consequently, the products in $\widetilde U$ take only $O(J)$ possible values.

\paragraph{Step 3: Solving a reduced dynamic program.} We proceed by explaining how to compute an optimal truncated policy $\widetilde A$ with respect to the universe $\widetilde U$. To this end, let us define {\em constrained} load vectors $\mathbf \bl=(\ell_1,\ldots,\ell_k)$ as those whose maximal component is at most our threshold, i.e., $\max_{i=1,\ldots,k}\ell_i \leq \thr/\Teps$. We denote by $\const$ the collection of all such vectors. While at a first glance, the number of constrained vectors is exponential in $n$, we present in \red{Appendix~\ref{apx:proofquasitime}} an efficient representation of these vectors, that effectively reduces their number to a quasi-polynomial magnitude.  Now, in order to compute an optimal truncated policy $\widetilde A$ with respect to $\widetilde U$, we solve the so-called reduced dynamic program, for every load vector $\bl \in \const$, given through the following recursive equations:   \begin{equation}\label{eq:reducedDP}
    		M_{t}(\mathbf{\bl})=\max_{S\subseteq \widetilde U}\left(M_{t-1}(\mathbf \bl)\cdot\phi_0(S)+\sum\limits_{i\in S}M_{t-1}(\mathbf{\bl}+\mathbf{e}_i)\cdot \phi_i(S)\right).
	   \end{equation}
However, we modify the boundary conditions of this program, such that $M_t(\bl) = \thr/\Teps$ for every vector $\bl$ with $\max_{i=1,\ldots,k}\ell_i=\thr/\Teps$.     
Note that the recursive equations here are identical to those characterizing \ref{DMLA} in Section \ref{subsec:DMLA}, as these two programs only differ in their boundary condition.

\paragraph{Step 4: Recovering the approximate policy.} 
\red{Our final step consists of converting the $\widetilde U$-policy $\widetilde A$ from Step~3 into an approximate policy $A$ with respect to $U$. To this end, in Appendix~\ref{apx:altering}, we introduce  the so-called $\eps$-tightness property of $U$ and $\widetilde U$, which stipulates that for every assortment $S\subseteq U$ and every product $i \in S$, we have $\phi_i(S) \geq (1-\eps)\cdot \phi_{\widetilde i}(\widetilde S)$. It is easy to see that this property is satisfied in our case, as we have $$
\phi_i(S) \geq \frac{v_{\widetilde i}}{1+\frac{1}{1-\Teps}\cdot \sum_{j\in S}v_{\widetilde j}}\geq (1-\Teps)\cdot \frac{v_{\widetilde i}}{1+\sum_{\widetilde j\in\widetilde S}v_{\widetilde j}}=(1-\Teps)\cdot \phi_{\widetilde i}(\widetilde S).
$$
Deferring the technical details to Appendix~\ref{apx:altering}, we show that Lemma~\ref{lem:alteruniverse} enables us to recover a policy $A$ whose expected maximum load is \begin{equation}\label{eq:lemmaequation}
    \obj^{A}\geq (1-\Teps)\cdot \obj^{\widetilde A}.
\end{equation}
In order to conclude the proof of Theorem \ref{thm:quasitime}, it remains to show that the policy described above can indeed be implemented in $O( n^{ O_{\eps}( \log^3 n) } )$ time, and that the computed policy attains the desired performance guarantee, namely achieving a $(1-\eps)$-approximation for \ref{DMLA}. We defer these technical proofs and thereby the conclusion of Theorem \ref{thm:quasitime} to Appendix \ref{apx:proofquasitime}.
}

\section{Numerical Studies}\label{sec:numerics}

In this section, we conduct numerical experiments in order to examine how optimal assortments for \ref{SMLA} behave with respect to relevant model primitives, namely, the number of customers $T$, and the preference weights $v_1,,\ldots,v_n$. These experiments are mostly intended to study the sensitivity of optimal assortments with respect to model primitives.
\subsection{Effect of the parameter \texorpdfstring{$\boldsymbol T$}{T}}\label{subsec:Teffect}
Here, we study how the number of customers $T$ affects the number of products in an optimal assortment of \ref{SMLA}. We consider the following experimental setup: We fix the number of products at $n=10$, and generate preference weights from the positive part of a normal distribution with mean $\mu$ and standard deviation $\mu/2$, varying $\mu$ in the range $\{0.05, 0.1, 0.3, 0.5\}$. We also vary the number of customers $T$ in the range $\{2, 3, \ldots, 12\}$. For $T=1$, offering the entire universe of products is clearly optimal, and we therefore exclude this case. For each pair $(T,\mu)$ in the specified ranges, we generate $1000$ \ref{SMLA} instances and determine the size of an optimal assortment by exhaustively enumerating all feasible assortments.
When the optimal assortment is not unique, we represent each assortment $S$  by an $n$-dimensional binary vector whose $i$-th entry is $1$ if and only if $i\in S$, and report the first optimal assortment in  lexicographical order. This choice does not affect our results, as we observe that the optimal assortment in these experiments is unique due to the randomness in generating preference weights.


\paragraph{Results.} Our results are summarized in Figure \ref{fig:assortmentsizeT}, where we generate a plot for each value of $\mu$ considered. Specifically, for every $T=2, \ldots, 10$ (depicted on the x-axis), we compute the size of the optimal assortment (depicted on the y-axis) for the $1000$ generated instances. Subsequently, we present a box plot highlighting the quartiles of the obtained optimal assortment sizes. In particular, for every value of $T$, the endpoints of the vertical line delimit the range of values taken by the optimal assortment sizes, excluding outliers, which are represented by the points outside the delimited region. The extremities of each box represent the first and third quartile, and the horizontal line inside this box represents the median. The dotted line inside the box represents the mean.
\begin{figure}[htbp!]
    \centering
    \includegraphics[scale=0.35]{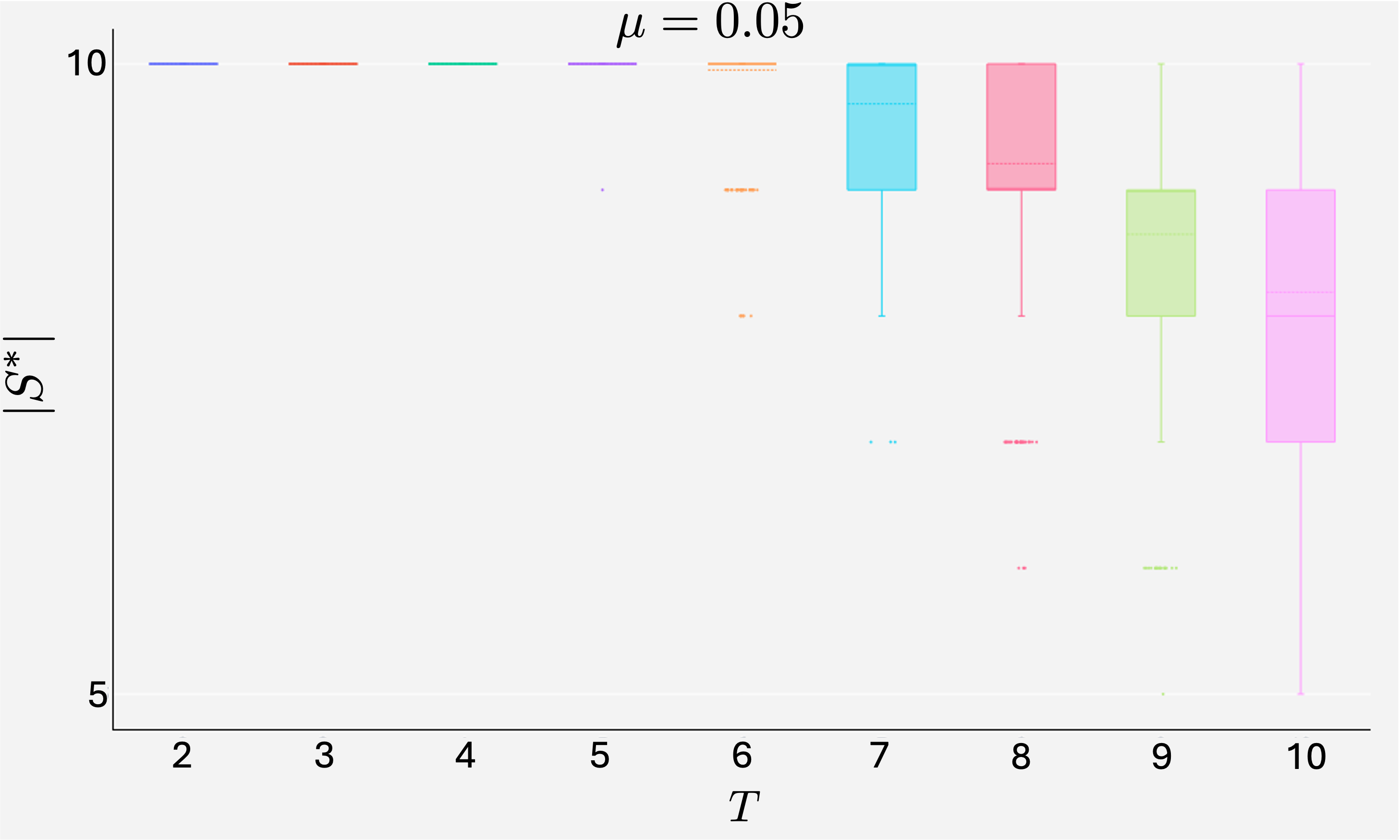}
    \includegraphics[scale=0.35]{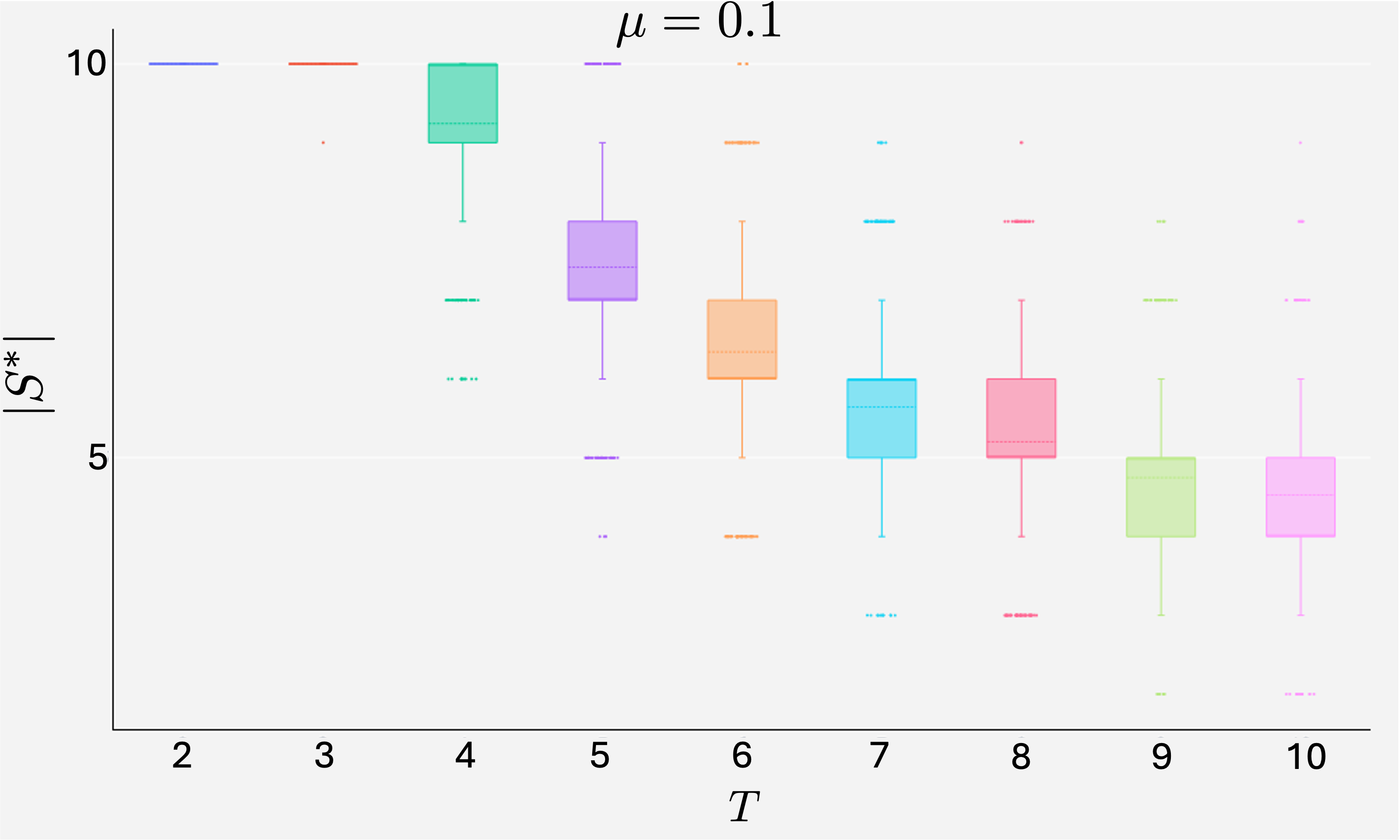}
    \includegraphics[scale=0.35]{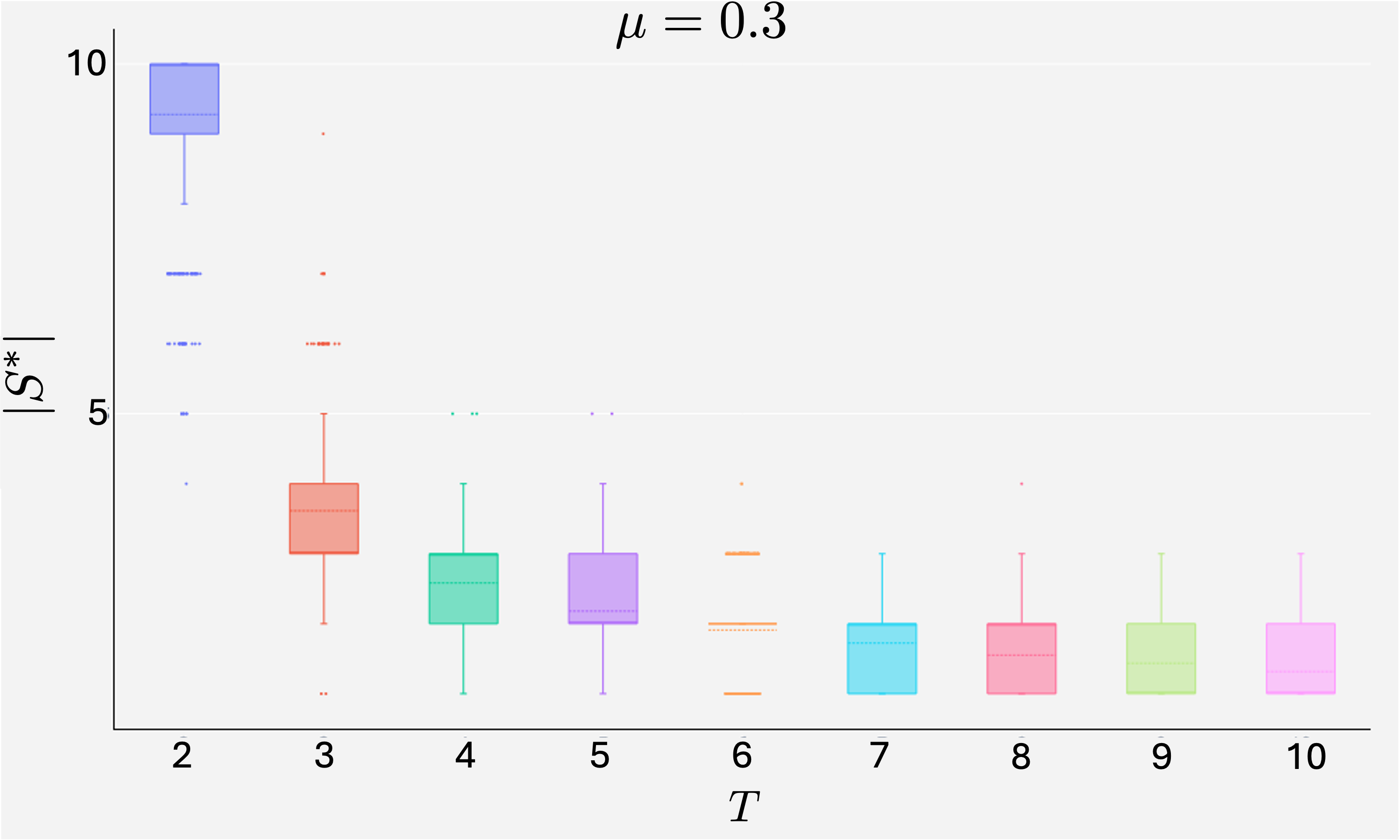}
    \includegraphics[scale=0.35]{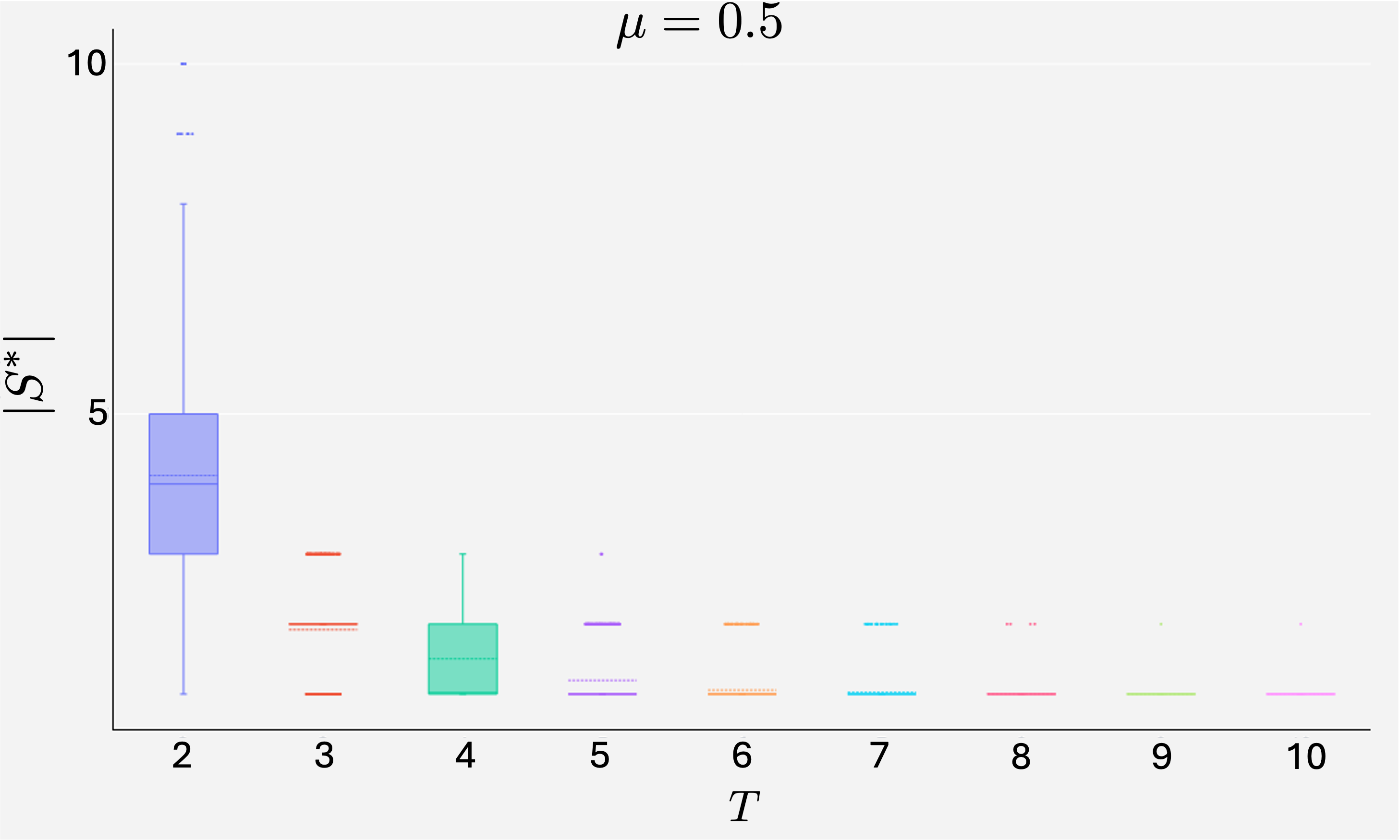}
    \caption{Optimal assortment size behaviour with respect to $\boldsymbol T$.}
    \label{fig:assortmentsizeT}
\end{figure}


We observe that for all values of $\mu$, the optimal assortment size tends to decrease as the number of customers $T$ grows. In particular, when $T$ is large enough, the optimal assortment ends up converging towards a single-product assortment, containing the heaviest product. Conversely, smaller values of $T$ result in larger assortments. \red{In particular, for $\mu = 0.05$ and $T\leq 5$, offering the whole universe of products is consistently optimal over all generated instances; then, the optimal assortment size starts decreasing for larger values of $T$}. The rate of decrease depends on the preference weights, as one can observe in the four plots of Figure \ref{fig:assortmentsizeT}.
To interpret this behavior, recall that the downside to offering a large number of products is the potential dispersal of demand across all offered options. However, when $\mu = 0.05$ and $T \leq 5$, the probability of capturing two or more customers is sufficiently small. Specifically, if there are either $0$ or $1$ captured customers, there is no demand to disperse. Consequently, the downside of offering more products does not play a meaningful role in this case, resulting in larger optimal assortment sizes. Moreover, by comparing the plots themselves, we notice that the settings with higher values for $\mu$ yield optimal assortments with smaller sizes. We further explore this trend in Section \ref{subsec:weighteffect}.

As expected, the phenomenon where the optimal assortment size shrinks as $T$ increases aligns with Lemma \ref{lem:largeT}, which states that offering only the heaviest product becomes optimal with sufficiently many customers. Let us provide an interpretation of this observation. When we offer an assortment $S$, the load vector of the products follows a multinomial distribution (see Section \ref{sec:SMLA}). In particular, these loads are negatively-correlated binomial random variables. By virtue of the concentration of binomial random variables around their mean, as $T$ grows, the loads of the less attractive products in $S$ (i.e., with lower preference weights) become less likely to surpass those of the more attractive products. In other words, the effect of cannibalization of the attractive products by the additional products is accentuated. Consequently, as $T$ grows, removing a greater number of unattractive products becomes optimal. This interpretation is the basic intuition behind the proof of Lemma \ref{lem:largeT}.

\subsection{Effect of the preference weights}\label{subsec:weighteffect}
In what follows, we study how the optimal assortment size behaves with respect to the preference weights. In the current experimental setup, the number of products is once again fixed at $n=10$, and the preference weights are drawn from the positive part of a normal distribution with mean $\mu$ and standard deviation $\mu/2$. The number of customers $T$ is varied in the range $\{2,5,8,10\}$, and the parameter $\mu$ is varied on an additive grid in $[0.1, 1]$ with a step size of $0.1$. For each pair $(T,\mu)$, we sample $1000$ \ref{SMLA} instances and compute the optimal assortment size through exhaustive enumeration. 



\paragraph{Results.} Our results are summarized in Figure \ref{fig:assortmentsizemu}, where the x-axis of each plot represents the values of $\mu$, and for each such value, we create a box plot describing the quartiles of the obtained optimal assortment sizes, similarly to Figure \ref{fig:assortmentsizeT}. We notice that the optimal assortment size generally decreases with respect to the parameter $\mu$. In other words, when the preference weights are larger on average, fewer products are included in an optimal static assortment. To better understand this behaviour, let us consider the two extreme cases, where the preference weights are either close $0$ or very large. In the former case, the probability that a customer purchases a given product is very small, and in particular, the probability that any product is purchased twice or more is very small. Consequently, when preference weights are small, the cannibalization effect is marginal, and adding products to our assortment guarantees a larger captured portion of customers. \red{Consequently, optimal assortments tend to increase in size for small values of preference weights, even reaching the whole universe of products for a number of instances with $\mu = 0.1$ and $T\in \{2,5,8\}$.} When preference weights are large, a considerable portion of customers is captured even with small-sized assortments. In particular, when there exists a product $i$ with $v_i \gg 1$, offering only this product would guarantee a load of nearly $T$ with high probability, due to basic concentration arguments. Therefore, adding products would only cannibalize this product. In other words, the cannibalization effect outweighs the marginal increase in the captured portion of customers. In the general case, adding products to any given assortment induces both a cannibalization effect and a marginal increase in the captured portion of customers. In our experiments, for greater values of the preference weights, the benefit of capturing a larger portion of customers is overshadowed by the loss incurred due to cannibalization. 

\begin{figure}[htbp!]
    \centering
    \includegraphics[scale=0.35]{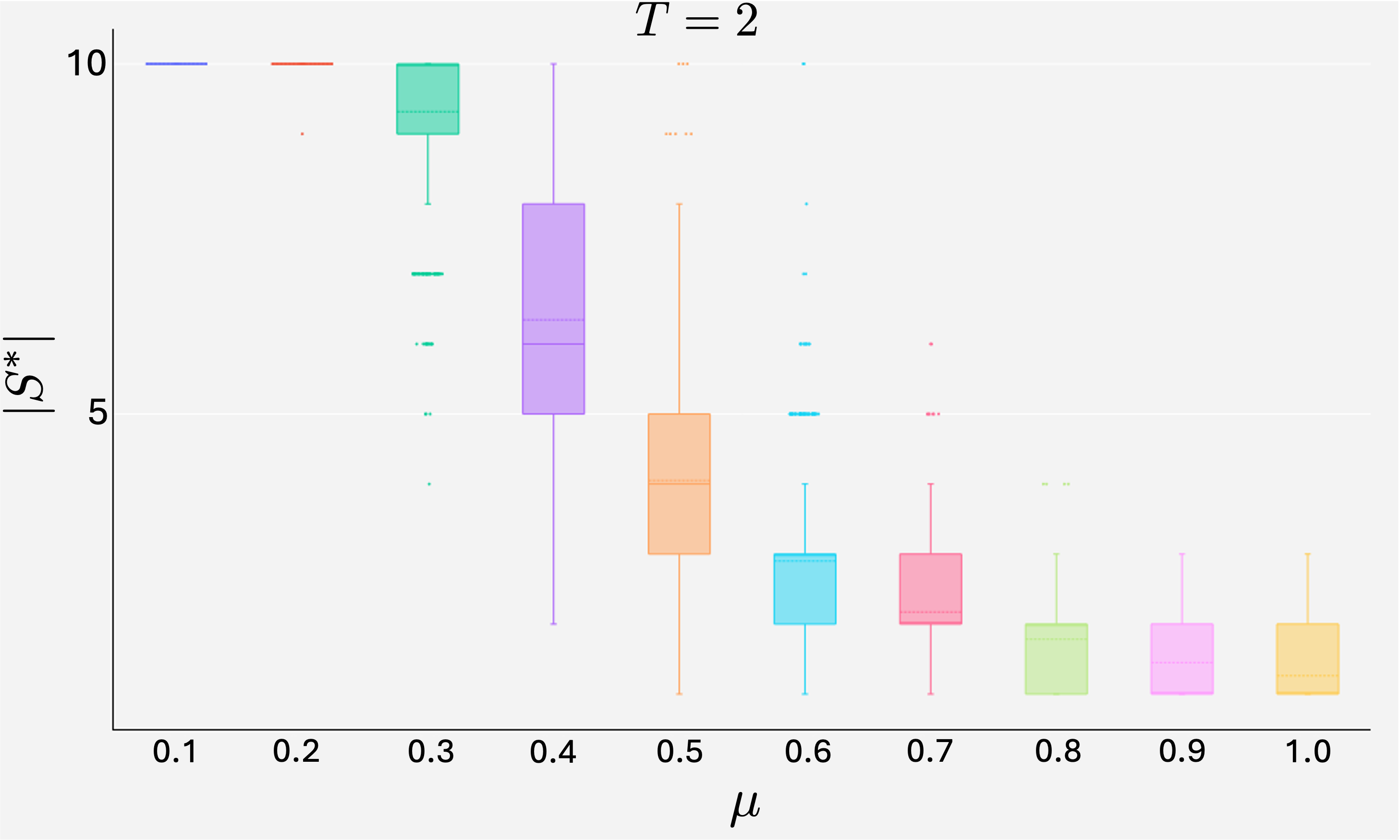}
    \includegraphics[scale=0.35]{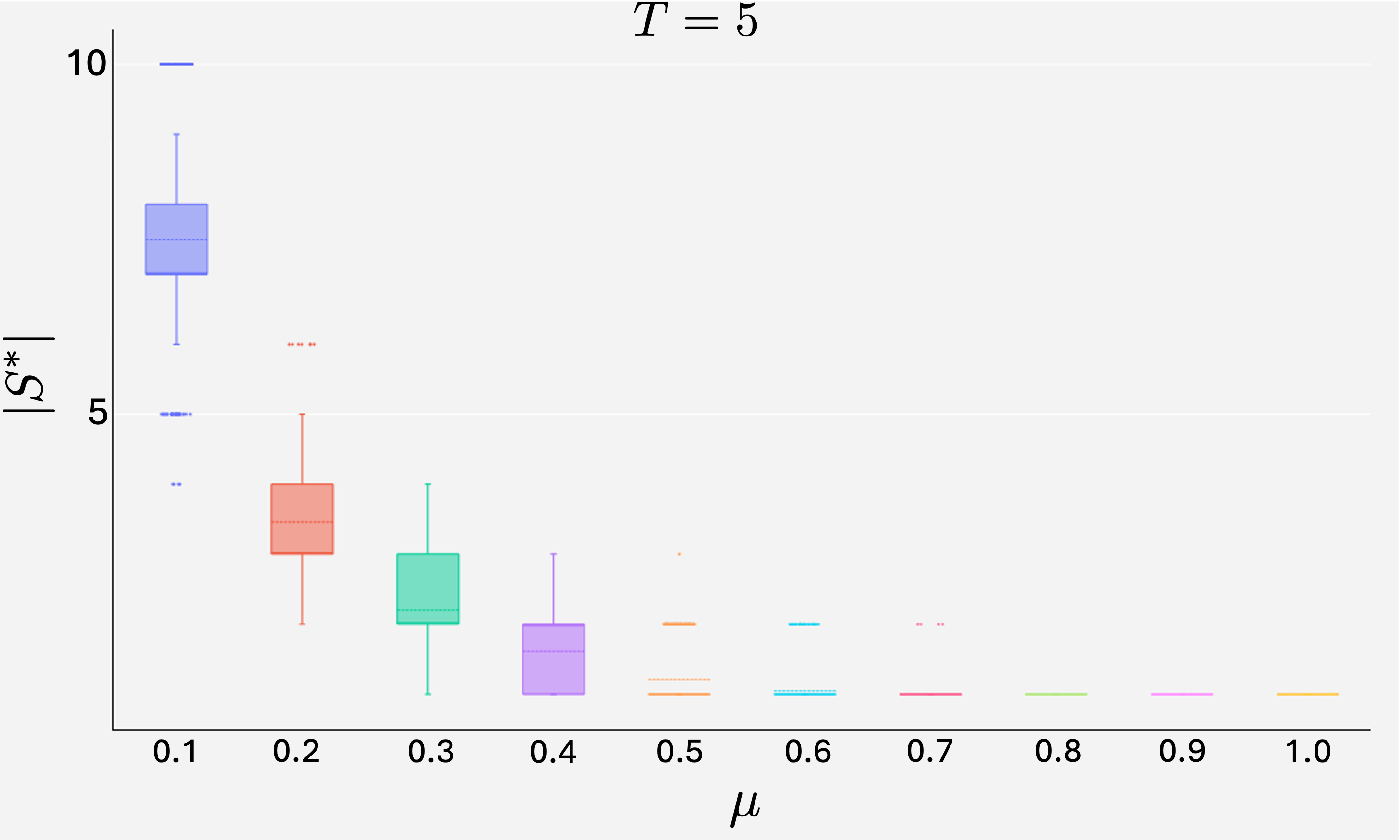}
    \includegraphics[scale=0.35]{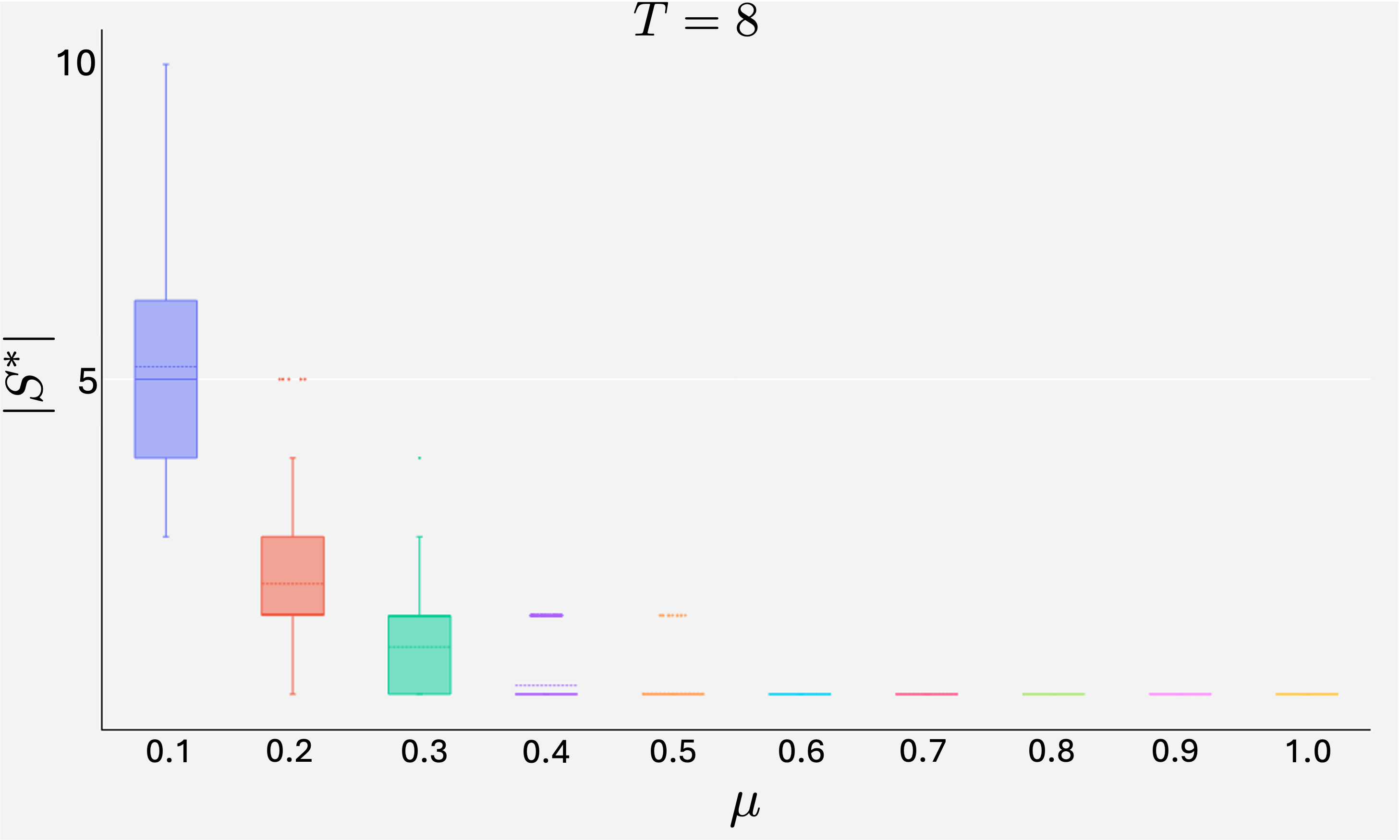}
    \includegraphics[scale=0.35]{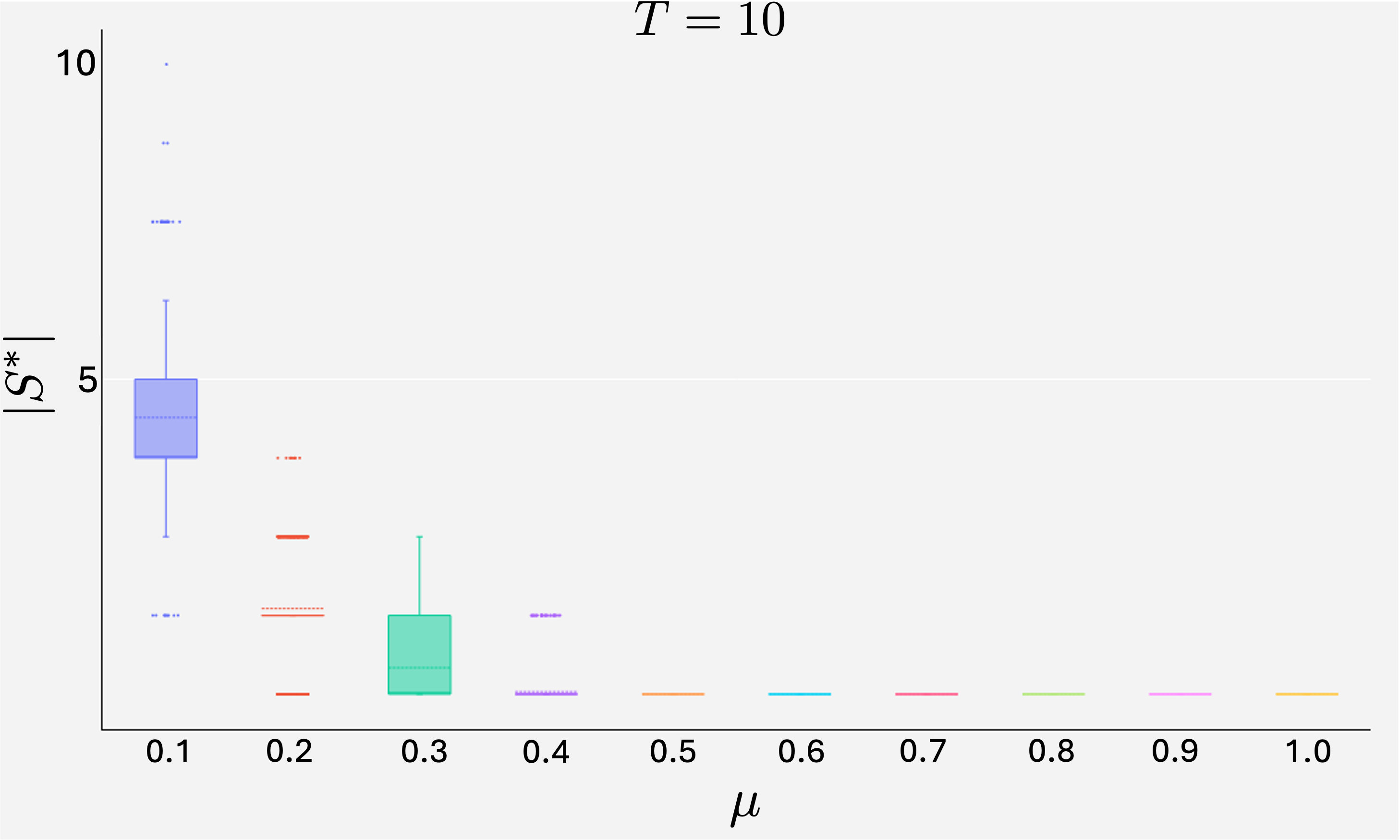}
    \caption{Optimal assortment size behaviour with respect to $\boldsymbol \mu$.}
    \label{fig:assortmentsizemu}
\end{figure}
\section{Concluding Remarks} \label{sec:conclusion}


This comprehensive study elucidates the potential of assortment optimization in manipulating customer choices towards maximum product selection, offering rigorous methods to address contemporary applications such as Attended Home Delivery and Preference-based Group Scheduling.
We believe that our work lays solid foundations for Maximum Load Assortment Optimization, potentially being the onset of further exploration. In what follows, we discuss several intriguing open questions, along with particularly appealing extensions of our modeling approach.

\red{\paragraph{Hardness of the static formulation?} Despite our best efforts, the computational complexity of the \ref{SMLA} problem remains an open question. Specifically, we still do not know whether this setting is NP-hard or whether optimal static assortments can be computed in polynomial time. This question is particularly challenging due to the unique problem structure, appearing to require either innovative optimization techniques or hardness proofs that are very different from what one typically meets in assortment optimization.

\paragraph{Dynamic formulation: Improved bounds and tightness of adaptivity gap?} In the dynamic setting,  devising a polynomial-time ($1-\eps$)-approximate policy poses a great technical challenge due to the inherent high-dimensional nature of this problem. Through new algorithmic techniques, we have been successful at attaining quasi-polynomial running times; however, further progress seem to necessitate yet-uncharted ideas. On a different front, even though we have established a lower bound of $4/3$ on the adaptivity gap of this problem, and an upper bound of $4$, there is still a meaningful room for improved constructions in this context, potentially bridging this gap and identifying the exact adaptivity gap.
}

\paragraph{Practical applications.}
Our work's practical implications bring forth captivating questions. In future research, it would be interesting to conduct data-driven case studies, examining the applicability of maximum load assortment optimization in real-world settings, thereby bridging the gap between theory and practice. Furthermore, exploring new domains and industries beyond Attended Home Delivery and Preference-based Group Scheduling could uncover novel challenges and untapped practical impact.

\paragraph{Extensions.} Along the above-mentioned lines, 
extending our problem formulation to additional families of choice models, such as the Markov Chain model \citep{blanchet2016markov,feldman2017revenue} or the non-parametric ranking-based model \citep{farias2013nonparametric, aouad2018approximability}, is an interesting direction for future research.
\red{While our evaluation oracle (see Section \ref{compute}) could, in theory, be extended to other choice models, most of our analysis relies on specific properties of the MNL model. For instance, the crucial merge and transfer operations depend on the invariance of choice probabilities for uninvolved products. Defining these operations in other models, such as Nested Logit or Markov Chain, is not straightforward. Moreover, in the dynamic setting, even solving a single step of the dynamic program can be NP-hard for models such as Mixture of Multinomial Logits. Extending our results to other choice models would therefore require new approaches, outside the scope of this work.}
Yet another fundamental question is that of exploring a wide array of constraints on the offered assortments, such cardinality, capacity, and matroid constraints. Finally, it would be interesting to investigate an extended formulation, where our goal is to optimize the expected summation of $k$-highest loads rather than solely focusing on the maximum load. At present time, this  objective function appears to be significantly more challenging to deal with.

\addcontentsline{toc}{section}{Bibliography}
\bibliographystyle{plainnat}
\bibliography{biblio}

\appendix

\section{Proofs from Section \ref{sec:SMLA}}\label{apx:static}

\subsection{Proof of Lemma \ref{lem:merge}}\label{apx:merge}

Assume without loss of generality that $S=\{1,\ldots,k\}$, and that we merge products $1$ and $2$ to obtain the assortment $\widetilde S$. Recall that  $\mathbf L(S)=(L_1(S),L_2(S),\ldots,L_k(S))$ is the random variable denoting the load vector when offering the assortment $S$. With the same notation for $\widetilde S$, it is easy to verify that $\mathbf L(\widetilde S)$ is equal in distribution to $(L_1(S)+L_2(S),L_3(S),\ldots,L_k(S))$. 
Clearly,
$$\max(L_1(S)+L_2(S),L_3(S),\ldots,L_k(S))\geq \max(L_1(S),L_2(S),\ldots,L_k(S)),$$
and by taking expectations we get
\begin{eqnarray*}
\E(M(\widetilde S)) & = & \E ( \max(L_1(S)+L_2(S),L_3(S),\ldots,L_k(S)) ) \\
& \geq & \E(\max (L_1(S), L_2(S),\ldots,L_k(S))) \\
& = & \E(M(S)).
\end{eqnarray*}

\subsection{Proof of Lemma \ref{lem:charging}}\label{apx:charging}

Assume without loss of generality that $S=\{1,\ldots,k\}$, and that we perform a $\delta$-weight transfer from product $2$ to product $1$, where $v_2 \leq v_1$ and $0 \leq \delta \leq v_2 $, obtaining the assortment $\widetilde S_\delta$. Let $v=(v_1+v_2)/2$. For any $\omega \in[0,v]$, we define $1_\omega$ and $2_\omega$ as virtual products with respective preference weights $v+\omega$ and $v-\omega$. Let $S_\omega$ be the assortment that results from $S$, after replacing products $1$ and $2$ with the virtual products $1_\omega$ and $2_\omega$, i.e., $S_\omega = \{1_\omega, 2_\omega \}\cup \{3, \ldots,k\}$. In order to prove Lemma \ref{lem:charging}, we establish the following claim in Appendix~\ref{app:proof_claim_apx_mono}.

\begin{claim} \label{claim:apx:mono}
The function $\omega\mapsto \E(M(S_\omega))$ is monotonically non-decreasing across the interval $[0,v]$.
\end{claim}

Let us show how this claim implies the result stated in Lemma~\ref{lem:charging}. Let $\omega_1=(v_1-v_2)/2$ and $\omega_2= \omega_1+\delta$, noting that $0 \leq \omega_1 \leq \omega_2 \leq v$. Hence, Claim~\ref{claim:apx:mono} implies that  $\E(M( S_{\omega_2}))\geq \E(M(S_{\omega_1}))$. However, $v+\omega_1= v_1$ and $v-\omega_1=v_2$, which means that $S_{\omega_1} = S$. On the other hand, $v+\omega_2= v_1+\delta$ and $v-\omega_2=v_2-\delta$, which means that $S_{\omega_2} =  \widetilde S_\delta$. Therefore, $\E(M(\widetilde S_\delta))\geq \E(M(S))$.

\subsection{Proof of Claim~\ref{claim:apx:mono}} \label{app:proof_claim_apx_mono}

To facilitate our analysis, let us define $V=\sum_{i\in S}v_i$, which represents the total preference weights of all products in the assortment $S$. It is important to note that, for any value of $\omega$, the total of preference weights of all products in $S_\omega$ is equal to $V$ as well. Additionally, we define $p = v/(1+V)$ and $p_i=v_i/(1+V)$ for $i\in S$. 

Instead of working directly with the variable $\omega$, we perform the following change of variables, ${q=\omega/(1+V)}$. As such, by defining the function $f(q) = \E(M(S_{(1+V)\cdot q}))$, it suffices to prove that $f$ is monotonically non-decreasing across  the interval $[0,p]$. To this end, for any $\omega \in [0,v]$, when we offer assortment $S_{\omega}$, the MNL choice probability of product $1_{\omega}$ is given by $\frac{v+\omega}{1+V}=p+q$. Similarly, the MNL choice probability of product $2_{\omega}$ is given by $\frac{v-\omega}{1+V}=p-q$. For any other product $i \in \{3, \ldots, k \}$, its choice probability is $\frac{v_i}{1+V}=p_i$. Therefore, according to the closed-form expression of the maximum load in Equation~\eqref{eq:multinomial}, 
\[ f(q)=\sum_{\mathbf{x}\in\Delta_T}h(\mathbf{x},T)  \cdot (p+q)^{x_1}\cdot(p-q)^{x_2}\cdot \left(\prod_{i=3}^k p_i^{x_i}\right)\cdot p_0^{T-\sum_{i=1}^kx_i}\cdot   \max_{i=1,\ldots,k} x_i, \]
where  $h(\mathbf x,T)$ refers to multinomial coefficient and $\Delta_T$ is the support set of $\mathbf x$, i.e.,
		$$
		h(\mathbf x,T)\coloneqq \binom{T}{ x_1,\ldots,x_k,T-\sum_{i=1}^k x_i} \qquad \text{ and } \qquad 
		\Delta_T\coloneqq\left\{{{\mathbf x}\in\N^k\,\bigg\vert\,\sum_{i=1}^k x_i\leq T }\right\} .
		$$
		Since $f$ is a polynomial function of $q$, it is differentiable with respect to $q$. Therefore, by differentiating, we obtain $\frac{d}{dq}f(q)=T \cdot (Q_1(q)-  Q_2(q))$, where
  \begin{eqnarray*}				Q_1(q)&=&\sum_{\mathbf{x}\in\Delta_T,x_1 \geq 1 }h(\mathbf{x}-\mathbf{e}_1,T-1)\cdot (p+q)^{x_1-1} \cdot (p-q)^{x_2}\cdot \left(\prod_{i=3 }^k p_i^{x_i}\right)    \cdot p_0^{T-\sum_{i=1}^kx_i}\cdot\max\limits_{i=1,\ldots,k} x_i, \\  Q_2(q)&=&\sum_{\mathbf{x}\in\Delta_T,x_2 \geq 1}  h(\mathbf{x}-\mathbf{e}_2,T-1) \cdot (p+q)^{x_1} \cdot (p-q)^{x_2-1} \cdot \left(\prod_{i=3 }^k p_i^{x_i}\right) \cdot p_0^{T-\sum_{i=1}^kx_i}\cdot\max\limits_{i=1,\ldots,k} x_i.
		\end{eqnarray*}
   By examining $Q_1(q)$, we observe that it corresponds to the expected maximum load when we offer the assortment $S_{\omega}$, conditional on customer $T$ selecting product $1_{\omega} \in S_{\omega}$. Similarly, $Q_2(q)$ corresponds to the expected maximum load when we offer the assortment  $S_{\omega}$, conditional on  customer $T$ selecting  product $2_{\omega}\in S_{\omega}$. In other words,
$$
		Q_1(q)= \E\left(M( S_{\omega})\,\vert \,X_{1_{\omega},T}(S_{\omega})=1\right) \qquad \text{and} \qquad
		Q_2(q)=\E\left(M(S_{\omega})\,\vert \,X_{2_{\omega},T}(S_{\omega})=1\right),
$$		
where $\{X_{1_{\omega}T}(S_{\omega})=1\}$ is the event in which customer $T$ selects product $1_{\omega}$, and similarly $ \{ X_{2_{\omega}T}(S_{\omega})=1\}$ corresponds to customer $T$ selecting product $2_{\omega}$. 
  
In order to prove that $f$ is monotonically non-decreasing, it suffices to show that $Q_1(q)\geq Q_2(q)$. Let us define $Q(q)$ as the expected maximum load  when we offer the assortment $S_{\omega}$, conditional on customer $T$ selecting the no-purchase option, i.e.,
$$ Q(q)=   \E\left(M( S_{\omega})\,\vert \,X_{0,T}(S_{\omega})=1\right) . $$
It is sufficient to show that $Q_1(q)-Q(q)\geq Q_2(q)-Q(q)$. By examining the difference $\{ M( S_{\omega})\,\vert \,X_{1_{\omega},T}(S_{\omega})=1 \}-  \{ M( S_{\omega})\,\vert \,X_{0,T}(S_{\omega})=1 \}$ in the same probability space, we observe that this difference is $1$ if product $1_{\omega}$ has the highest load after $T-1$ customers; otherwise, the difference is $0$.  Therefore, $Q_1(q)-Q(q)$ is exactly the probability that product $1_{\omega}$ has the highest load  given $T-1$ customers. Similarly, $Q_2(q)-Q(q)$ is exactly the probability that product $2_{\omega}$ has the highest load  given $T-1$ customers. However, the choice probability of product $1_{\omega}$ is $p+q$, which is greater than the choice probability of product $2_{\omega}$, given by $p-q$. Hence, a straightforward coupling argument on the customers $1, \ldots, T-1$ implies that $Q_1(q)-Q(q)\geq Q_2(q)-Q(q)$.

\subsection{Proof of Lemma \ref{lem:probs}}\label{apx:probs}
    
Let us start by defining an intermediate Multinomial vector $\mathbf Z$ with parameters $(T,p_0^Z,\ldots,p_m^Z)$ where $p_i^Z = \min(p_i^Y,p_i^W)$ for all $i\in\{1,\ldots,m\}$, and $p_0^Z = 1-\sum_{i=1}^m p_i^Z$. By this definition, $p^W_i\geq p^Z_i$ for all $i\in\{1,\ldots,m\}$, and we can therefore easily couple $\mathbf W$ and $\mathbf Z$ such that $W_i\geq Z_i$ for all $i\in \{1,\ldots,m\}$, which implies that
\begin{equation}\label{eq:couple1}
\E\left(\max_{i=1,\ldots,m}W_i\right) \geq \E\left(\max_{i=1,\ldots,m}Z_i\right).
\end{equation}

Next, we introduce a coupling between $\mathbf Y$ and $\mathbf Z$. For every $i\in\{1,\ldots,m\}$ and for every $t\in[T]$, let $B_{i,t}$ be a Bernoulli random variable, with success probability $p^Z_i/p^Y_i$. These Bernoulli random variables are independent. Given the random variable $\mathbf Y$, we construct a new random vector $\mathbf{\widetilde Z}=(\widetilde Z_0,\ldots,\widetilde Z_m)$ as follows. If \blue{the outcome of a trial $t\in [T]$ is $0$} for $\mathbf{Y}$, then \blue{its outcome is also $0$} for $\mathbf{\widetilde Z}$. Otherwise, if the \blue{outcome of the trial is} some $i\in\{1,\ldots,m\}$ for $\mathbf Y$, we distinguish between two cases: When $B_{i,t} = 1$, the \blue{the outcome of the trial is} $i$ for $\mathbf{\widetilde Z}$; when $B_{i,t} = 0$, \blue{the outcome of this trial is} $0$. The first key idea to notice is that $\mathbf{\widetilde Z}$ is equal in distribution to $\mathbf{Z}$. Indeed, \blue{the outcome of a trial $t$ is $0$} for $\mathbf{\widetilde Z}$ if one of the following happens: (i)~\blue{Its outcome is} $0$ for $\mathbf Y$; or (ii)~\blue{Its outcome is some} option $i\in\{1,\ldots,m\}$ for $\mathbf Y$ and $B_{i,t}=0$. The probability of one of the two events happening is given by $p^Y_0+\sum_{i=1}^m p^Y_i\cdot (1-p_i^Z/p_i^Y)=p_0^Z.$ Similarly, \blue{the outcome of a trial $t$ is} some option $i\in\{1,\ldots,m\}$ in $\mathbf{\widetilde Z}$ if both of the following happen: (i) \blue{The outcome of trial $t$ is $i$} for $\mathbf{Y}$; and (ii) $B_{i,t}=1$. The probability of both these events happening is given by $p^Y_i\cdot p_i^Z/p_i^Y = p_i^Z$. 
        
Now, letting $\ayy$ be the random index of the highest-load option among $Y_1,Y_2,\ldots,Y_m$, i.e., $\ayy = \argmax_{i=1,\ldots,m}Y_i$, breaking ties by taking the smallest index, we have
\[ \E\left( \left. \max_{i=1,\ldots,m} \widetilde Z_i\,\right|\, \mathbf Y\right) \geq \E\left( \left. \widetilde Z_{\ayy}\,\right|\, \mathbf Y\right)            =\frac{p^Z_{\ayy}}{p^Y_{\ayy}} \cdot Y_{\ayy}. \]
Here, the latter equality follows from the construction of $\mathbf{\widetilde Z}$, since
\blue{if some trial's outcome is $I$ for $\mathbf{Y}$, then its outcome is $I$ for $\widetilde{ \mathbf{Z}}$ with probability ${p^Z_{\ayy}/p^Y_{\ayy}}$}. By the lemma's assumption,  ${p^Z_{\ayy}/p^Y_{\ayy}}\geq 1-\eps$, and therefore
\begin{equation*}
\E\left( \left. \max_{i=1,\ldots,m} \widetilde Z_i\,\right|\, \mathbf Y\right)\geq (1-\eps) \cdot Y_{\ayy} = (1-\eps)\cdot \E\left( \left. \max_{i=1,\ldots,m} Y_i\,\right|\, \mathbf Y\right).
\end{equation*}
Now by taking expectations in the previous inequality with respect to $\mathbf Y$ and applying the law of total expectation, we get 
        \begin{equation*}
            \E\left(\max_{i=1,\ldots,m} \widetilde Z_i\right)\geq (1-\eps)\cdot \E\left(\max_{i=1,\ldots,m} Y_i\right).
        \end{equation*}
        Since $\mathbf{\widetilde Z}$ is equal in distribution to $\mathbf Z$,                 \begin{equation}\label{eq:couple2}
            \E\left(\max_{i=1,\ldots,m} Z_i\right)\geq (1-\eps)\cdot \E\left(\max_{i=1,\ldots,m} Y_i\right).
        \end{equation}
        Finally, combining Equations \eqref{eq:couple1} and \eqref{eq:couple2}  gives the desired result, i.e., \begin{equation*}
            \E\left(\max_{i=1,\ldots,m}W_i\right)\geq (1-\eps)\cdot \E\left(\max_{i=1,\ldots,m} Y_i\right).
        \end{equation*}

\subsection{Proof of Lemma \ref{lem:weights}}\label{apx:weights}

To establish the desired claim, we will apply Lemma~\ref{lem:probs}. Using the notation of the latter claim, let ${\bf Y}=(Y_0,Y_1,\ldots,Y_m)$ and ${\bf W}=(W_0,W_1,\ldots,W_m)$ be Multinomial vectors, with parameters $(T,p^Y_0,\ldots,p^Y_m)$ and  $(T,p^W_0,\ldots,p^W_m)$, respectively, where
 for all $i\in\{1,\ldots,m\}$:  $$p_i^Y= \frac{ v^+_i }{ 1+\sum_{j=1}^m v^+_j} \qquad \text{and} \qquad p_i^W= \frac{ v^-_i }{ 1+\sum_{j=1}^m v^-_j }.$$ 
We also define ${p_0^Y = 1-\sum_{j=1}^mp^Y_j}$ and ${p_0^W = 1-\sum_{j=1}^mp^W_j}$. Therefore,
$\E(M(S^-))= \E(\max_{i=1,\ldots,m}W_i)$ and $ \E(\max_{i=1,\ldots,m} Y_i)= \E(M(S^+)\blue{)}$.
Moreover, for all $i\in\{1,\ldots,m\}$, we have

$$
             p_i^W = \frac{v^-_i}{1+\sum_{j=1}^m v^-_j}\\
                  \geq  \frac{v^-_i}{1+\sum_{j=1}^m v^+_j}\\
                  \geq (1-\eps) \cdot \frac{v^+_i}{1+\sum_{j=1}^m v^+_j} \\
                  =  (1-\eps) \cdot p_i^Y,
$$
where the first inequality is a consequence of the condition $v^-_i\leq v^+_i$, and the second inequality holds since $v^-_i\geq (1-\eps)\cdot v^+_i$. Therefore, applying Lemma~\ref{lem:probs} yields the desired result.

\subsection{Proof of Lemma \ref{lem:virtual}}\label{apx:virtual}

In what follows, we define a virtual assortment as a couple $(S,k)$, where $S\subseteq \Nc$  and $k$ is a virtual product with weight $v_k \leq \min_{i \in S} v_i$. In addition, we define the \fil operation as one that takes as input a virtual assortment $(S,k)$ and applies the following steps. First, if $S$ is preference-weight-ordered, then \fil simply returns $(S,k)$. Otherwise, let $h$ be the heaviest  product in $\Nc \setminus S$, i.e., $h=\argmax_{i \in \Nc\setminus S} v_i$, where $\argmax$ breaks ties by selecting the product with lowest index. In addition, let ${\cal T} = \{i \in S : i > h \}$, to which we refer as the collection of {\em tail products}. Since $S$ is not preference-weight-ordered, $ {\cal T} \neq \emptyset$. Note that $\{1,\ldots, h-1\}$ is the largest preference-weight-ordered assortment included in $S$ and that $S=\{1,\ldots, h-1\} \cup {\cal T}$. The \fil  operation proceeds by considering two cases.

\paragraph{Case 1: ${v_k+\sum_{i\in {\cal T}}v_i \leq v_h}$.} Here, the total weight of the tail products plus the virtual product $k$ is at most the weight of product $h$. In this case, we remove the products in ${\cal T}$ from $S$ and let  $\widetilde S$ be the resulting assortment, i.e., $\widetilde S = S\setminus \{i \in S : i >h \}$. Then, we merge the tail products along with the virtual product $k$ into a single virtual product, denoted by   $\widetilde k$, whose weight is given by $v_{\widetilde k} = v_k+\sum_{i\in {\cal T}}v_i$. \fil returns the virtual assortment $(\widetilde S,\widetilde k)$.
Clearly, the assortment $\widetilde S$ is preference-weight-ordered, since all tail products were removed. In addition, by the case hypothesis, ${v_{\widetilde k}\leq v_h \leq \min_{i\in \widetilde S}v_i}$.

\paragraph{Case 2: ${v_k+\sum_{i\in {\cal T}}v_i > v_h}$.} In this case, we will use a subset of $\cal T$ and the virtual product $k$ to create a replica of the missing product $h$. Formally, suppose that ${\cal T}=\{p_1,\ldots,p_m\}$, where  without loss of generality $v_{p_1}\leq \cdots \leq v_{p_m}$. Recall that $v_k$ is upper-bounded by the weight of every product in $S$, meaning in particular that the latter is upper-bounded by the weight of every tail product, and in turn that $v_k \leq v_h$. On the other hand, we have $v_k+\sum_{i\in \cal T}v_i > v_h$ by the case hypothesis.  Let $j$ be the unique index for which ${v_k+\sum_{i=1}^{j-1}v_{p_i} \leq v_h < v_k+\sum_{i=1}^{j}v_{p_i}}$.  The \fil operation starts by merging the products $p_1,\ldots,p_{j-1}$ and $k$, creating a virtual product $\widehat k$ with weight $v_{\widehat k} = v_k+ \sum_{i=1}^{j-1}v_{p_i}$. We proceed by considering two cases:
\begin{itemize}
    \item {\em When $v_{\widehat k}>v_{p_j}$:} We perform a $\delta$-weight transfer from product $p_j$ to the virtual product $\widehat k$, with $\delta = v_h-v_{\widehat k}$. This transfer is well defined since $v_{p_j} \geq  \delta \geq 0$, by definition of $j$. We have therefore created a replica of product $h$, as well as a virtual product $\widetilde k$ with weight $v_{\widetilde k} = v_{p_j}-\delta$. Finally, the \fil operation returns the virtual assortment $(\widetilde S, \widetilde k)$, where $\widetilde S = (S\cup \{h\})\setminus \{p_1,\ldots, p_j\}$. It is important to note that the virtual product $\widetilde k$ satisfies $v_{\widetilde k} \leq v_{p_j}\leq \min_{i\in \widetilde S}v_i$, and therefore $(\widetilde S, \widetilde k)$ is a virtual assortment.
    
    \item {\em When $v_{\widehat k} \leq v_{p_j}$:} We perform a $\delta$-weight transfer from the virtual product $\widehat k$ to product $p_j$ with $\delta = v_h-v_{p_j}$. This transfer is well defined since $  v_{\widehat k} \geq  \delta \geq 0$, where the first inequality follows from the definition of $j$ and the second inequality holds since $p_j$ is a tail product, and thus, lighter than $h$. We have therefore created a replica of product $h$, as well as a virtual product $\widetilde k$ with preference weight $v_{\widetilde k} = v_{\widehat k} -\delta$. The \fil operation returns the virtual assortment $(\widetilde S, \widetilde k)$, where $\widetilde S = (S\cup \{h\})\setminus \{p_1,\ldots,p_j\}$. Again, the virtual product $\widetilde k$ satisfies $v_{\widetilde k}\leq v_{\widehat k}\leq v_{p_j} \leq \min_{i\in \widetilde S}v_i$, meaning that $(\widetilde S, \widetilde k)$ is a virtual assortment.
\end{itemize}

Given these definitions, with respect to any assortment $S \subseteq \Nc$, we apply the \fil operation to the virtual assortment $(S,k)$, where initially $v_k=0$. If $1\notin S$, then the condition $v(S)\geq v_1$ guarantees that product $1$ is included in the resulting assortment. Otherwise, this product is already in the resulting assortment. We then repeat this operation until it returns a virtual assortment $(\widetilde{S}, \widetilde{k})$, where $\widetilde{S}$ is preference-weight-ordered. Such an assortment will eventually be obtained since, at each step, if $\widetilde{S}$ is not preference-weight-ordered, the \fil operation increases the size of the largest preference-weight-ordered assortment included in $\widetilde{S}$ by at least one product, as discussed in Case 2. 
Finally, since the \fil operation is a composition of a sequence of Merge and Transfer operations, as stated in Lemmas~\ref{lem:merge} and~\ref{lem:charging}, we know that the expected maximum load of the resulting assortment $\widetilde{S} \cup \{\widetilde{k}\}$ is lower-bounded by that of the initial assortment $S$.
	
\subsection{Proof of Lemma \ref{lem:drop}}\label{apx:drop}

Since $v_k\leq \min_{i\in S}v_i$, we can perform a weight transfer from product $k$ to any product in $S$ without decreasing the objective function. In the proof of this lemma, we start by successively performing a $\delta$-weight transfer from the virtual product $k$ to each of the products in $S$, with $\delta = v_k/|S|$. Eventually, the weight of product $k$ becomes $0$, whereas the weight of any product in $S$ increases by $\delta$. We therefore obtain an assortment $ S^+ = \{ i^+\,|\, i\in S\}$, where each product $i^+$ has weight $v_{i^+} = v_i+ v_k/|S|$. Moreover, since these weight transfers cannot decrease the expected maximum load to Lemma \ref{lem:charging}, we have
\begin{equation}\label{eq:firsteq}
\E\left(M\left(S^+\right)\right)\geq\E\left(M\left(S\cup\{k\}\right)\right).
\end{equation}
We proceed to show that the objective values of $\widetilde S$ and $S$ are within $|S|/(|S|+1)$ from each other, using Lemma \ref{lem:weights}. Indeed, for all $i\in S$, we have $v_{i^+} = v_i+\delta \geq v_i$. Moreover, we have $$
v_{i^+} = v_i+\frac{v_k}{|S|} \leq v_i+\frac{v_i}{|S|} = \frac{|S|+1}{|S|}\cdot v_i,$$
where the inequality above holds since $v_k\leq \min_{j\in S}v_j$. As a result, $v_i\geq \frac{|S|}{|S|+1}\cdot v_{i^+}$, and according to Lemma~\ref{lem:weights}, we have
\[ \E\left(M\left(S\right)\right)\geq \frac{|S|}{|S|+1}\cdot \E\left(M\left(S^+\right)\right) \geq \frac{|S|}{|S|+1}\cdot \E\left(M\left(S\cup \{k\}\right)\right), \]
where the last inequality follows from~\eqref{eq:firsteq}.

\subsection{Proof of Lemma \ref{lem:thmstep1}}\label{apx:thmstep1}

The proof of this lemma is similar in spirit to that of Lemma \ref{lem:virtual}. We remind the reader that, as defined in Section \ref{subsec:PTAS}, the $\eps$-hole of an assortment $S$ is given by $h_{\eps}(S) = \min\{j : j \geq \ost_{1/\eps}(S) \text{  and  } j\notin S\}$, where $I_{1/\eps}(S)$ is the $1/\eps$\blue{-th} heaviest product in $S$. Also, a virtual assortment was defined in Appendix \ref{apx:virtual} as a pair $(S,k)$, where $S\subseteq \Nc$  and $k$ is a virtual product  whose weight satisfies $v_k\leq \min_{i\in S} v_i$.
In addition, we define the \fiil operation as one that takes as input a virtual assortment  $(S,k)$ with  $|S|\geq 1/ \eps$, and returns a pair $(\widetilde S, \widetilde k)$ where $\widetilde S\subseteq \Nc$ and $\widetilde k$ is a virtual product. We proceed to explain how the latter operation is performed. 

Consider a virtual assortment $(S,k)$ with $|S|\geq 1/ \eps$.
Let $h$ be the $\eps$-hole of $S$, and let $S_0$ be the subset of $S$, consisting of all products whose weights are smaller than $\eps\cdot v_h$, i.e., $S_0 = \{i\in S\,\mid\, v_i < \eps\cdot v_h\}$. When $S$ is $\eps$-restricted, \fiil simply outputs $(S,k)$. Otherwise, we consider the next two cases:

\paragraph{Case 1: $S$ is not $\eps$-restricted and $v_k+\sum_{i\in S_0}v_i<v_{h}$.} In this case, we merge product $k$ and all products in $S_0$, creating a new virtual product $\widetilde k$ whose weight is $v_{\widetilde k} = v_k+\sum_{i\in S_0}v_i$. The \fiil operation then outputs the pair $(\widetilde S,\widetilde k)$, where $\widetilde S = S\setminus S_0$.
This pair satisfies two important properties: First, since the $\eps$-hole of $S$ is identical to that of $\widetilde S$ and we removed $S_0$, then $\widetilde S$ is $\eps$-restricted. Second, since  $v_k+\sum_{i\in S_0}v_i<v_{h}$,  product $\widetilde k$ is lighter than the $\eps$-hole of $\widetilde S$, i.e., $v_{\widetilde k}\leq v_{h}=v_{h_\eps(\widetilde S)}.$

\paragraph{Case 2: $S$ is not $\eps$-restricted and ${v_k+\sum_{i\in S_0}v_i\geq v_{h}}$.} Since our input $(S,k)$ is a virtual assortment, we have $v_k\leq \min_{i\in S}v_i$. Also, since $S$ is not $\eps$-restricted, $S_0\neq \emptyset$. Therefore, $\min_{i\in S}v_i < \eps\cdot v_h\leq v_h$, implying that $v_k\leq v_h$. In what follows, we employ the Merge and Transfer operations to create a replica of the $\eps$-hole of $S$. Suppose that $S_0=\{p_1,\ldots,p_m\}$, with $v_{p_1}\leq\cdots\leq v_{p_m}$, and let $j$ be the unique index for which ${v_k+\sum_{i=1}^{j-1}v_{p_i}\leq v_h< v_k+\sum_{i=1}^{j}v_{p_i}}$, which is well defined since $v_k\leq v_h$ and $v_k+\sum_{i\in S_0}v_i\geq v_h$. Then, the \fiil operation starts by merging products $1,\ldots,j-1$ and $k$, thereby creating a virtual product $\widehat k$ with weight $v_{\widehat k} = v_k+\sum_{i=1}^{j-1}v_{p_i}$. We then have two possibilities:
\begin{enumerate}
    \item {\em When $v_{\widehat k}>v_{p_j}$}: We perform a $\delta$-weight transfer from product $p_j$ to the virtual product $\widehat k$, with $\delta = v_h-v_{\widehat k}$. Since $v_h\geq v_k+\sum_{i=1}^{j-1}v_{p_i}$, we know that $\delta\geq 0$. Also, by the inequality $v_h< v_k+\sum_{i=1}^j v_{p_i}$, it is easy to see that $\delta \leq v_{p_j}$. We have therefore created a copy of product $h$, as well as a virtual product $\widetilde k$ with preference weight $v_{\widetilde k} = v_{p_j}-\delta$. Finally, the \fiil operation returns the pair $(\widetilde S, \widetilde k)$, where $\widetilde S = (S\cup \{h\})\setminus\{p_1,\ldots,p_j\}$. Note that the virtual product $\widetilde k$ satisfies $v_{\widetilde k} \leq v_{p_j}\leq \min_{i\in \widetilde S} v_i$, meaning that $(\widetilde S, \widetilde k)$ is a virtual assortment.
    
    \item {\em When $v_{\widehat k}\leq v_{p_j}$}: We perform a $\delta$-weight transfer from the virtual product $\widehat k$ to product $p_j$, with $\delta = v_h-v_{p_j}$. Since $v_{p_j} < \eps\cdot v_h$ by definition of $S_0$, we know that $\delta\geq 0$. Also, by the inequality $v_h<v_k+\sum_{i=1}^{j-1}v_{p_i}$, it is easy to see that $\delta\leq v_{\widehat k}$.  We have therefore created a copy of the $\eps$-hole $h$, as well as a virtual product $\widetilde k$ with preference weight $v_{\widetilde k} = v_{\widehat k}-\delta$. Finally, the \fiil operation returns the pair $(\widetilde S, \widetilde k)$, where $\widetilde S = (S\cup \{h\})\setminus\{p_1,\ldots,p_j\}$. Note that the virtual product $\widetilde k$ satisfies $v_{\widetilde k} \leq v_{\widehat k} \leq v_{p_j}\leq \min_{i\in \widetilde S} v_i$, implying  that $(\widetilde S, \widetilde k)$ is a virtual assortment.
    \end{enumerate}

To complete the proof, consider an assortment $S$ with $|S|\geq 1/\eps$. We apply the \fiil operation to the virtual assortment $(S,k)$, where $k$ is a virtual product with weight $0$. If $S$ is $\eps$-restricted, \fiil simply outputs $(\widetilde S, \widetilde k)=(S,k)$, which satisfies the conditions of our lemma. Otherwise,  if  we are in Case 1, then \fiil outputs a pair $(\widetilde S, \widetilde k)$ such that $\widetilde S$ is $\eps$-restricted and $v_{\widetilde k}\leq v_{h_{\eps}(\widetilde S)}$, again satisfying the required conditions. Finally, if we are in Case 2, then \fiil outputs a virtual assortment $(\widetilde S,\widetilde k).$ Here, \fiil will be reapplied to the pair $(\widetilde S,\widetilde k)$, so on and so forth, as long as we encounter Case 2. The main observation is that, in each such iteration, the $\eps$-hole of the assortment $S$ is included in $\widetilde S$, and is never removed in subsequent steps. Therefore, there are at most $n$ iterations of  Case 2.

Finally, since the \fiil operation is a composition of a sequence of Merge and Transfer operations, as stated in Lemma \ref{lem:merge} and Lemma \ref{lem:charging}, we know that the expected maximum load of the resulting assortment $\widetilde{S} \cup \{\widetilde{k}\}$ is lower-bounded by that of the initial assortment $S$.

\subsection{Proof of Lemma \ref{lem:thmstep2}}\label{apx:thmstep2}

Let $\widehat S\subseteq \Nc$ be an assortment with $|\widehat S|> 1/\eps$ and let $k$ be a virtual product with preference weight $v_k \leq v_{h_{\eps}(\widehat S)}$. Recall that, for an assortment $S$, its $\eps$-hole $h_\eps(S)$ is given by $h_{\eps}(S) = \min\{j : j \geq \ost_{1/\eps}(S) \text{  and  } j\notin S\}$, where $I_{1/\eps}(S)$ is the $1/\eps$ heaviest product in $S$. Therefore, product $k$ is lighter than each of the $1/\eps$\blue{-th} heaviest products in $S$. Let the subset of these $1/\eps$ products be  $\{p_1,\ldots,p_{1/\eps}\}$. 

We proceed by performing a weight transfer from product $k$ to each of the products $p_1,\ldots,p_{1/\eps}$. Specifically, for every $j=1,\ldots,1/\eps$, we successively perform a $\delta$-weight transfer from product $k$ to product $p_j$, with $\delta = \eps\cdot v_k$. At the end of these transfers, the virtual product $k$ is removed, whereas each product $p_j$ was replaced by a virtual product $p_j^{\eps}$ whose weight is $v_{p_j^{\eps}} = v_{p_j}+ \eps\cdot v_k$. Let $\widehat S_\eps$ be the resulting assortment, i.e., ${\widehat S_\eps = (\widehat S \cup \{p_1^\eps,\ldots,p_{1/\eps}^\eps\} )\setminus \{p_1,\ldots,p_{1/\eps}\}}$. According to Lemma \ref{lem:charging}, the transfer operation cannot decrease the expected maximum load, and therefore
\begin{equation} \label{eq:morocco1}
    \E\left(M\left(\widehat S_\eps\right)\right) \geq \E\left(M\left(\widehat S\cup \{k\}\right)\right).
\end{equation}
Moreover, for any $j=1,\ldots,1/\eps$, since $v_k\leq v_{p_j}$ and $v_{p_j^{\eps} }= v_{p_j}+ \eps\cdot v_k$, we get $ v_{p_j^{\eps}} \geq v_{p_j}\geq (1-\eps)\cdot v_{p_j^{\eps}}$.


Note that currently $\widehat S_{\eps}$ may not be block-based, as it contains virtual products. Hence, for each product in $\widehat S_\eps$, we define its counterpart in a block-based assortment as follows: 
\begin{itemize}
    \item For each product $p_j^\eps$, its counterpart is simply the product $p_j$. Let us denote this subset of products by $S_1=\{p_1,\ldots,p_{1/\eps}\}$.
    
    \item For every product $i$ with  $p_{1/\eps} < i < h_{\eps}(\widehat S)$, its counterpart is the product itself. We refer to this subset of products as $S_2$.

    \item Every remaining product is contained in one of the classes $C_1,\ldots, C_L$ that were introduced in Section \ref{subsec:PTAS} to define block-based assortments. Indeed, since $\widehat S$ is $\eps$-restricted, the weight of any product in $\widehat S$ is at least $\eps\cdot v_{h_\eps(\widehat S)}$. Therefore, each product of $\widehat S_\eps$ that is lighter than product $h_\eps({\widehat S})$ is contained in one of the classes $C_1,\ldots, C_L$, since their union contains every product whose weight resides within $[\epsilon \cdot v_{h_{\epsilon}(\widehat S)} ,v_{h_{\epsilon}(\widehat S)}) $. For every class $C_i$, let $\widehat R_i$ be the subset of $\widehat S_\eps$ comprised of the products in the class $C_i$, and let $N_i$ be the number of these products. In other words, $\widehat R_i = \widehat S_{\eps}\cap C_i$ and $N_i=|\widehat R_i|$. Then the counterpart of the products of $\widehat R_i$ are the $N_i$ lightest products of $C_i$. We denote the latter subset by $R_i$. It is important to note that, by definition  $C_1,\ldots,C_L$, the weights of any two products in the same class are within $1-\eps$ of each other. Let $S_3=\bigcup_{i=1}^L R_i$.
\end{itemize}
To summarize, for each product in $\widehat S_\eps$, we have defined a counterpart whose weight is within factor $1-\eps$. Therefore, letting $\widetilde S = S_1\cup S_2\cup S_3$ be our resulting assortment, by Lemma \ref{lem:weights}, we get
\begin{equation} \label{eq:morocco2}
    \E\left(M\left(\widetilde S\right)\right) \geq (1-\eps)\cdot  \E\left(M\left(\widehat S_\eps\right)\right) \geq (1-\eps)\cdot \E\left(M\left(\widehat S\cup \{k\}\right)\right),
\end{equation}
where the last inequality follows from \eqref{eq:morocco1}. By construction, $\widetilde S$ is a block-based assortment.

\subsection{Proof of Lemma \ref{lem:largeT}}\label{apx:largeT}
We show that for every assortment $S$, there exists a threshold value $T_S$ such that $\E(M(S)) \leq \E(M(\{1\}))$ for all $T\geq T_S$. Since the number of possible assortments is finite, taking $\bar T = \max_{S\subseteq \Nc}T_S$ suffices to conclude the proof. To this end, let us fix an assortment $S\subseteq \Nc$. First, if $|S| = 1$, then the claim is trivial since the maximum load when offering $S$ is a binomial random variable whose success probability is at most $v_1/(1+v_1)$, and we can take $T_S = 1$. In the remainder of this proof, we assume that $|S| \geq 2$. First, letting $j$ be the heaviest product in $S$, notice that
\begin{equation}\label{eq:largeT}
        \max_{i\in S}\phi_i(S) = \frac{v_j}{1+V(S)} = \frac{v_j}{1+v_j+V(S\setminus \{j\})} \leq \frac{v_j}{1+v_j+v_n} \leq \frac{v_1}{1+v_1+v_n}.
    \end{equation}
Here, the first inequality holds since the weight of any non-empty assortment is trivially lower bounded by $v_n$; note that $S\setminus \{j\}$ is non-empty when $|S| \geq 2$. The second inequality follows from the monotonicity of $x\mapsto x/(1+x+v_n)$.
    We denote the right-hand side of Equation \eqref{eq:largeT} by $q= v_1/(1+v_1+v_n)$, and let $p = v_1 / (1+v_1)$, which is the choice probability of product $1$ when offering the assortment $\{1\}$. Notice that $p>q$, since $v_n>0$, as stipulated in Section \ref{sec:form}. For simplicity of notation, let $A$ be the event ``there exists a product $i\in S$ with $L_i(S)\geq T(p+q)/{2}$", where we recall that $L_i(S)$ is the random load of product $i$ when offering the assortment $S$. Let $\Bar{A}$ be the complementary event of $A$.
    Then,\begin{align*}
        \E(M(S)) &= \E\left(M(S) | A\right)\cdot \P\left(A\right) +  \E\left(M(S) | \Bar{A}\right)\cdot \P\left(\Bar{A}\right)\\
        &\leq T\cdot \P(A) + T\cdot \frac{p+q}{2}.
    \end{align*}
    Let $\alpha =  \frac{3(p-q)^2}{4(5q+p)}$. We conclude the proof by establishing the following claim:
    \begin{claim}\label{cl:largeT}
        If $T\geq \frac{1}{\alpha}\log\frac{2n}{p-q}$, then $\P(A) \leq {(p-q)}/{2}$.
    \end{claim}
    As a result, by taking $T_S = \lceil\frac{1}{\alpha}\log\frac{p+q}{2n}\rceil$, we have $
        \E(M(S)) \leq Tp = \E(M(\{1\}))$ if $T\geq T_S$.
    
    {\bf Proof of Claim \ref{cl:largeT}.}
    Using a union bound we have $
        \P(A)\leq \sum_{i\in S}\P(L_i \geq T\cdot\frac{p+q}{2}).
    $
        Let $Z$ be a binomial random variable with $T$ trials and success probability $q$. Noticing that each $L_i$ is a binomial random variable with $T$ trials and a success probability of at most $q$, as shown in Equation~\eqref{eq:largeT}, we have 
            ${\P(Z \geq T \cdot \frac{p+q}{2}) \geq \P(L_i \geq T \cdot \frac{p+q}{2})}.$
        Therefore,
        \begin{align*}
            \P(A) &\leq n\P\left(Z \geq T\cdot\frac{p+q}{2}\right)\\
            &\leq n\exp\left(-\frac{(\frac{p-q}{2q})^2Tq}{2+\frac{p-q}{3q}}\right)\\
            &= n\exp\left(-\frac{3(p-q)^2T}{4(5q+p)}\right)\\
            & = n\exp(-\alpha T)\\
            & \leq \frac{p-q}{2}.
        \end{align*}
    Here, the second inequality uses the Chernoff bound of \citet[Theorem~1.10.1]{doerr2020probabilistic}, and the last inequality holds since $T \geq \frac{1}{\alpha}\log\frac{2n}{p-q}$.
    \endproof

\section{Proofs from Section \ref{sec:adaptivitygap}}\label{apx:adapt_gap}




\subsection{Proof of Lemma \ref{lem:multinom}}\label{apx:multinom}

To establish the desired claim, we construct a coupling between the Multinomial vector $(L_1,\ldots,L_k)$ and the load vector of an optimal dynamic policy with respect to the universe of products $U$. In the process of generating $(L_1,\ldots,L_k)$, we view the $T$ trials as if they occur sequentially, in the order $1, \ldots, T$, letting $L_{i,t}$ be the random value of component $i$ after $t$ trials. In addition, let $L_{i,t}^{\dpp}$ be the random load of product $i$ after $t $ customers in a fixed optimal dynamic policy. By convention, $L_{i,t}=0$ and $p_i = 0$ for $i>k$, and $L_{j,t}^{\dpp}=0$ for $j \notin U$. Let us initialize both $(L_{1,0},\ldots,L_{n,0})$ and  $(L_{1,0}^{\dpp}, \ldots,L_{n,0}^{\dpp})$ to be the zero vector. 

\paragraph{Sampling the Multinomial vector.} For $t \in [T]$, we sample the component for the $t$-th trial of the Multinomial vector as follows. Let $(\varphi_{t-1}(1) ,\ldots, \varphi_{t-1}(n))$ be the permutation of $\{1,\ldots,n\}$ for which  $L_{\varphi_{t-1}(1),{t-1}} \geq\cdots\geq L_{\varphi_{t-1}(n),{t-1}}$, breaking ties by order of increasing indices. 
We partition $(0,1]$ into a collection of pairwise-disjoint intervals $\{ I_{\varphi_{t-1}(i),t} \}_{i \in [n]}$, where
\[ I_{\varphi_{t-1}(i),t} = \left(\sum_{j=1}^{i-1}p_{\varphi_{t-1}(j)}, \sum_{j=1}^{i}p_{\varphi_{t-1}(j)}\right]. \]
We now sample a uniform random variable $U_t$ in $[0,1]$, and increment component $i$ by one if and only if $U_t \in I_{i,t} $. In other terms, we have $L_{i,t} = L_{i,t-1} + \mathbbm{1}(U_t\in I_{i,t})$.

\paragraph{Sampling the load vector.} For $t \in [T]$, we generate the choice of customer $t$ with respect to the optimal dynamic policy as follows. Let $(\psi_{t-1}(1), \ldots,\psi_{t-1}(n))$ be the permutation of $\{1,\ldots,n\}$ for which $L^{\dpp}_{\psi_{t-1}(1),{t-1}} \geq\cdots\geq L^{\dpp}_{\psi_{t-1}(n),{t-1}}$, again breaking ties in order of increasing product indices. Let $S_t$ be the assortment offered by the optimal adaptive policy to customer $t$. Note that this assortment is a-priori random, as it depends on the choices of customers $1,\ldots,t-1$; however, it is deterministic, conditional on the choices of customers $1,\ldots,t-1$. Let $p_{i,t}^{\dpp}=\phi_i(S_t)$ be the MNL choice probability of product $i$ with respect to this assortment; in particular, $p_{i,t}^{\dpp}=0$ when $i\notin S_t$. As before, we define the collection of pairwise-disjoint intervals $\{ J_{\psi_{t-1}(i),t} \}_{i \in [n]}$,  where
\[ J_{\psi_{t-1}(i),t}=\left(\sum_{j=1}^{i-1}p_{\psi_{t-1}(j),t}^{\dpp}, \sum_{j=1}^{i}p_{\psi_{t-1}(j),t}^{\dpp}\right]. \]
To generate the choice of customer $t$, we make use of exactly the same uniform random variable $U_t$ that was previously sampled, when
generating the Multinomial vector. Specifically, customer $t$ selects the product $i$ for which $U_t \in J_{i,t}$. When none of the intervals $\{ J_{\psi_{t-1}(i),t} \}_{i \in [n]}$ contains $U_t$, customer $t$ selects the no-purchase option. Formally, $L^{\dpp}_{i,t} = L^{\dpp}_{i,t-1} + \mathbbm{1}(U_t\in J_{i,t})$. It is worth emphasizing again that the same random variable $U_t$ is utilized to simulate the $t$-th Multinomial trial as well as the choice of customer $t$, consequently coupling the two vectors. 

\paragraph{Analysis.} Moving forward, for $i \in [n]$ and $t\in [T]$, let $\sumstat_{i,t}$ be the cumulative sum of the $i$ highest components of the Multinomial vector after $t$ trials, i.e., $\sumstat_{i,t} = \sum_{j=1}^iL_{\varphi_t(j),t}$. Similarly, let $\sumdyn_{i,t}$ be the cumulative sum of the $i$ highest loads of the load vector after $t$ customers have made their choice, i.e., $\sumdyn_{i,t} = \sum_{j=1}^iL^{\dpp}_{\psi_t(j),t}$. In both definitions, the cumulative sums are taken at the end of step $t$. The crux of our analysis resides in establishing the next relation between these two cumulative sums. The proof of this result is provided in Appendix~\ref{apx:invariant}.

\begin{claim} \label{cl:invariant}
For every $i \in [n]$ and $t\in  [T]$, we have $\sumstat_{i,t}\geq \sumdyn_{i,t}$.
\end{claim}

We conclude the proof by arguing that the latter claim indeed implies $\E(\max(L_1,\ldots,L_k)) \geq \opt^{\dpp}(U)$. For this purpose, we almost surely have
\[ \max(L_1,\ldots,L_k) = L_{\varphi_{T}(1),T} = \sumstat_{1,T}\geq \sumdyn_{1,T} =L_{\psi_{T}(1),T}^{\dpp} = \max_{i\in U}L^{\dpp}_{i,T}, \] 
where the inequality above is obtained by instantiating Claim~\ref{cl:invariant} with $i=1$ and $t=T$. By taking expectations, we indeed get $\E(\max(L_1,\ldots,L_k))\geq\E(\max_{i\in U}L^{\dpp}_{i,T}) = \opt^{\dpp}(U)$.
		
\subsection{ Proof of Claim \ref{cl:invariant}}\label{apx:invariant}

Our proof works by induction on $t$. Indeed, the result is trivial for $t=0$ since $\sumstat_{i,t} =  \sumdyn_{i,t} = 0$ for every $i \in [n]$. Now, for $t \geq 1$, suppose by induction that \blue{for all $i \in [n]$,
\begin{equation}\label{eq:inductionhyp}
    \sumstat_{i,t}\geq \sumdyn_{i,t},
\end{equation}
}
and let us prove that $\sumstat_{i,t+1}\geq \sumdyn_{i,t+1}$ by considering two cases. 

\paragraph{Case 1: ${\sumstat_{i,t}> \sumdyn_{i,t}}$.} Here, the sum of the $i$ highest loads in the dynamic policy is strictly smaller than the sum of the $i$ highest components of the Multinomial vector at step $t$. As only one customer arrives at each time step, both sums can increase by at most $1$. Therefore, we clearly have $\sumstat_{i,t+1}\geq \sumdyn_{i,t+1}$.

\paragraph{Case 2: ${\sumstat_{i,t}= \sumdyn_{i,t}}$.} In this case, the sum of the $i$ highest loads in the dynamic policy is equal to the sum of the $i$ highest components of the Multinomial vector. First, if  customer $t+1$ chooses the no-purchase option, then the invariant is trivially maintained, since we would have $\sumdyn_{i,t+1} = \sumdyn_{i,t}$ and $\sumstat_{i,t+1}\geq \sumstat_{i,t}$. Otherwise, suppose that $U_{t+1}\in I_{\varphi_{t}(a),t+1}$ and $U_{t+1}\in J_{\psi_{t}(b), t+1}$, for some $a$ and $b$, i.e., the $a$-th highest component of the Multinomial vector after $t$ trials is assigned to the $(t+1)$-th trial, and the $b$-th most loaded product after $t$ customers is selected by customer $t+1$. 

We first observe that $a\leq b$, advising the reader to consult Figure~\ref{fig:multinomcoupling} to better understand our next explanation. The key idea is to exploit the inequality ${\max_{i\in U} \frac{ v_i }{ 1+v_i }\leq \min_{i=1,\ldots,k}p_i}$, stating that any choice probability in the optimal dynamic policy is upper-bounded by the selection probability of each of the $k$ components in our Multinomial vector. As demonstrated in Figure~\ref{fig:multinomcoupling}, the length of each interval  \blue{$J_{j,t}$} is upper-bounded by the length of each interval \blue{$I_{i,t}$} for all $i,j\in\{1,\ldots,n\}$. Therefore, when the $(t+1)$-th trial is assigned to the $a$-th highest component, and the $(t+1)$-th customer selects the $b$-th most loaded product, we must have $a\leq b$. Our proof proceeds by considering two subcases, depending on the relation between $i$ and $b$.

\begin{figure}[htbp!]
    \centering
    \tikzset{every picture/.style={line width=0.75pt}} 

\begin{tikzpicture}[x=0.75pt,y=0.75pt,yscale=-1,xscale=0.9]

\draw   (14.47,42.79) -- (100.06,42.79) -- (100.06,131.09) -- (14.47,131.09) -- cycle ;
\draw   (100.06,61.44) -- (205.62,61.44) -- (205.62,131.09) -- (100.06,131.09) -- cycle ;
\draw   (205.62,61.44) -- (292.64,61.44) -- (292.64,131.09) -- (205.62,131.09) -- cycle ;
\draw   (292.64,93.94) -- (419.59,93.94) -- (419.59,131.09) -- (292.64,131.09) -- cycle ;
\draw   (419.59,93.98) -- (505.18,93.98) -- (505.18,131.07) -- (419.59,131.07) -- cycle ;
\draw   (14.47,183.83) -- (78.66,183.83) -- (78.66,253.48) -- (14.47,253.48) -- cycle ;
\draw   (78.66,183.83) -- (155.7,183.83) -- (155.7,253.48) -- (78.66,253.48) -- cycle ;
\draw   (155.7,201.44) -- (209.9,201.44) -- (209.9,253.48) -- (155.7,253.48) -- cycle ;
\draw   (209.97,218.85) -- (246.99,218.85) -- (246.99,253.48) -- (209.97,253.48) -- cycle ;
\draw   (246.99,218.85) -- (326.87,218.85) -- (326.87,253.48) -- (246.99,253.48) -- cycle ;
\draw   (326.87,235.1) -- (399.62,235.1) -- (399.62,253.67) -- (326.87,253.67) -- cycle ;
\draw    (590.77,131.07) -- (677.79,130.6) ;
\draw   (399.62,235.1) -- (433.86,235.1) -- (433.86,253.67) -- (399.62,253.67) -- cycle ;
\draw   (505.18,113.21) -- (590.77,113.21) -- (590.77,131.07) -- (505.18,131.07) -- cycle ;
\draw [draw opacity=1 ]   (433.86,253.5) -- (677.79,252.34) ;
\draw    (207.74,142.97) -- (290.64,142.97) ;
\draw [shift={(292.64,142.97)}, rotate = 180] [][line width=0.75]    (10.93,-3.29) .. controls (6.95,-1.4) and (3.31,-0.3) .. (0,0) .. controls (3.31,0.3) and (6.95,1.4) .. (10.93,3.29)   ;

\draw    (13.47,18.79) -- (14.47,253.48) ;
\draw    (677.79,130.6) -- (677.79,252.51) ;
\draw    (156.41,266.37) -- (207.9,266.33) ;
\draw [shift={(209.9,266.33)}, rotate = 179.96] [][line width=0.75]    (10.93,-3.29) .. controls (6.95,-1.4) and (3.31,-0.3) .. (0,0) .. controls (3.31,0.3) and (6.95,1.4) .. (10.93,3.29)   ;

\draw    (259.67,143) -- (207.52,143) ;
\draw [shift={(205.52,143)}, rotate = 360] [][line width=0.75]    (10.93,-3.29) .. controls (6.95,-1.4) and (3.31,-0.3) .. (0,0) .. controls (3.31,0.3) and (6.95,1.4) .. (10.93,3.29)   ;

\draw    (209.69,266) -- (157.54,266) ;
\draw [shift={(155.54,266)}, rotate = 360] [][line width=0.75]    (10.93,-3.29) .. controls (6.95,-1.4) and (3.31,-0.3) .. (0,0) .. controls (3.31,0.3) and (6.95,1.4) .. (10.93,3.29)   ;

\draw (34.1,18.4) node [anchor=north west][inner sep=0.75pt]    {\footnotesize $L_{\varphi_t(1),t}$};
\draw (129.66,37.37) node [anchor=north west][inner sep=0.75pt]    {\footnotesize$L_{\varphi_t(2),t}$};
\draw (10.57,264.79) node [anchor=north west][inner sep=0.75pt]    {\footnotesize${\textstyle 0}$};
\draw (672.46,264.09) node [anchor=north west][inner sep=0.75pt]    {$1$};
\draw (220.55,147.7) node [anchor=north west][inner sep=0.75pt]    {\footnotesize$I_{\varphi_t(3),t+1}$};
\draw (229.12,37.4) node [anchor=north west][inner sep=0.75pt]    {\footnotesize$L_{\varphi_t(3),t}$};
\draw (330.5,72.37) node [anchor=north west][inner sep=0.75pt]    {\footnotesize$L_{\varphi_t(4),t}$};
\draw (434.74,72.37) node [anchor=north west][inner sep=0.75pt]    {\footnotesize$L_{\varphi_t(5),t}$};
\draw (525.11,92.37) node [anchor=north west][inner sep=0.75pt]    {\footnotesize$L_{\varphi_t(6),t}$};
\draw (28.03,153.2) node [anchor=north west][inner sep=0.75pt]    {\footnotesize$L^{\dpp}_{\psi_t(1),t}$};
\draw (95.56,153.2) node [anchor=north west][inner sep=0.75pt]    {\footnotesize$L^{\dpp}_{\psi_t(2),t}$};
\draw (161.05,171.17) node [anchor=north west][inner sep=0.75pt]    {\footnotesize$L^{\dpp}_{\psi_t(3),t}$};
\draw (210.39,188.17) node [anchor=north west][inner sep=0.75pt]    {\footnotesize$L^{\dpp}_{\psi_t(4),t}$};
\draw (262.82,189.17) node [anchor=north west][inner sep=0.75pt]    {\footnotesize$L^{\dpp}_{\psi_t(5),t}$};
\draw (338.38,204.17) node [anchor=north west][inner sep=0.75pt]    {\footnotesize$L^{\dpp}_{\psi_t(6),t}$};
\draw (391.79,204.17) node [anchor=north west][inner sep=0.75pt]    {\footnotesize$L^{\dpp}_{\psi_t(7),t}$};
\draw (609.46,108.4) node [anchor=north west][inner sep=0.75pt]    {\footnotesize$L_{\varphi_t(7),t}$};
\draw (513.48,235.65) node [anchor=north west][inner sep=0.75pt]  [font=\small] [align=left] {no-purchase};

\draw (157.03,270.7) node [anchor=north west][inner sep=0.75pt]    {\footnotesize$J_{\psi_t(3),t+1}$};

\end{tikzpicture}
    \caption{\small Each of the bottom rectangles corresponds to a product $\boldsymbol j$ in $\boldsymbol U$. Its height corresponds to the load of this product $\boldsymbol j$, and its length corresponds to the associated interval $\boldsymbol{J_{\psi_t(j),t+1}}$. Similarly, each of the top rectangles corresponds to a component $\boldsymbol i$ in the Multinomial vector, its height is the number of trials assigned to this component $\boldsymbol i$, and its length corresponds to the associated interval $\boldsymbol{I_{\varphi_t(i),t+1}}$ .  The left-to-right order of these rectangles is by decreasing order of height.}
    \label{fig:multinomcoupling}
\end{figure}

\paragraph{Case 2a: ${b\leq i}$.} Here, the $b$-th highest load in the load vector was increased by $1$, and since $a\leq b\leq i$, the sum of the $i$ highest components in the Multinomial vector was increased by $1$ as well, i.e., $\sumstat_{i,t+1} = \sumstat_{i,t}+1$. Therefore,
\[ \sumstat_{i,t+1} = \sumstat_{i,t}+1 \geq \sumdyn_{i,t}+1 \geq \sumdyn_{i,t} \ , \]
where the first inequality follows from the induction hypothesis.

\paragraph{Case 2b: ${b>i}$.} In particular, according to the definition of $\psi_{t}$, we have, $L_{\psi_{t}(b),t}^{\dpp}\leq L_{\psi_{t}(i),t}^{\dpp}$. We consider two cases, depending on whether the latter inequality is strict or not:
\begin{itemize}
         \item When $L_{\psi_{t}(b),t}^{\dpp}<L_{\psi_{t}(i),t}^{\dpp}$: Then, when customer $t+1$ selects product $\psi_{t}(b)$, the load of that product is increased by exactly $1$, and therefore does not surpass $L_{\psi_{t}(i),t}^{\dpp}$. As such, the sum of the $i$ highest loads remains unchanged at step $t+1$. 
         Consequently, 
         $$
        \sumstat_{i,t+1}\geq \sumstat_{i,t} \geq \sumdyn_{i,t}=\sumdyn_{i,t+1},
     $$
     where the second inequality holds by the induction hypothesis. 
     
    \item When $L_{\psi_{t}(b),t}^{\dpp} = L_{\psi_{t}(i),t}^{\dpp}$: Here, the load of the $b$-th and $i$-th most loaded products are equal. It follows that all  products in between have the same loads, i.e., \begin{equation}\label{eq:equality}
            L_{\psi_{t}(i),t}^{\dpp}=L_{\psi_{t}(i+1),t}^{\dpp}=\cdots=L_{\psi_{t}(b),t}^{\dpp}.
    \end{equation}
    Therefore, when customer $t+1$ selects product $\psi_t(b)$, product $\psi_t(b)$ becomes more loaded than $\psi_t(i), \psi_t(i+1),\ldots,\psi_t(b-1)$. Consequently, the sum of the $i$ highest loads increases by $1$, i.e.,$$
        \sumdyn_{i,t+1}=\sum_{j=1}^i L_{\psi_{t+1}(j),t+1}^{\dpp}=\sum_{j=1}^i L_{\psi_{t}(j),t}^{\dpp}+1 = \sumdyn_{i,t}+1.
    $$
    It remains to show that the sum of the $i$ highest components of the Multinomial vector also increases by $1$. First, if $a\leq i$, this claim  is trivial, as the $a$-th highest component is increased by $1$, and therefore the sum of the $i$ highest components is also increased by $1$. Now suppose that $a>i$. We prove in the next paragraph that, for all $c\in \{i,i+1\ldots,b\}$, we have    \begin{equation}\label{eq:claim}
        \blue{L_{\psi_{t}(c),t}^{\dpp} = L_{\varphi_{t}(c),t}}.
    \end{equation}
    As a result, according to Equation \eqref{eq:equality}, 
    $$
    \blue{L_{\varphi_{t}(i),t} = L_{\varphi_{t}(i+1),t}=\cdots =L_{\varphi_{t}(b),t}}.
    $$
    Therefore, since $i<a\leq b$, when the $a$-th highest component is increased by $1$, it becomes strictly greater than each of the components $\varphi_t(i),\varphi_t(i+1),\ldots,\varphi_t(a-1)$, meaning that the sum of the $i$-th highest components also increases by $1$.
\end{itemize}

    \paragraph{\blue{Proof of Equation \eqref{eq:claim}.}} \blue{
    We prove this result by induction on $c$.
    
    \noindent {Base case}: For $c=i$, we have on the one hand $\sumstat_{i-1,t} \geq \sumdyn_{i-1,t}$, according to Equation \eqref{eq:inductionhyp}. On the other hand, we have $\sumstat_{i,t} = \sumdyn_{i,t}$ by the case hypothesis. Therefore, $\sumstat_{i,t}-\sumstat_{i-1,t} \leq \sumdyn_{i,t} - \sumdyn_{i-1,t}$, and it follows that $L_{\psi_{t}(i),t}^{\dpp} \geq L_{\varphi_{t}(i),t}$. Assume by contradiction that $L_{\psi_{t}(i),t}^{\dpp} > L_{\varphi_{t}(i),t}$, then
    $
        {L_{\psi_{t}(i+1),t}^{\dpp} = L_{\psi_{t}(i),t}^{\dpp} > L_{\varphi_{t}(i),t} \geq L_{\varphi_{t}(i+1),t}}
    $, where the equality follows from Equation \eqref{eq:equality}, and the second inequality follows from the definition of $\varphi_t$.
    Therefore, recalling that $\sumstat_{i,t} = \sumdyn_{i,t}$, we have
    $$
        \sumstat_{i+1,t}=\sum_{j=1}^{i+1} L_{\varphi_{t}(j),t} <  \sum_{j=1}^{i+1} L_{\psi_{t}(j),t}^{\dpp} = \sumdyn_{i+1,t},
    $$
    which contradicts Equation \eqref{eq:inductionhyp}. Therefore $L_{\psi_{t}(i),t}^{\dpp} = L_{\varphi_{t}(i),t}$, which concludes the case $c=i$.
    
    \noindent {Inductive step}: Let $c\in \{i,\ldots, b-1\}$, and assume by induction that $L_{\psi_{t}(d),t}^{\dpp} = L_{\varphi_{t}(d),t}$ for all $d\in \{i,\ldots,c\}$. First, since $\sumstat_{i,t} = \sumdyn_{i,t}$, by directly applying the induction hypothesis, we have \begin{equation}\label{eq:inductionstep1}
        \sumstat_{c, t} = \sumstat_{i,t} + \sum_{j=i+1}^c L_{\varphi_t(j),t} = \sumdyn_{i,t} + \sum_{j=i+1}^c L^{\dpp}_{\psi_t(j),t} = \sumdyn_{c,t}.
    \end{equation}
    In addition, $L^{\dpp}_{\psi_t(c+1),t}= L^{\dpp}_{\psi_t(c),t} = L_{\varphi_t(c),t} \geq L_{\varphi_t(c+1),t}$, 
    where the first equality follows from Equation \eqref{eq:equality}, and the second follows from the induction hypothesis, while the inequality follows from the definition of $\varphi_t$. Now, assume by contradiction that  $L^{\dpp}_{\psi_t(c+1),t} > L_{\varphi_t(c+1),t}$, then using Equation \eqref{eq:inductionstep1}, we have
    $$
        \sumstat_{c+1, t} = \sumstat_{c, t} + L_{\varphi_t(c+1),t} < \sumdyn_{c,t} + L^{\dpp}_{\psi_t(c+1),t} = \sumdyn_{c+1,t},
    $$
    which contradicts Equation \eqref{eq:inductionhyp}. Therefore, $L^{\dpp}_{\psi_t(c+1),t} = L_{\varphi_t(c+1),t}$, which concludes the induction.
    }

\subsection{Proof of Lemma \ref{lem:augment}}\label{apx:augment}

Let $A$ and $B$ be two adaptive policies for \ref{DMLA}. We will make use of $1^A,\ldots,T^A$ and $1^B,\ldots,T^B$ to denote the sequences of customers that will encounter these policies, respectively. We couple their choices as follows. 

\paragraph{Sampling the policies.} For any stage $t\in[T]$, we explain how to sample the choices of customers $t^A$ and $t^B$. First, the choice of customer $t^A$ is sampled according to policy $A$. In particular,  if $S^A_t$ is the assortment offered to this customer, each product $i\in U \cup \{0\}$ is chosen with the MNL probability $\phi_i(S^A_t)$. The choice of customer $t^B$ is sampled in a coupled manner:
\begin{enumerate}
    \item When customer $t^A$ selects the no-purchase option: In this case, the choice of customer $t^B$ is sampled using an MNL choice model, with respect to the assortment $S_t^B\setminus S_t^A$, i.e., each product $i$ is selected with probability $\phi_i(S_t^B\setminus S_t^A)$. 

    \item When customer $t^A$ selects some product $i\in S^A_t$: Here, with probability $\phi_i(S^B_t)/\phi_i(S^A_t)$, customer $t^B$ is assigned to product $i$. With probability $1-\phi_i(S^B_t)/\phi_i(S^A_t)$, as in item~1, the choice of customer $t^B$ is sampled using an MNL choice model with respect to the assortment $S_t^B\setminus S_t^A$, i.e., each product $j$ is selected with probability $\phi_j(S_t^B\setminus S_t^A)$. As a side note, we indeed have $\phi_i(S^B_t)/\phi_i(S^A_t) \leq 1$, since $S^A_t \subseteq S^B_t$.
\end{enumerate}

\paragraph{Equivalence with offering ${S_t^B}$.}
We will show that sampling the choice of customer $t^B$ as described above is equivalent to using an MNL choice model with respect to the assortment $S^B_t$. Indeed, for every $i\in S_t^A$, the probability that this product is selected by customer $t^B$ is given by\begin{equation}\label{eq:prob}
    \phi_i(S_t^A)\cdot     
\frac{\phi_i(S^B_t)}{\phi_i(S^A_t)}  = \phi_i(S^B_t).
\end{equation}
Now, for a product $i\in S_t^B\setminus S_t^A$ to be selected, we first have to be choosing using the MNL model with respect to the assortment $S_t^B\setminus S_t^A$, which happens with probability 
$$  \phi_0(S_t^A) + \sum_{j \in S_t^A }  \phi_j(S_t^A)  \left( 1 -   
\frac{\phi_j(S^B_t)}{\phi_j(S^A_t)} \right)
 =1-\sum_{j\in S_t^A}\phi_j(S_t^B)= 1- \frac{v(S_t^A)}{1+v(S_t^B)} = \frac{1+v(S_t^B) - v(S_t^A)}{1+v(S_t^B)}.$$
Conditional on this event, customer $t^B$  selects  product $i\in S_t^B\setminus S_t^A$ with probability $\phi_i(S_t^B\setminus S_t^A)$, and we indeed get
$$
    \frac{1+v(S_t^B) - v(S_t^A)}{1+v(S_t^B)} \cdot \phi_i(S_t^B\setminus S_t^A) = \frac{v_i}{1+v(S_t^B)} = \phi_i(S^B
_t).$$

\paragraph{Concluding the proof of Lemma \ref{lem:augment}.} For each product $i\in U$, let $L_i^A$ and $L_i^B$ be the (random) loads of product $i$ with respect to the policies $A$ and $B$. Based on the above-mentioned coupling, we will establish the next auxiliary claim, whose proof is deferred to the end of this section. 

\begin{claim}\label{cl:conditional}
$\E(L_i^B\,\mid\, L_i^A)\geq \frac{1}{1+\eps}\cdot L_i^A$, for every product $i\in U$,.
\end{claim}

To derive Lemma~\ref{lem:augment}, let $I$ be the random product where the maximum load of policy $A$ is attained. In the case of ties, the product with smallest index is selected, i.e., $I = \min\{i\,\mid\, L_i^A = \max_{j\in U}L_j^A\}$. Then, by Claim~\ref{cl:conditional},
\[ \E\left(L^B_I\,\left|\,(L_i^A)_{i\in U} \right. \right)\geq \frac{1}{1+\eps}\cdot L_I^A = \frac1{1+\epsilon}\cdot \E\left(L^A_I\,\left|\,(L_i^A)_{i\in U} \right. \right). \]
Therefore,
$$
\E\left( \left. \max_{i\in U}L^B_i\,\right|\,(L_i^A)_{i\in U}\right)\geq \frac{1}{1+\eps}\cdot \E\left(L^A_I\,\left|\,(L_i^A)_{i\in U} \right. \right) = \frac{1}{1+\eps}\cdot\E\left( \left. \max_{i\in U}L^A_i\,\right|\,(L_i^A)_{i\in U}\right).
$$
The desired result now follows by introducing the expectation over $(L_i^A)_{i\in U}$ and using the tower property.


\paragraph{Proof of Claim  \ref{cl:conditional}.} We will show that $\E(L_i^B\,\mid\, L_i^A=\ell)\geq \frac{\ell}{1+\eps}$ for all $\ell\in\{0,\ldots,T\}$. To this end, it suffices to prove that, for each customer $t^A$ of the $\ell$ customers who selected product $i$ in the arrival sequence of policy $A$, its corresponding customer $t^B$ in the sequence of policy $B$ selects this product with probability at least $1/(1+\eps)$. In other words, $\P(t^B\text{ selects }i\,\mid\,t^A\text{ selects }i) \geq \frac{1}{1+\eps}$.

Suppose that customer $t^A$ selects product $i$. In particular, $i\in S_t^A$, and we therefore have
\begin{eqnarray*}
\P(t^B\text{ selects }i\,\mid\,t^A\text{ selects }i) & = & \frac{\P(t^B\text{ selects }i\text{ and }t^A\text{ selects }i)}{\P(t^A\text{ selects }i)} \\
& = & 
\frac{\P(t^B\text{ selects }i)}{\P(t^A\text{ selects }i)} \\
& = & \frac{\phi_{i}(S_t^B)}{\phi_{i}(S_t^A)} \ ,
\end{eqnarray*}
where the second equality holds since, by case~2 of our coupling, given that customer $t^B$ was assigned to product $i\in S_t^A$, we know that customer $t^A$ selected this product as well. We conclude the proof by noting that 
$$
    \phi_i(S^B_t)=\frac{v_i}{1+v(S^B_t)}\geq \frac{v_i}{ 1+\eps + v(S^A_t)} \geq \frac1{1+\epsilon}\cdot\phi_i(S^A_t),
$$
where the first inequality is obtained by recalling that $v(S_t^B)-v(S_t^A)\leq \eps$.

\subsection{Proof of Lemma \ref{lem:subadd}} \label{apx:subadd}


Letting $U_0=U_1\cup U_2$, we assume to have three sequences of $T$ customers each. Specifically, we denote the first sequence by ${\cal T}_0$, consists of the customers $1^{(0)},\ldots, T^{(0)}$. Similarly, we refer to the second and third sequences by ${\cal T}_1$ and ${\cal T}_2$, with customers $1^{(1)},\ldots, T^{(1)}$ and $1^{(2)},\ldots,T^{(2)}$, respectively. On one hand, the sequence ${\cal T}_0$ will encounter an optimal dynamic policy for the universe of products $U_0$. On the other hand, ${\cal T}_1$ and ${\cal T}_2$ will respectively encounter dynamic policies for $U_1$ and $U_2$. The latter two policies will not necessarily be optimal. In the following, we first start by describing our policies for ${\cal T}_1$ and ${\cal T}_2$. Second, we sample the choices of these sequences in a coupled fashion. 

\paragraph{Describing the policies.} Let ${\cal P}_0$ be an optimal dynamic policy for $U_0$. In this proof, we use the notation $S^{{\cal P}_0}_t$ to denote the random assortment offered by the policy ${\cal P}_0$ to customer $t^{(0)}$. As such, we define the policy ${\cal P}_1$, offering the assortment $S^{{\cal P}_1}_t = S^{{\cal P}_0}_t\cap U_1$ to its $t$-th customer. By definition, this policy only offers products from the universe $U_1$. We denote the expected maximum load of this policy by ${\cal E}^1$. Similarly, we define ${\cal P}_2$ as the policy that offers the assortment $S^{{\cal P}_2}_t = S^{{\cal P}_0}_t\cap U_2$ to its $t$-th customer, noting that only products from $U_2$ are offered. We denote the expected maximum load of this policy by ${\cal E}^2$. 

\paragraph{Sampling customer choices.}
Let $t\in \{1,\ldots,T\}$. First, we sample the choice of customer $t^{(0)}$ using an MNL choice model with respect to the assortment $S^{{\cal P}_0}_t$. Let us show how to sample the choices of customers $t^{(1)}$ and $t^{(2)}$ in a coupled fashion. For customer $t^{(1)}$, if customer $t^{(0)}$ selected some product $i\in S^{{\cal P}_1}_t\cup\{0\}$, then customer $t^{(1)}$ is also assigned to product $i$. Otherwise, $i\in S^{{\cal P}_0}_t\setminus S^{{\cal P}_1}_t$, and the choice of customer $t^{(1)}$ is decided according to an MNL choice model with respect to $S^{{\cal P}_1}_t$. Similarly, for customer $t^{(2)}$, if $t^{(0)}$ selected some product $i\in S^{{\cal P}_2}_t\cup\{0\}$, then $t^{(2)}$ is also assigned to product $i$. Otherwise, the choice of customer $t^{(2)}$ is decided according to a choice model with respect to $S^{{\cal P}_2}_t$. 

Next, we show that sampling the choices of customers $t^{(1)}$ and $t^{(2)}$ as described above is equivalent to simply offering the assortment $S^{{\cal P}_1}_t$ and $S^{{\cal P}_2}_t$, respectively. We explain why this is true for $t^{(1)}$, noting that the argument for for $t^{(2)}$ is symmetrical. First, customer $t^{(1)}$ can only select a product in $S^{{\cal P}_1}_t\cup \{0\}$. For every $i\in S^{{\cal P}_1}_t$, product $i$ is selected by customer $t^{(1)}$ if and only if one of the next two disjoint events occurs:
\begin{itemize}
    \item Product $i$ is selected by customer $t^{(0)}$. This happens with probability $v_i/(1+v(S^{{\cal P}_0}_t))$.  
    
    \item Customer $t^{(0)}$ selected some product in $S^{{\cal P}_0}_t\setminus S^{{\cal P}_1}_t$ and then product $i$ was selected by the MNL choice model when $S^{{\cal P}_1}_t$ was offered. This  happens with probability
    $$
        \frac{v(S^{{\cal P}_0}_t)-v(S^{{\cal P}_1}_t)}{1+v(S^{{\cal P}_0}_t)} \cdot 
        \frac{v_i}{1+v(S^{{\cal P}_1}_t)}.
    $$
\end{itemize}
Therefore, the overall probability that customer $t^{(1)}$ selects product $i$ is given by
\[     \frac{v_i}{1+v(S^{{\cal P}_0}_t)}+\frac{v(S^{{\cal P}_0}_t)-v(S^{{\cal P}_1}_t)}{1+v(S^{{\cal P}_0}_t)}\cdot\frac{v_i}{1+v(S^{{\cal P}_1}_t)}
 = \frac{v_i}{1+v(S^{{\cal P}_1}_t)} = \phi_i(S^{{\cal P}_1}_t). \]

\paragraph{Concluding the proof.} Let $(L_i^0\,\mid\,i\in U_0)$ be the load vector attained by applying the policy ${\cal P}_0$ for the arrival sequence ${\cal T}_0$.  Similarly,  $(L_i^1\,\mid\,i\in U_1)$ and $(L_i^2\,\mid\,i\in U_2)$ will be the load vectors corresponding to ${\cal P}_1$ and ${\cal P}_2$, applied for ${\cal T}_1$ and ${\cal T}_2$, respectively. The key observation is that, for every $t\in[T]$, when customer $t^{(0)}$ selects some product $i\in U_1$, then customer $t^{(1)}$ also selects this product. Therefore, for every $i\in U_1$, we have $L^0_i\leq L^1_i$. By analogy, for every $i\in U_2$, we have $L^0_i\leq L^2_i$. As a result, 
\begin{eqnarray}
    \max_{i\in U_0} L^0_i &= & \max\left(\max_{i\in U_1} L_i^0, \max_{i\in U_2}L_i^0\right)\nonumber\\
    &\leq & \max\left(\max_{i\in U_1} L_i^1, \max_{i\in U_2}L_i^2\right)\nonumber\\
    &\leq & \max_{i\in U_1} L_i^1+ \max_{i\in U_2}L_i^2.\label{eq:lasteq}
\end{eqnarray}
Therefore,
\begin{eqnarray*}
    \opt^{\dpp}(U_0) &= &\E\left(\max_{i\in U_0} L^0_i\right) \\
    &\leq& \E\left(\max_{i\in U_1} L_i^1\right) + \E\left(\max_{i\in U_2} L_i^2\right) \\
    &=& {\cal E}^1+{\cal E}^2 \\
    &\leq& \opt^{\dpp}(U_1)+\opt^{\dpp}(U_2). 
\end{eqnarray*}
Here, the first equality follows by recalling that  $(L^0_i\,\mid\,i\in U_0)$ is the load vector of an optimal dynamic policy for $U_0$. The next inequality is a consequence of Equation \eqref{eq:lasteq}.
The following equality follows from the definition of ${\cal E}^1$ and ${\cal E}^2$. The last inequality is obtained by noting that ${\cal P}_1$ and ${\cal P}_2$ are feasible dynamic policies with respect to the universes $U_1$ and $U_2$, respectively.

\subsection{Proof of Theorem \ref{thm:eqv}}\label{apx:eqv}

In what follows, we consider \ref{DMLA} instances where all  products have the same preference weight, which will be denoted by $v$. Our approach proceeds by distinguishing between three cases, depending on the magnitude of this parameter.

\paragraph{Case 1: ${v\geq 1}$.} In this case, we statically offer the same single product at each time step. The choice probability of this product is $v/(1+v)$, and its load is a Binomial random variable with $T$ trials and success probability $v/(1+v)$. This yields an expected maximum load of $Tv/(1+v)$. Therefore,
\[ \opt^{\ord}(\Nc)\geq T\cdot\frac{v}{1+v}\geq \frac T2\geq\frac{\opt^{\dpp}(\Nc)}2, \]
where the second inequality holds since $v\geq 1$, and the third inequality follows by noting that
the maximum load is always upper bounded by the total number of  customers $T$. 

\paragraph{{Case 2:} ${v<1/n}$.} Here, we  statically offer all products in  the universe $\Nc$ to each customer, arguing that this policy guarantees a $1/2$-approximation.  Using the notation of Lemma~\ref{lem:augment}, let $A$ be an optimal adaptive policy for the universe $\Nc$, and let $B$ be the static policy that offers the whole universe of products. In this case, the condition $S_t^A\subseteq S_t^B$ is trivially satisfied since $S_t^B=\Nc$. Moreover, $v(\Nc) < 1$ since $v < 1/n$, meaning that $v(S_t^B\setminus S_t^A) < 1$, for all customers $t \in [T]$. Therefore, by employing Lemma~\ref{lem:augment} with $\eps=1$, we have 
\[ \opt^{\ord}(\Nc)\geq  \mathbb E (M(\Nc))       \geq\frac{\opt^{\dpp}(\Nc)}2. 
 \]


\paragraph{{Case 3:}  ${1/n\leq v<1}$.}
In this case, let $ 2 \leq k\leq n$ be the unique integer for which $\frac1k\leq v< \frac{1}{k-1}$. In order to prove that $\opt^{\ord}(\Nc)\geq \frac{1}{2} \cdot \opt^{\dpp}(\Nc)$, we argue that
a $1/2$-approximation can be attained by statically offering the same set of $k$ products to all customers, say $S = \{ 1, \ldots, k \}$. To analyze the exact guarantee of this policy, let $ \widetilde S =  \{\Tilde 1,\ldots, \Tilde k\}$ be a collection of $k$ virtual products, where each product $\Tilde{i}\in \widetilde S$ has a preference weight of $1/k$. Also, let $(\widehat L_1, \ldots, \widehat L_k)$ be a Multinomial vector with $T$ trials and probabilities $1/k$ for each outcome, with $\widehat M = \max_{i=1,\ldots,k}\widehat L_i$. Our analysis is based on proving the next three claims.

\begin{lemma}\label{lem:eqv1}
$\E(M(S)) \geq \E(M(\widetilde S))$.
\end{lemma}

\begin{lemma}\label{lem:eqv2}
$\E(M(\widetilde S)) \geq \E(\widehat M)/2$.
\end{lemma}

\begin{lemma}\label{lem:eqv3}
$\E(\widehat M) \geq \opt^\dpp(\Nc)$.
\end{lemma}

Consequently, by combining Lemmas~\ref{lem:eqv1}-\ref{lem:eqv3}, it follows that
\[ \opt^{\ord}(\Nc)\geq         \E(M(S)) \geq \E(M(\widetilde S))\geq \frac{ \E(\widehat M) }{ 2 } \geq \frac{ \opt^\dpp(\Nc) }{ 2 }. \]

\paragraph{Proof of Lemma~\ref{lem:eqv1}.} When statically offering $S$, the choice probability of each product is $p = v/(1+kv)$. Similarly, when statically offering, $\widetilde S$, the choice probability of each product is $\Tilde p = 1/2k$. The key idea is to notice that
\[ p = \frac{v}{1+kv} =  \frac{1}{1/v + k}\geq \frac{1}{2k} = \Tilde p, \]
where inequality above holds since $v\geq 1/k$. Hence, the choice probability of each product when offering $S$ is at least that of any product when offering $\widetilde S$. It is therefore easy to construct a coupling where $M(S) \geq M(\widetilde S)$, implying that $\E(M(S)) \geq \E(M(\widetilde S))$. 

\paragraph{Proof of Lemma~\ref{lem:eqv2}.} Let us first notice that, when statically offering $\widetilde S$, the choice probability of every product is $1/(2k)$, whereas the probability of every component of the Multinomial vector $(\widehat L_1,\ldots,\widehat L_k)$ is $1/k$. To couple between $M(\widetilde S)$ and $\widehat M$, for each customer $t\in [T]$, we select a component $I_t\in \{1,\ldots,k\}$ uniformly at random; the $t$-th trial of the Multinomial vector $(\widehat L_1, \ldots,\widehat L_k)$ is then assigned to component $I_t$. In order to simulate the selection of customer $t$ with respect to the load vector $\mathbf L(\widetilde S)$, we sample a Bernoulli random variable $Z_t$ with success probability $1/2$. If $Z_t=1$, then customer $t$ is assigned to product $I_t$. Otherwise, this customer is assigned to the no-purchase option. It is easy to verify that the constructed load vector is equal in distribution to $\mathbf{L}(\widetilde S)$. 

To complete the proof, let $I$ be the random variable specifying the index of the maximum component of the Multinomial vector $(\widehat L_1, \ldots,\widehat L_k)$. In the case of ties, we take the one with lowest index, i.e., $I=\min\{i : \widehat L_i = \widehat M\}$. Conditioning on the outcome of $(\widehat L_1,\ldots,\widehat L_k)$, we have
\begin{eqnarray*}
\E(M(\widetilde S)\,\mid\,\widehat L_1,\ldots,\widehat L_k) &\geq &\E(L_I(\widetilde S)\,\mid\,\widehat L_1,\ldots,\widehat L_k)\\
&=&\frac12\cdot\E(\widehat L_I\,\mid\,\widehat L_1,\ldots,\widehat L_k)\\
&=&\frac{1}{2}\cdot\E(\widehat M\,\mid\,\widehat L_1,\ldots,\widehat L_k).    
\end{eqnarray*}
Here, the first inequality comes from the fact that $M(\widetilde S)$ is by definition greater than or equal to all the loads of the vector $\mathbf{L}(\widetilde S)$. Finally, by taking expectations on both sides of this inequality and applying the tower property, it follows that $\E(M(\widetilde S)) \geq \E(\widehat M)/2$.

\paragraph{Proof of Lemma~\ref{lem:eqv3}.} We argue that this claim is a direct implication of Lemma \ref{lem:multinom}. Let us show that the conditions of the latter lemma are met. First, the probability of every component of the Multinomial vector $(\widehat L_1,\ldots,\widehat L_k)$ is exactly $1/k$. Second, since all products have the same preference weight, $v$, the right hand side of $\min_{i=1,\ldots,k}p_i\geq \max_{i\in U} \frac{v_i}{1+v_i}$ is simply $v/(1+v)$. In addition, since $v < \frac{ 1 }{k-1}$, we indeed have $\frac{ 1 }{ k } \geq \frac{ v }{ 1 + v }$. Therefore, by applying Lemma \ref{lem:multinom}, we conclude that $\E(\widehat M) \geq \opt^\dpp(\Nc)$.

\section{Numerical Study of the Adaptivity Gap}\label{apx:justif}
In this section, we present numerical experiments to study the effect  of  several model primitives on the adaptivity gap. Interestingly, this examination will  motivate our lower bound construction  for  the best possible adaptivity gap (see Section \ref{subsec:lowbd}). First, recall that the latter measure is defined as the worst case ratio between the expected maximum load of an optimal dynamic policy and that of an optimal static policy, over all possible instances, i.e., \begin{equation*}
    \max_{I\in {\cal I}} \frac{\opt^{\dpp}_I}{\opt^{\stat}_I}.
\end{equation*}
As such, proving that some constant $C>1$ forms an upper bound on the adaptivity gap requires as to show that $\opt^{\dpp}_I/\opt^{\stat}_I \leq C$ for all instances $I\in \cal I$.
A result of this nature is given by Theorem \ref{thm:adaptivitygap}, which shows that an optimal dynamic policy cannot exceed an optimal static policy by a factor greater than $C=4$, for any instance. Conversely, showing that some constant $D>1$ is a lower bound requires a simpler condition, namely, proving the existence of a single instance $I\in \cal I$ for which $
    \opt^{\dpp}_I/\opt^{\stat}_I \geq D. 
$
\subsection{Effect of the parameter  \texorpdfstring{$\bf T$}{T}}\label{subsec:Tregime}
\paragraph{Experimental setup.}
Here, our goal is to identify the regimes of the parameter $T$, which experimentally display the largest adaptivity gaps. 
We  vary the number of products in the range $\{2,5,10\}$. We generate the preference weights from the positive part of a normal distribution with parameters $\mu$ and $\sigma$, where $\mu$ varies in the range $\{0.01, 0.1, 0.5, 1, 5\}$ and $\sigma$ varies in $\{0, 0.1, 1\}$. The choice of the normal distribution (as opposed to an exponential distribution, for example) is meant to control both the mean and the variance of the sampled values, as we also wish to investigate the effect of variance in preference weights on the adaptivity gap in Appendix \ref{subsec:variancedaptivity}. The number of customers $T$ varies in the range $\{2,3,4,5,6,7,10,12,15, 16\}$. Finally, for each set of parameters $(n,T,\mu,\sigma)$, we generate $1000$ instances. For each instance $I$, let $A_I$ denote its corresponding adaptivity gap, i.e., the ratio between the objective values of an optimal dynamic policy and an optimal static policy. We solve each static instance through an exhaustive enumeration of all possible assortments. The optimal dynamic policy is obtained by solving the dynamic program \eqref{DMLA} with a top-down approach using memoization. 
For each problem instance, we compute the metric \begin{equation}\label{eq:aggregate}
    r_I = 100 \cdot \left(1- \frac{1}{A_I}\right),
\end{equation}
which refers to the percentage gained in the objective function when employing an optimal dynamic policy as opposed to a static one. Finally for every $T\in \{2,3,4,5,6,7,10,12,15, 16\}$, we pool together all generated instances with $T$ customers, and return several useful statistics on the metric $r_I$, namely, the mean, median, and  maximum. The statistic of most interest is the maximum, as our goal is to identify the highest possible adaptivity gap. The remaining statistics provide a more general overview of how the adaptivity gap behaves with respect to the number of customers $T$.

\paragraph{Analysis.}  Our results are summarized in Table \ref{fig:tableT}, demonstrating a clear decrease in adaptivity gap as $T$ increases. This  evidence suggests that the maximum gain in the expected maximum load when employing an optimal dynamic policy, as opposed to a static one, is attained when there are fewer customers. This observation aligns with our theoretical analysis in Lemma \ref{lem:high}, which indicates that large values of $T$ yield instances where statically offering the heaviest product is nearly optimal, which in turn signifies a smaller gain in employing a dynamic policy instead of a static one. In essence, for large values of $T$, a dynamic policy must commit early to offering products likely to attract a large number of customers, which are precisely the heavier products. Consequently, the impact of adaptivity is marginally felt, as committing to the offered products at an early stage renders each sample path of a dynamic policy similar to that of a static policy.

\begin{table}[htbp!]
        \centering
\begin{tabular}{||c|c|c|c|c|c|c|c|}
\hline
\rule{0pt}{12.5pt}{}    &  \multicolumn{3}{c|}{\begin{tabular}{l}$r_I$ (\%)\end{tabular}} & {} &\multicolumn{3}{c|}{\begin{tabular}{l}$r_I$ (\%)\end{tabular}} \\
\hline
{$T$} & {Median} & {Mean} & {Max} & {$T$} & {Median} & {Mean} & {Max} \\
\hline
  2  & 4.72 & 6.68 & 22.21 & 7  & 2.1 & 3.49 & 15.35 \\
\hline
  3  & 4.03 & 6.01 & 20.21 & 10 & 0.99 & 2.82 & 14.04 \\
\hline
  4  & 3.4 & 5.07 & 18.55 &   12 & 0.89 & 2.49 & 13.04 \\
\hline
  5  & 2.79 & 4.51 & 17.62 &  15  & 0.66 & 2.09 & 10.71 \\
\hline
  6  & 2.25 & 3.96 & 16.28 &  16 &  0.62 & 1.99 & 10.41 \\
\hline

\end{tabular}
        \caption{Comparison of the percentage gain in objective when employing an optimal dynamic policy instead of an optimal static policy.}
        \label{fig:tableT}
    \end{table}
    
\subsection{Effect of the preference weights}\label{subsec:variancedaptivity}

In what follows, we numerically analyze how preference weight values influence the adaptivity gap. Specifically, our focus is on determining the impact of variance in the preference weights. Is the adaptivity gap  higher when the preference weights are closer to each other, or when there is considerable variance among them? 
\paragraph{Experimental setup.} 
We generate our data set in a similar fashion to Appendix \ref{subsec:Tregime}. In light of that analysis, we focus on smaller values of $T$, and hence vary this parameter in the range $\{2, 3, 4\}$. We vary the parameter $\mu$ in $\{0.001, 0.01, 0.1, 0.5, 0.7, 1, 1.2, 1.5, 2, 3, 5\}$, whereas $\sigma$ varies in an evenly spaced grid of the interval $[0, \mu]$ with $20$ steps. We vary the number of products $n$ in the range $\{2,3,5,6,9,10\}$. For every value of $n$, and every triplet $(T,\mu,\sigma)$, we generate $1000$ instances using the same process described in Appendix \ref{subsec:Tregime}. For each generated instance, we compute $r_I$ as described by Equation \eqref{eq:aggregate}, and return the maximal $r_I$ obtained over all instances tested. Finally, we rank the triplets $(T,\mu,\sigma)$ for every fixed $n$ by the maximal obtained $r_I$.

\paragraph{Analysis.}
In Table \ref{fig:tablevariance}, for each value of $n$, we display the triplets $(T,\mu,\sigma)$ for which the top five highest adaptivity gaps were reached. The second column reports the ranking of the top five instances where we have obtained the highest values of $r_I$.
First, we observe that the highest values of $r_I$ in our numerical experiments are reached with a large number of products. 
Moreover, we observe from Table \ref{fig:tablevariance} that the highest adaptivity gaps are reached for lower values of the variance parameter $\sigma$. In fact, in most instances the highest values of $r_I$ were reached with $\sigma = 0$, i.e., when all preference weights are equal. Moreover, despite pooling together instances for different values of $T\in \{2,3,4\}$, the numerical evidence suggests that the highest adaptivity gaps are reached for $T=2$. 

\begin{table}[htbp!]
        \centering
\begin{tabular}{||c|c|c|c|c|c||c|c|c|c|c|c|}
\hline
{$n$} & {Rank} & {$T$} & {$\mu$} & {$\sigma$} & {$r_I(\%)$}& {$n$} & {Rank} & {$T$} & {$\mu$} & {$\sigma$} & {$r_I(\%)$}\\
\hline
   & 1 & 2 & 1.5 & 0.9 & 12.72 & & 1  & 2 & 1.2 & 0.0 & 20.59  \\
   & 2 &2 & 0.7 & 0.66 & 12.71 & & 2  & 2 & 1.3 & 0.0 & 20.46 \\
  2 & 3  & 2 & 1.5 & 0.97 & 12.7 & 6 & 3  & 2 & 1.2 & 0.18 & 20.42 \\
   & 4  & 2 & 1.5 & 0.15 & 12.69 & & 4 & 2 & 1.2 & 0.06 & 20.39  \\
   &  5  & 2 & 1.5 & 0.23 & 12.69 && 5  & 2 & 1.4 & 0.0 & 20.38  \\
\hline
   & 1  & 2 & 1.4 & 0.56 & 16.6 & & 1  & 2 & 1.2 & 0.0 & 22.08  \\
   & 2  & 2 & 1.5 & 0.07 & 16.55 & & 2  & 2 & 1 & 0.0 & 21.88  \\
  3 & 3  & 2 & 1.5 & 0.23 & 16.52 & 9 & 3  & 2 & 1.3 & 0.0 & 21.85 \\
   & 4  & 2 & 1.4 & 0.21 & 16.51 & & 4 & 2 & 1.2 & 0.06 & 21.78  \\
   &  5  & 2 & 1.4 & 0.28 & 16.51 && 5  & 2 & 1 & 0.05 & 21.75  \\
\hline
   & 1  & 2 & 1.2 & 0.0 & 19.66 & & 1  & 2 & 1.2 & 0.0 & 22.37  \\
   & 2  & 2 & 1.3 & 0.0 & 19.59 & & 2  & 2 & 1 & 0.0 & 22.21  \\
  5 & 3  & 2 & 1.5 & 0.0 & 19.58 & 10 & 3  & 2 & 1.3 & 0.0 & 22.12  \\
   & 4  & 2 & 1.4 & 0.0 & 19.57 & & 4 & 2 & 1.2 & 0.06 & 21.98  \\
   &  5  & 2 & 1.5 & 0.07 & 19.53 && 5  & 2 & 1 & 0.05 & 21.97  \\
\hline

\end{tabular}
        \caption{Comparison of the percentage gain in objective when employing an optimal dynamic policy instead of an optimal static policy.}
        \label{fig:tablevariance}
    \end{table}

In light of these observations, we focus in Section \ref{subsec:lowbd} on identifying the instance displaying the maximal adaptivity gap, by restricting our analysis to instances with $\sigma=0$, i.e., identical preference weights for all products, $T=2$, and a large number of products.

 \section{Proofs from Section \ref{sec:dynamic}}


\subsection{Proof of Lemma \ref{lem:high}}\label{apx:high}

First, when $v_{\max} \geq 1/\eps$, it is easy to verify that statically offering the heaviest product achieves an expected maximum load of at least $\frac{ v_{\max} }{ 1 + v_{\max} } \cdot T \geq (1-\eps) \cdot T$, and therefore yields the desired $(1-\eps)$-approximation. In the remainder of this proof, we consider the case where $T\alpha\geq 12\ln(nT)/\eps^3$.
Focusing on a fixed optimal dynamic policy, let $(L_1,\ldots,L_n)$ be its random load vector, and let $M = \max_{i\in \Nc}L_i$. In particular, we have $\E(M) = \opt^{\dpp}(\Nc)$. First, at every step, we observe that the choice probability of any product is at most $\alpha$, since for every assortment $S\subseteq \Nc$  and every product $i\in S$, 
\[ \phi_i(S) = \frac{v_i}{1+\sum_{j\in S}v_j}\leq \frac{v_i}{1+v_i}\leq \frac{v_{\max}}{1+v_{\max}}=\alpha. \]
Therefore, we can couple each random load $L_i$ with a Binomial random variable $Z_i\sim B(T,\alpha)$, such that $Z_i\geq L_i$ almost surely. As a result,
\begin{eqnarray*}
\P\left(L_i\geq \left(1+\frac \epsilon2\right)\cdot T\alpha\right) & \leq & \P\left(Z_i\geq\left(1+\frac\epsilon2\right)\cdot \E(Z_i)\right) \\
& \leq & \exp\left(-\frac{\epsilon^2}{12}T\alpha\right) \\
& \leq & \exp\left(-\frac{\ln(nT)}{\eps}\right) \\
& = &\left(\frac{1}{nT}\right)^{\frac{1}{\epsilon}}. 
\end{eqnarray*}
Here, the second inequality comes from the following Chernoff bound \citep[Sec.~1.10.1]{doerr2020probabilistic}, stating that when $X$ is a Binomial random variable and $\delta <1$, 
\begin{equation*}
 	\P(X\geq (1+\delta) \cdot \E(X)) \leq \exp\left(  -\frac{ \delta^2\E(X)}{3}  \right).
\end{equation*}
The third inequality follows from the case hypothesis, $T\alpha \geq \frac{12\ln(nT)}{\epsilon^3}$.
Using a union bound, we get
\begin{equation} \label{eq:imoback}
	\P\left(M\geq \left(1+\frac\epsilon 2\right) \cdot T\alpha\right)\leq \sum_{i=1}^n\P\left(L_i\geq \left(1+\frac\epsilon2\right) \cdot T\alpha\right) \leq n\cdot \left(\frac1{nT}\right)^{\frac1\eps}.
\end{equation}
By conditioning on the event $\{M\geq \left(1+\frac\epsilon 2\right) \cdot T\alpha\}$ and on its complement, we have
\begin{align*}
    \E(M) &\leq \P\left(M\geq \left(1+\frac\epsilon 2\right)T\alpha\right)\cdot T+\left(1+\frac\epsilon2\right) \cdot T\alpha\\
    & \blue{\leq n\cdot \left(\frac{1}{nT}\right)^{\frac{1}{\eps}} \cdot T + \left(1+\frac{\eps}{2}\right)\cdot T\alpha}\\
    &\leq (1+\eps)\cdot T\alpha
\end{align*}
The first inequality holds since $\E (M | M\geq \left(1+\frac\epsilon 2\right)T\alpha)$ is trivially bounded by the number of customers $T$ and since $\P (M < \left(1+\frac\epsilon 2\right)T\alpha) \leq 1$. In the second inequality, we substitute Equation \eqref{eq:imoback}. The last inequality follows Claim~\ref{cl:ineq} below. Consequently, we have just shown that $T \alpha \geq (1-\eps) \cdot \E(M)$. Thus, by statically offering the heaviest product to all customers, we secure at least a $(1-\epsilon)$-fraction of the optimal expected maximum load.

\begin{claim}\label{cl:ineq}
$(\frac1{nT})^{\frac1\eps-1}\leq \frac{\eps}{2}T\alpha$.
\end{claim}
\begin{proof}
Since $n\geq 2$ and $T\geq 2$, as argued in the beginning of Section~\ref{subsec:highr}, we have 
\[ \frac{\eps}{2}T\alpha\geq \frac{\eps}{2}\cdot\frac{12\ln(nT)}{\eps^3}\geq \frac{6\ln(4)}{\eps^2} \geq \frac{4}{\eps^2}, \] where the first inequality holds by the case hypothesis, $T\alpha\geq 12\ln(nT)/\eps^3$.
On the other hand, $(\frac1{nT})^{\frac1\eps-1}\leq (\frac1{4})^{\frac1\eps-1}$. Therefore, it suffices to show that $\frac{4}{\eps^2}\geq \left(\frac1{4}\right)^{\frac1\eps-1}$, which is equivalent to  $\frac{\eps^2}{4^{1/\eps}}\leq 1$.
This inequality holds since the function $x\mapsto \frac{x^2}{4^{1/x}}$
is nondecreasing on $(0,1]$, reaching its its maximum at $x=1$. Therefore, $\frac{\eps^2}{4^{1/\eps}}\leq 1/4\leq 1$. 
\end{proof}


\subsection{Proof of Lemma \ref{lem:low}}\label{apx:low}

By recycling the notation of Appendix~\ref{apx:high}, for a fixed optimal dynamic policy, let $(L_1,\ldots,L_n)$ be its random load vector, and let $M = \max_{i\in \Nc}L_i$. We first argue that $\P(L_i=k) \leq   \frac{1}{nT^2}$, for every product $i \in \Nc$ and for every integer $k \in [(\frac{12\ln(nT)}{\epsilon^3})^2, T]$. To this end, since $\alpha = v_{\max}/(1+v_{\max})$ is an upper bound on the choice probability of any product with respect to any assortment, we have
\begin{equation} \label{eqn:light_mass_Li}
\P(L_i=k) \leq \binom{T}{k}\alpha^k (1 - \alpha)^{T-k}
\leq \left(\frac{eT\alpha}{k}\right)^k \leq \left(\frac{e}{\sqrt{k}}\right)^k \leq   \frac{1}{nT^2}. 
\end{equation}
Here, the second and third  inequalities hold since $\binom{T}{k}\leq (eT/k)^k$ and since $k \geq (\frac{12\ln(nT)}{\epsilon^3})^2 \geq (T\alpha)^2$, by the lemma's hypothesis. The final inequality is stated as the next claim, whose proof appears at the end of this section.

\begin{claim}\label{cl:bound}
$(\frac{e}{\sqrt{k}})^k\leq \frac{1}{nT^2}$.
\end{claim}

By combining inequality~\eqref{eqn:light_mass_Li} and the union bound, \begin{equation}\label{eq:equation1000}
     \P\left(M\geq \left(\frac{12\ln(nT)}{\epsilon^3}\right)^2\right) \leq \sum_{i=1}^n \sum_{k \geq (\frac{12\ln(nT)}{\epsilon^3})^2}      \P\left(L_i = k\right) \leq \frac{ 1 }{ T }.
\end{equation}
We are now ready to derive the desired upper bound on $\opt^{\dpp}(\Nc) = \E(M)$. Specifically, by conditioning on the event $\{M\geq (\frac{12\ln(nT)}{\epsilon^3})^2\}$ and on its complement, 
\begin{eqnarray*}
		\E(M)&\leq& T\cdot \P\left(M\geq \left(\frac{12\ln(nT)}{\epsilon^3}\right)^2\right) + \left(\frac{12\ln(nT)}{\epsilon^3}\right)^2\cdot \P\left(M < \left(\frac{12\ln(nT)}{\epsilon^3}\right)^2\right)\\
		&\leq & 1 + \left(\frac{12\ln(nT)}{\epsilon^3}\right)^2\\
        & \leq & 2\cdot \left(\frac{12\ln(nT)}{\epsilon^3}\right)^2\\
        &\leq& \frac{300\ln^2(nT)}{\eps^6},
\end{eqnarray*}
where the second inequality follows from~\eqref{eq:equation1000} and third inequality holds since $T\geq 2$ and $n\geq 2$.

\paragraph{Proof of Claim \ref{cl:bound}.} 
To obtain the desired inequality, note that
\begin{eqnarray*}
    \left(\frac{e}{\sqrt{k}}\right)^k&\leq&\left(\frac{e}{12\ln(nT)}\right)^{k}\\
    &\leq & \left(\frac{e}{12\ln(4)}\right)^k\\
    &\leq & \left(\frac{1}{e}\right)^k\\
    &\leq & \left(\frac{1}{e}\right)^{144\ln^2(nT)}\\
    &\leq& \frac{1}{nT^2}.
\end{eqnarray*}
Here, the first and fourth inequalities hold since $k\geq (\frac{12\ln(nT)}{\epsilon^3})^2\geq 144\ln^2(nT)$. The second inequality is obtained by recalling that $n\geq 2$ and $T\geq 2$, as assumed without loss of generality in Section~\ref{subsec:highr}.

\subsection{Stability of policies with respect to weight alterations.}\label{apx:altering}
Here, we showcase the possibility of altering a universe of products by performing slight weight modifications, while still controlling the extent to which the expected maximum load is affected. For any universe of products $U \subseteq \Nc$, a dynamic policy that limits its offered assortments to products from  $U$ will be referred to as a $U$-policy. Let us introduce an auxiliary  universe $\widetilde U$, with a one-to-one correspondence to $U$, assuming without loss of generality that $U=\{1,\ldots,k\}$ and $\widetilde U=\{\widetilde 1,\ldots,\widetilde k\}$. 

We proceed by considering a technical condition on this pair of universes, stipulating that for every $i\in U$, the choice probabilities of the products $i$ and $\widetilde i$ are within factor $1-\eps$ of each other, with respect to any assortment. To formalize this condition, for any assortment $S\subseteq U$, we denote by $\widetilde S = \{\widetilde i  \in  \widetilde U  \,\mid\,i\in S\}$ its corresponding assortment in $\widetilde U$. Given $\delta \in [0,1)$, we say that the universes $U$ and $\widetilde U$ satisfy the $\delta$-tightness condition if, for every assortment $S\subseteq U$ and for every product $i\in S$, we have
\begin{equation}\label{eq:condbound}
\phi_{i}(S)\geq (1-\delta)\cdot\phi_{\widetilde i}(\widetilde S).
\end{equation}

In Lemma \ref{lem:alteruniverse}, we show that this condition is sufficient to prove that, for any $\widetilde U$-policy $\widetilde P$, there exists an analogous $U$-policy $P$ whose expected maximum load deviates only slightly from that of $\widetilde P$. For ease of notation, we designate the expected maximum loads of these policies by $\obj^P$ and $\obj^{\widetilde P}$. Interestingly, along the below proof of this claim, the implementation time of $P$ will be shown to match that of $\widetilde P$, up to factors that are polynomial in $n$ and $T$. 

\begin{lemma}\label{lem:alteruniverse}
Suppose that $U$ and $\widetilde U$ satisfy the  $\delta$-tightness condition. Then, for every $\widetilde U$-policy $\widetilde P$, there exists a $U$-policy $P$ such that $\obj^P\geq (1-\delta)\cdot \obj^{\widetilde P}$.
\end{lemma}

At a high level, our proof shows that given the policy $\widetilde P$, we can determine specific assortments of products from the universe $U$ to be offered at each step to the arriving customer, given the choices of all previous customers, thereby defining a new $U$-policy $P$. Using this elaborate form of simulation, we show that the achieved expected maximum load of the $U$-policy $P$ is at least $1-\delta$ times that of the $\widetilde U$-policy $\widetilde P$.

\begin{proof}
Let $(1^{ P}, \ldots, T^{ P})$ and $(1^{\widetilde P}, \ldots, T^{\widetilde P})$ be two sequences of customers. While the sequence $(1^P,\ldots,T^P)$ encounters the policy $P$ that will be designed below, we make use of the second sequence $(1^{\widetilde P},\ldots,T^{\widetilde P})$ to sample outcomes of the policy $\widetilde P$. In what follows, the load of each product $\widetilde i\in \widetilde U$ will be referring to the number of customers from the sequence $(1^{\widetilde P},\ldots,T^{\widetilde P})$ who selected this product.

\paragraph{Describing the policy ${P}$.}  In order to construct our policy $P$, at each time step $t=1,\ldots,T$, let us describe the assortment offered to customer $t^P$, given the choice outcomes of all previously-arriving customers. To this end, suppose that the choices of customers $1^P,\ldots,(t-1)^P$ and customers $1^{\widetilde P},\ldots,(t-1)^{\widetilde P}$ are already known. Let $\widetilde S_t$ be the assortment offered by the policy $\widetilde P$ to customer $t^{\widetilde P}$. Note that, conditional on the known choices of customers $1^{\widetilde P},\ldots,(t-1)^{\widetilde P}$, this assortment is deterministic. Then, the policy $P$ offers the assortment $S_t = \{i\,\mid\, \widetilde i\in \widetilde S_t\}$.

\paragraph{Simulating the outcome of ${\widetilde P}$.} After offering the assortment $S_t$ to customer $t^P$, we proceed to observe her choice according to the MNL model. Subsequently, in a coupled manner with the choice of customer $t^P$, we simulate the choice of customer $t^{\widetilde P}$, when offered $\widetilde S_t$. Specifically, let $S_t^\uparrow = \{i\in S_t\cup \{0\}\,\mid\,\phi_{\widetilde i}(\widetilde S_t)\geq \phi_i(S_t) \}$ and let $S_t^\downarrow = \{i\in S_t\cup \{0\}\,\mid\,\phi_{\widetilde i}(\widetilde S_t) < \phi_i(S_t) \}$. In addition, let ${\alpha_t = \sum_{i\in S_i^\downarrow} ( \phi_{i}(S_t) - \phi_{\widetilde i}(\widetilde S_t)}) \geq 0$. Now, suppose that product $i \in S_t \cup \{ 0 \}$ is the one  selected by customer $t^P$. If $i\in S_t^\uparrow$, then customer $t^{\widetilde P}$ selects product $\widetilde i$. Otherwise, $i\in S_t^\downarrow$, implying that $\alpha_t>0$. In this case, we proceed as follows:
\begin{itemize}
	\item With probability $\phi_{\widetilde i}(\widetilde S_t)/\phi_i(S_t)$, customer $t^{\widetilde P}$ selects product $\widetilde i$.
 
	\item With probability $1-\phi_{\widetilde i}(\widetilde S_t)/\phi_i(S_t)$, customer $t^{\widetilde P}$ randomly selects one of the products $\{ {\widetilde j} | j\in S_t^\uparrow \}$, where each product $\widetilde j$ is selected with probability $p_j = \frac{\phi_{\widetilde j}(\widetilde S_t)-\phi_j(S_t)}{\alpha_t} \geq 0$. It is worth noting that these terms indeed add up to $1$, since 
	\begin{eqnarray*}
		\sum_{j\in S_t^\uparrow}p_j & = & \frac{1}{\alpha_t} \cdot \sum_{j\in S_t^\uparrow} \left( \phi_{\widetilde j}(\widetilde S_t)- \phi_j(S_t)\right) \\
        & = & \frac{1}{\alpha_t} \cdot \left( \left( 1-\sum_{j\in S_t^\downarrow}\phi_{\widetilde j}(\widetilde S_t) \right) - \left( 1-\sum_{j\in S_t^\downarrow}\phi_j(S_t) \right)\right) \\
                & = & 1.
	\end{eqnarray*}
 \end{itemize}	

\paragraph{Correctness of the simulation.} In what follows, we show that for each product $\widetilde j\in \widetilde S_t\cup \{0\}$, the probability for customer $t^{\widetilde P}$ to select this product, via the simulation process described above, is exactly $\phi_{\widetilde j}(\widetilde S_t)$. For this purpose, we consider two cases:
\begin{itemize}
	\item When $j\in S_t^\downarrow$: In this case, if customer $t^P$ selected some product different from $j$, then customer $t^{\widetilde P}$ cannot select product $\widetilde j$. If customer $t^P$ selected product $j$, which happens with probability $\phi_{ j}(S_t)$, then customer $t^{\widetilde P}$ selects product $\widetilde j$ with probability $\phi_{\widetilde j}(\widetilde S_t)/\phi_{j}(S_t)$. Therefore, the overall probability for customer $t^{\widetilde P}$ to select product $\widetilde j$ is $\phi_{\widetilde j}(\widetilde S_t)$.
 
	\item When $j\in S_t^\uparrow$: There are three cases to examine:
	\begin{enumerate}
		\item[(i)] If customer $t^P$ selected product $j$, with probability $\phi_{j}(S_t)$, then customer $t^{\widetilde P}$ selects $\widetilde j$ with probability $1$.
  
		\item[(ii)] If customer $t^P$ selected product $i\in S_t^\uparrow\setminus\{j\}$, then customer $t^{\widetilde P}$ cannot select product $\widetilde j$ according to the described process.
  
		\item[(iii)] If customer $t^P$ selected some product $i\in S_t^\downarrow$, which happens with probability $\phi_{i}(S_t)$, then with probability $1-\phi_{\widetilde i}(\widetilde S_t)/\phi_i(S_t)$, some random product from $\{ {\widetilde j} | j\in S_t^\uparrow \}$ will be selected, and it will be product $\widetilde j$ with probability $p_j$.
	\end{enumerate}
	Therefore, the overall probability that customer $t^{\widetilde P}$ selects product $\widetilde j$ is given by 
    \[ \phi_{j}(S_t) + \sum_{i\in S_t^\downarrow}\phi_{i}(S_t)\cdot \left(1-\frac{\phi_{\widetilde i}(\widetilde S_t)}{\phi_i(S_t)}\right)\cdot p_j =  \phi_{ j}( S_t) + \alpha_t\cdot p_j =\phi_{\widetilde j}(\widetilde S_t), \]
    where the last equality holds since $p_j = \frac{\phi_{\widetilde j}(\widetilde S_t)-\phi_j(S_t)}{\alpha_t}$.
\end{itemize}

\paragraph{Approximation guarantee of ${P}$.} 
In the remainder of this proof, we show that $\obj^P\geq (1-\delta)\cdot \obj^{\widetilde P}$, i.e., the expected maximum load attained by the policy $P$ is at least $1-\delta$ times that of $\widetilde P$. Similarly to the notation introduced in Section~\ref{subsec:SMLA} for the static formulation, let $X_{it}$ be a Bernoulli random variable, indicating whether customer $t^P$ selects product $i$. This way, the load of each product $i\in U$ with respect to the policy $P$ is $L_i= \sum_{t=1}^TX_{it}$. Similarly, let $X_{\widetilde it}$ be a Bernoulli random variable, indicating whether customer $t^{\widetilde P}$ selects product $\widetilde i$. As such, the load of each product $\widetilde i\in\widetilde U$ with respect to the policy $\widetilde P$ is  $L_{\widetilde i} = \sum_{t=1}^TX_{\widetilde it}$. Finally, let $\widetilde I$ be the random index of the most loaded products in $\widetilde U$, namely, $\widetilde{I} = \argmax_{\widetilde{i} \in \widetilde{U}} L_{\widetilde{i}}$, breaking ties by taking the smallest index. The crucial invariant we establish is captured by the next claim, whose proof is provided in Appendix~\ref{apx:proofcomplement}.

\begin{claim}\label{cl:proofcomplement}
$\E(X_{It})\geq (1-\delta)\cdot\E(X_{\widetilde It})$, for all $t = 1,\ldots,T$.
\end{claim}

Given this result, we conclude the proof by observing that 
\begin{eqnarray*}
\obj^P &= &\E\left(\max_{i\in U}L_i\right)\\
&\geq & \E\left(L_I\right)\\
&=& \sum_{t=1}^T\E\left(X_{It}\right)\\
&\geq& (1-\delta)\cdot \sum_{t=1}^T\E\left(X_{\widetilde It}\right)\\
&=& (1-\delta)\cdot \E\left(L_{\widetilde I}\right)\\
&=& (1-\delta)\cdot\obj^{\widetilde P},
\end{eqnarray*}
where the inequality above is a direct application of Claim \ref{cl:proofcomplement}.
\end{proof}

\subsection{Proof of Claim \ref{cl:proofcomplement}}\label{apx:proofcomplement}

Instead of directly working with  $(X_{it})_{i\in U,t\in[T]}$ and $(X_{\widetilde it})_{\widetilde i\in \widetilde U,t\in[T]}$, we propose a new construction of these random variables, $(\ex_{it})_{i\in U,t\in[T]}$ and $(\ex_{\widetilde it})_{\widetilde i\in \widetilde U,t\in[T]}$, such that
\[ 
    \left((X_{it})_{i\in U,t\in[T]},(X_{\widetilde it})_{\widetilde i\in \widetilde U,t\in[T]}\right) \stackrel{d}{=} \left((\ex_{it})_{i\in U,t\in[T]},(\ex_{\widetilde it})_{\widetilde i\in \widetilde U,t\in[T]}\right). \]
However, in this construction, the choices of customers $1^P,\ldots,T^P$ do not affect those of customers $1^{\widetilde P}, \ldots, T^{\widetilde P}$. In particular, we will first sample $(\ex_{\widetilde it})_{\widetilde i\in \widetilde U,t\in[T]}$, and only then sample  $(\ex_{it})_{i\in U,t\in[T]}$ in a coupled manner.

\paragraph{Stage 1: Constructing ${(\ex_{\widetilde it})_{\widetilde i\in \widetilde U,t\in[T]}}$.} First, in order to construct the policy $\widetilde P$, and hence the choices $(\ex_{\widetilde it})_{\widetilde i\in \widetilde U,t\in[T]}$ of customers $1^{\widetilde P},\ldots, T^{\widetilde P}$, upon the arrival of each customer $t^{\widetilde P}$, we observe the choices of all previous customers, and  use the policy $\widetilde P$ to determine the assortment $\widetilde S_t$ that will be offered to this customer. Then, we sample the choice $(\ex_{\widetilde it})_{\widetilde i\in \widetilde U}$ of customer $t^{\widetilde P}$ according to the MNL choice model, where each product $\widetilde{i} \in {\widetilde S}_t$ has a probability of $\phi_{\widetilde{i}}( {\widetilde S}_t )$ to be the one selected.  

\paragraph{Stage 2: Constructing ${(\ex_{it})_{i\in U,t\in[T]}}$.} Once $(\ex_{\widetilde it})_{\widetilde i\in \widetilde U,t\in[T]}$ have already been determined, let us describe how to construct the choices $(\ex_{it})_{i\in U,t\in[T]}$ of customers $1^P,\ldots, T^P$ according to the policy $P$. To this end, upon the arrival of each customer $t^{P}$, we determine her choice $(\ex_{it})_{i\in U}$ in a coupled fashion. As in Appendix~\ref{apx:altering}, we will make use of $S_t^\uparrow = \{i\in S_t\cup \{0\}\,\mid\,\phi_{\widetilde i}(\widetilde S_t)\geq \phi_i(S_t) \}$, $S_t^\downarrow = \{i\in S_t\cup \{0\}\,\mid\,\phi_{\widetilde i}(\widetilde S_t) < \phi_i(S_t) \}$, and ${\alpha_t = \sum_{i\in S_i^\downarrow} ( \phi_{i}(S_t) - \phi_{\widetilde i}(\widetilde S_t)}) \geq 0$. Now, suppose that product $\widetilde i \in \widetilde S_t \cup \{ 0 \}$ is the one selected by customer $t^{\widetilde P}$, meaning that $\ex_{\widetilde it} = 1$. If $i\in S_t^\downarrow$, then customer $t^{P}$ selects product $i$. Otherwise, $i\in S_t^\uparrow$, and we proceed as follows:
\begin{itemize}
	\item With probability $\phi_i(S_t)/\phi_{\widetilde i}(\widetilde S_t)$, customer $t^{P}$ selects product $i$, i.e., $\ex_{it} = 1$.
 
	\item With probability $1-\phi_i(S_t)/\phi_{\widetilde i}(\widetilde S_t)$, customer $t^{P}$ selects one of the products in $S_t^\downarrow$, where each product $j$ is selected with probability $q_j = \frac{\phi_j(S_t)-\phi_{\widetilde j}(\widetilde S_t)}{\alpha_t} \geq 0$. Similarly to Appendix~\ref{apx:altering}, it is easy to verify that these terms add up to $1$.
 \end{itemize}	

\paragraph{Proving equality in distribution.} We proceed by showing that $((\ex_{it})_{i\in U,t\in[T]},(\ex_{\widetilde it})_{\widetilde i\in \widetilde U,t\in[T]})$ and $((X_{it})_{i\in U,t\in[T]},(X_{\widetilde it})_{\widetilde i\in \widetilde U,t\in[T]})$ are indeed equal in distribution. In particular, we show that at each step, the joint choice probabilities of the customers $t^P$ and $t^{\widetilde P}$ are identical for both constructions. Formally, we argue that for all $i\in U$, $\widetilde j\in \widetilde U$, and $t\in[T]$, 
\begin{equation}\label{eq:equation219}    \P(\ex_{it}=1,\ex_{\widetilde jt}=1) = \P(X_{it}=1,X_{\widetilde jt}=1).
\end{equation}
Note that we do not need to consider events of the form $X_{it}=0$, since they can be written as a disjoint union of the events $X_{jt}=1$ for $j\neq i$, i.e., $\{X_{it}=0\} = \bigvee_{j\neq i}\{X_{jt}=1\}$. We prove Equation~\eqref{eq:equation219} via the following case analysis.
\begin{itemize}
    \item Case 1: $i=j$:\begin{itemize}
        \item If $i\in S_t^{\uparrow}$: Then $X_{it}=1$ implies $X_{\widetilde it}=1$. Therefore, $$ \P(X_{it}=1,X_{\widetilde it}=1)=\P(X_{it}=1)=\phi_i(S_t).$$
        On the other hand,
        \begin{equation*}
            \P(\ex_{it}=1,\ex_{\widetilde it}=1) = \P(\ex_{\widetilde it}=1)\cdot \P(\ex_{it}=1|\ex_{\widetilde it}=1) = \phi_{\widetilde i}(\widetilde S_t)\cdot \frac{\phi_{i}(S_t)}{\phi_{\widetilde i}(\widetilde S_t)} = \phi_i(S_t).
        \end{equation*}

        \item If $i\in S_t^{\downarrow}$: Then,
        $$            \P(X_{it}=1,X_{\widetilde it}=1)= \P(X_{it}=1)\cdot\P(X_{\widetilde it}=1|X_{it}=1) = \phi_i(S_t)\cdot \frac{\phi_{\widetilde i}(\widetilde S_t)}{\phi_i(S_t)} = \phi_{\widetilde i}(\widetilde S_t).
        $$
        On the other hand,  if $\ex_{\widetilde it}=1$ then $\ex_{it} =1$, and therefore
        $$            \P(\ex_{it}=1,\ex_{\widetilde it}=1) =\P(\ex_{\widetilde it}=1) =  \phi_{\widetilde i}(\widetilde S_t).
        $$
    \end{itemize}
    
    \item Case 2: $i\neq j$:
    \begin{itemize}
        \item If $i\in S_t^{\uparrow}$, then $X_{it}=1$ implies $X_{\widetilde it}=1$, and $\ex_{it}=1$ implies $\ex_{\widetilde it}=1$. Therefore, since $i \neq j$,
        $$        \P(X_{it}=1,X_{\widetilde jt}=1)=\P(\ex_{it}=1,\ex_{\widetilde jt}=1)=0.$$
        \item If $j\in S_t^{\downarrow}$, then for similar reasons, 
         $$
        \P(X_{it}=1,X_{\widetilde jt}=1)=\P(\ex_{it}=1,\ex_{\widetilde jt}=1)=0.$$
        \item If $i\in S_t^{\downarrow}$ and $j\in S_t^{\uparrow}$, then
            \begin{eqnarray*}
                \P(X_{it}=1,X_{\widetilde jt}=1)&= &\P(X_{it}=1)\cdot\P(X_{\widetilde jt}=1|X_{it}=1)\\
                &=&\phi_i(S_t)\cdot\left(1-\frac{\phi_{\widetilde i}(\widetilde S_t)}{\phi_i(S_t)}\right)\cdot p_j\\
                &=&\frac{1}{\alpha_t}\cdot (\phi_i(S_t) - \phi_{\widetilde i}(\widetilde S_t))\cdot(\phi_{\widetilde j}(\widetilde S_t)-\phi_j(S_t)).
            \end{eqnarray*}
            On the other hand, 
            \begin{eqnarray*}
                \P(\ex_{it}=1,\ex_{\widetilde jt}=1)&= &\P(\ex_{\widetilde jt}=1)\cdot\P(\ex_{it}=1|\ex_{\widetilde jt}=1)\\
                &=&\phi_{\widetilde j}(S_t)\cdot\left(1-\frac{\phi_j(S_t)}{\phi_{\widetilde j}(\widetilde S_t)}\right)\cdot q_i\\
                &=&\frac{1}{\alpha_t}\cdot (\phi_i(S_t) - \phi_{\widetilde i}(\widetilde S_t))\cdot(\phi_{\widetilde j}(\widetilde S_t)-\phi_j(S_t)).
            \end{eqnarray*}
    \end{itemize}
\end{itemize}

\paragraph{Concluding the proof.} In the construction we have just described, the choices of customers $1^P,\ldots,T^P$ in stage~2 obviously do not affect the policy $\widetilde P$ in stage~1. Thus, we can initially sample the choices $\widetilde \ex = (\ex_{\widetilde it})_{\widetilde i\in \widetilde U,t\in[T]}$ of customers $1^{\widetilde P}, \ldots, T^{\widetilde P}$, and then use this realization to sample the choices $(\ex_{it})_{i\in U,t\in[T]}$ of customers $1^P,\ldots, T^P$. The important observation is that, for every possible realization $\widetilde x$ of $\widetilde \ex$, we have
\begin{eqnarray}
        \E(\ex_{I t}\,|\,\widetilde \ex = \widetilde x) &=& \P(\ex_{I t}=1\,|\,\widetilde \ex = \widetilde x) \nonumber\\
    &=&\P(\ex_{I t}=1\,|\,\ex_{\widetilde I t}=1,\widetilde \ex = \widetilde x)\cdot \P(\ex_{\widetilde I t}=1\,|\,\widetilde \ex = \widetilde x)\nonumber\\
    && \mbox{}+ \P(\ex_{I t}=1\,|\,\ex_{\widetilde I t}=0,\widetilde \ex = \widetilde x)\cdot \P(\ex_{\widetilde I t}=0\,|\,\widetilde \ex = \widetilde x)\nonumber\\
    & \geq &\P(\ex_{I t}=1\,|\,\ex_{\widetilde I t}=1,\widetilde \ex = \widetilde x)\cdot \P(\ex_{\widetilde I t}=1\,|\,\widetilde \ex = \widetilde x) \nonumber\\    
    &\geq &(1-\delta)\cdot \P(\ex_{\widetilde I t}=1\,|\,\widetilde \ex = \widetilde x)\nonumber\\
    &=&(1-\delta)\cdot \E(\ex_{\widetilde I t}\,|\,\widetilde \ex = \widetilde x). \label{eqn:cond_sample_path}
\end{eqnarray}
Here, the second inequality holds since $\P(\ex_{I t}=1\,|\,\ex_{\widetilde I t}=1,\widetilde \ex = \widetilde x) \geq 1-\delta$.
This claim is a direct consequence of our simulation process. Indeed, conditional on $\widetilde \ex = \widetilde x$, the random index $\widetilde{I}$ of the most loaded product with respect to $\widetilde{P}$ is clearly deterministic. As such, given that customer $t^{\widetilde{P}}$ selects product $\widetilde{I}$ (i.e., $\ex_{\widetilde I t}=1$), there are two cases: Either $I \in S_t^\downarrow$, in which case $\ex_{it}=1$ almost surely, or $I\in S_t^{\uparrow}$, in which case, $\ex_{it}=1$ with probability $\phi_{I}(S_t)/\phi_{\widetilde I}(\widetilde S_t) \geq 1-\delta$, where the last inequality follows from our $\delta$-tightness condition.

As a consequence, by summing inequality~\eqref{eqn:cond_sample_path} over all possible sample paths $\widetilde x$, weighted by their probability, we have $    \E(\ex_{It}) \geq (1-\delta)\cdot \E(\ex_{\widetilde I t})$. Finally, since $((\ex_{it})_{i\in U,t\in[T]},(\ex_{\widetilde it})_{\widetilde i\in \widetilde U,t\in[T]})$ and $((X_{it})_{i\in U,t\in[T]},(X_{\widetilde it})_{\widetilde i\in \widetilde U,t\in[T]})$ are equal in distribution, we deduce that $\E(X_{It}) \geq (1-\delta)\cdot \E(X_{\widetilde I t})$.

\subsection{Concluding the proof of Theorem \ref{thm:quasitime}}\label{apx:proofquasitime}

In order to prove Theorem~\ref{thm:quasitime}, we first propose an efficient representation of constrained vectors, allowing us to implement our overall approach in $O( n^{ O_{\eps}( \log^3 n) } )$ time. Subsequently, we prove that the expected maximum load obtained upon utilizing the policy $A$, constructed in Section~\ref{subsec:policy}, is within factor $1-\eps$ of the optimal expected maximum load.   

\paragraph{Implementation and running time analysis.} In what follows, we will be assuming that, under the low-weight regime, the number of arriving customers $T$ is polynomial in $n$ and $1/\eps$. This is a consequence of the next claim, which we prove in Appendix~\ref{apx:polywlog}. 

\begin{claim}\label{cl:polywlog}
Under the low-weight regime, when  $T\geq 576 n^3/\eps^8$, by offering the whole universe of products to every customer we attain an expected maximum load of at least $(1-\eps) \cdot \opt^{\dpp}(\Nc)$.
\end{claim}

Let $\cal S$ be the collection of states considered by the reduced dynamic program in Step~$3$. Each such state corresponds to a pair $(t,\bl)$, where $t$ is the remaining number of customers, and $\bl \in \const$ is our current load vector. We remind the reader that $\const$ stands for the collection of constrained load vectors, namely, those where each product has a load of at most $\thr/\Teps$. We start by providing an efficient representation of each state $(t,\bl)\in \cal S$. To this end, for every $j\in \{0,\ldots, J\}$ and $m\in \{0,\ldots,\thr/\Teps\}$, let $N_{j,m}(\bl)$ be the number of products with weight $v_{\min}\cdot (1+\Teps)^j$, whose load with respect to $\bl$ is precisely $m$, i.e.,  
\[ N_{j,m}(\bl) = \left| \{\widetilde i\in \widetilde U\mid v_{\widetilde i} = v_{\min}\cdot (1+\Teps)^j\text{ and } \ell_i = m\} \right| . \]
Given this notation, we represent each vector $\bl\in\const$ by its corresponding vector $N(\bl)= (N_{j,m}(\bl)\,\mid\,j\in \{0,\ldots, J\}\text{ and } m\in \{0,\ldots,\thr/\Teps\})$. Clearly, the collection $\{N(\bl) \,\mid\, \bl\in\const\}$ consists of only $O(n^{O(J\thr/\Teps)})$ vectors. Therefore, by representing each state $(t,\bl)\in \cal S$ by its corresponding vector $(t,N(\bl))$, and recalling that
$\thr = O_{\eps}( \log^2(nT) )$, $J = O_{\eps}(\log n)$, and $T< 576 n^3/\eps^8$, our search space size becomes $O(T n^{O_{\eps}(J\thr)} ) = O( n^{ O_{\eps}( \log^3 n ) } )$.

As a side note, we remark that this representation is not injective. However, it encompasses all the information needed to solve our reduced dynamic program. Indeed, since all products with the same weight and load are interchangeable, it is easy to see that for each pair of constrained vectors $\bl^1$ and $\bl^2$ that share the same representation, we can transform $\bl^1$ into $\bl^2$ by performing a finite number of permutations on the names of the products that share the same weight and load.

To conclude that Steps~1-4 can be performed in $O( n^{ O_{\eps}( \log^3 n ) } )$ overall time, it remains to argue that the optimization problem in Equation~\eqref{eq:reducedDP} can be efficiently solved, as stated in the next claim. This result will be established via a reduction to an appropriately constructed revenue maximization problem under the MNL model, which is solvable in polynomial time (see, e.g., \cite{talluri2004revenue}). The details of this proof are included in Appendix \ref{apx:revenuemax}.

\begin{lemma}\label{cl:revenuemax}
Problem \eqref{eq:reducedDP} can be solved to optimality in $O(n)$ time.
\end{lemma}

\paragraph{Approximation guarantee.} In the remainder of this section, we show that the policy $A$, as introduced in Section \ref{subsec:policy}, yields the desired performance guarantee. In other words, we argue that $\obj^A\geq (1-4 \eps)\cdot \opt^{\dpp}(\Nc)$. To this end, we establish the following sequence of claims:
\begin{itemize}
    \item {\em Loss due to dropping tiny-weight products.} We remind the reader that the universe of products $U$ was created in Step~1 by eliminating all products whose preference weight is at most $\eps^2v_{\max}/n$. Our first claim is that this alteration leads to losing at most an $\Teps$-fraction of the optimal expected maximum load. The proof of the next lemma is included in Appendix~\ref{apx:quasi1}. 
    \begin{lemma}\label{lem:quasi1}
    $\opt^{\dpp}(U)\geq (1-\Teps)\cdot \opt^{\dpp}(\Nc)$.
    \end{lemma}

    \item {\em Loss due to altering product weights.} Recall that in Step~2, each product $i \in U$ was replaced by a corresponding product $\widetilde{i} \in \widetilde U$ whose weight is the left endpoint of the bucket containing $v_i$. In Lemma \ref{lem:quasi2}, whose proof appears in Appendix \ref{apx:quasi2}, we show that the optimal expected maximum loads of  $U$ and $\widetilde U$ are within factor $1-\Teps$ of one another. 
    \begin{lemma}\label{lem:quasi2}
        $(1-\Teps)\cdot \opt^{\dpp}(U) \leq\opt^{\dpp}(\widetilde U) \leq \frac{1}{1-\Teps}\cdot \opt^{\dpp}(U)$.
    \end{lemma}

    \item {\em Loss due to considering truncated policies.} In Step~3, we make use of our reduced dynamic program to compute an optimal truncated $\widetilde U$-policy, $\widetilde A$. In the following lemma, we prove that $\widetilde A$ is in fact a $(1-\Teps)$-approximate $\widetilde U$-policy. The proof of this result is deferred to Appendix~\ref{apx:quasi3}. 
    \begin{lemma}\label{lem:quasi3}
        $\obj^{\widetilde A}\geq (1-\Teps)\cdot \opt^{\dpp}(\widetilde U)$.
    \end{lemma}
\end{itemize}

In conclusion,  it follows that the expected maximum load of the policy $A$ is 
\begin{eqnarray*}
\obj^A &\geq & (1-\Teps)\cdot \obj^{\widetilde A} \\
& \geq & (1-\Teps)^2\cdot \opt^{\dpp}(\widetilde U) \\
& \geq & (1-\Teps)^3\cdot \opt^{\dpp}(U) \\
& \geq & (1-\Teps)^4\cdot \opt^{\dpp}(\Nc) \\
& \geq & (1-4\eps)\cdot \opt^{\dpp}(\Nc).
\end{eqnarray*}
Here, first inequality is simply a restatement of Equation \eqref{eq:lemmaequation}. The second inequality follows from Lemma \ref{lem:quasi3}. In the third inequality, we plug in the result of Lemma \ref{lem:quasi2}. The fourth inequality is a consequence of Lemma \ref{lem:quasi1}. Finally, the last inequality follows from Bernoulli's inequality.

\subsection{Proof of Claim~\ref{cl:polywlog}}\label{apx:polywlog}

Let us first notice that
\[ \alpha\leq\frac{12\ln(nT)}{T\eps^3}\leq  \frac{12\sqrt{nT}}{T\eps^3} \leq \frac{\eps}{2n}\leq \frac{\eps/n}{1+\eps/n}. \]
Here, the first inequality holds since $T\alpha \leq 12\ln(nT)/\eps^3$, due to being in the low-weight regime. The second inequality comes from applying the identity $\ln x\leq \sqrt{x}$ for all $x>0$. Finally, the third inequality is obtained by recalling that $T\geq 576 n^3/\eps^8$, according to the claim's hypothesis. Consequently, since $\alpha = v_{\max}/(1+v_{\max})$, we must have $v_{\max}\leq \eps/n$, implying in turn that $v(\Nc)\leq \eps$.

Now, circling back to Lemma~\ref{lem:augment}, let $A$ be an optimal dynamic policy for the universe $\Nc$, and let $B$ being the policy that statically offers the whole universe to each arriving customer. Then, we almost surely have $S^A_t\subseteq \Nc = S^B_t $ and $v( S^B_t \setminus S^A_t) \leq v(\Nc) \leq \epsilon$, as shown above. It follows that both conditions of this lemma are satisfied, and therefore $\obj^B \geq \frac1{1+\epsilon}\cdot\obj^A \geq (1 - \eps) \cdot \opt^{\dpp}(\Nc)$, as desired.

\subsection{Proof of Lemma \ref{cl:revenuemax}} \label{apx:revenuemax}

Noting that $\phi_0(S) = 1-\sum_{i\in S}\phi_i(S)$ for any $S\subseteq \widetilde U$, the optimization problem~\eqref{eq:reducedDP} can be reformulated as:
\[ M_{t}(\mathbf{\bl})=M_{t-1}(\mathbf{\bl})+\max_{S\subseteq \widetilde U}\left(\sum\limits_{i\in S}\left(M_{t-1}(\mathbf{\bl}+\mathbf{e}_i)-M_{t-1}(\mathbf \bl)\right)\cdot \phi_i(S)\right). \]
Therefore, letting $r_i = M_{t-1}(\mathbf{\bl}+\mathbf{e}_i)-M_{t-1}(\mathbf \bl) \geq 0$ be the so-called price of each product $i\in \widetilde U$, we are left with computing an optimal solution to $\max_{S\subseteq \widetilde U}(\sum_{i\in S}r_i\cdot \phi_i(S))$. We have therefore obtained an instance of the revenue maximization problem under the Multinomial Logit model, which is well-known to be solvable in polynomial time (see, e.g., \cite{talluri2004revenue}).

\subsection{Proof of Lemma \ref{lem:quasi1}}\label{apx:quasi1} 

According to Lemma~\ref{lem:subadd}, we know that $\opt^{\dpp}(\cdot)$ is a subadditive function, implying in particular that    \begin{equation}\label{eq:quasisubadd}        \opt^\dpp(\Nc)\leq\opt^{\dpp}(U) + \opt^{\dpp}(\Nc \setminus U),
\end{equation}
and therefore, it suffices to show that $\opt^{\dpp}(\Nc\setminus U) \leq  \Teps\cdot \opt^{\dpp}(\Nc)$.

For this purpose, by definition of $U$, every product in $\Nc\setminus U$ has a preference weight of at most $\Teps^2\cdot v_{\max}/n$. Therefore, the random number of purchases across all products in $\Nc\setminus U$ is stochastically smaller than a Binomial random variable with $T$ trials and success probability $\frac{ v(\Nc\setminus U) }{ 1 + v(\Nc\setminus U) }$. Additionally, the  maximum load of any $\Nc\setminus U$-policy is upper-bounded by the total number of purchases. Consequently, 
\begin{eqnarray*}
    \opt^{\dpp}(\Nc\setminus U) & \leq & T\cdot \frac{ v(\Nc\setminus U) }{ 1 + v(\Nc\setminus U) } \\
    & \leq & T\cdot \frac{\Teps^2\cdot v_{\max}}{1+\Teps^2\cdot v_{\max}} \\
    & \leq & \Teps\cdot T\cdot \frac{v_{\max}}{1+v_{\max}} \\
    & \leq & \Teps\cdot \opt^{\dpp}(\Nc).
\end{eqnarray*}
Here, the second inequality holds since  $v(\Nc\setminus U) \leq \Teps^2v_{\max}$, as the weight of each product in $\Nc\setminus U$ is at most $\Teps^2v_{\max}/n$ and there are at most $n$ such products. For the third inequality, it is easy to verify that $\frac{\Teps}{1+\Teps^2\cdot v_{\max}} \leq \frac{1}{1+v_{\max}}$ when $\Teps < 1$ and $v_{\max}\leq 1/\Teps$. In the last inequality, we bound $Tv_{\max}/(1+v_{\max})$ by $\opt^{\dpp}(\Nc)$, since $Tv_{\max}/(1+v_{\max})$ corresponds to the expected maximum load when statically offering only the heaviest product in $\Nc$. The latter is obviously dominated by the expected maximum load attained by an optimal dynamic policy.

\subsection{Proof of Lemma \ref{lem:quasi2}}\label{apx:quasi2} 

We argue that this result is a direct consequence of Lemma \ref{lem:alteruniverse}. We start by proving the first inequality, $(1-\Teps)\cdot \opt^{\dpp}(U) \leq\opt^{\dpp}(\widetilde U) $. Let $P$ be an optimal policy for the universe of products $U$, meaning in particular that $\obj^{P}=\opt^{\dpp}(U)$. By definition of $\widetilde U$, we have $(1-\eps) \cdot v_{i}\leq v_{\widetilde i}\leq v_i$. Therefore, for any assortment $S\subseteq U$ and for any product $i\in S$,

$$
            \phi_{\widetilde i}(\widetilde S) = \frac{v_{\widetilde i}}{1+\sum_{\widetilde j\in \widetilde S}v_{\widetilde j}} \geq \frac{(1-\Teps)\cdot v_{i}}{1+\sum_{j\in S}v_j} = (1-\Teps)\cdot \phi_i(S),
        $$
where $\widetilde S = \{\widetilde j  \in \widetilde U   \mid j\in S\}$. 
Consequently, according to Lemma \ref{lem:alteruniverse}, there exists a policy $\widetilde P$ for the universe $\widetilde U$ such that $\obj^{\widetilde P}\geq (1-\Teps)\cdot \obj^{P}$. In turn, 
\[ {\opt^{\dpp}(\widetilde U)\geq \obj^{\widetilde P} \geq (1-\Teps)\cdot \obj^P= (1-\Teps) \cdot \opt^{\dpp}(U)}. \]

To derive the second inequality, $\opt^{\dpp}(\widetilde U) \leq \frac{1}{1-\Teps}\cdot \opt^{\dpp}(U)$, note that for every assortment $\widetilde S\subseteq U$ and for every product $\widetilde i\in \widetilde S$, we have

$$
        \phi_i(S) = \frac{v_i}{1+\sum_{j\in S}v_j}\geq \frac{v_{\widetilde i}}{1+\frac{1}{1-\Teps}\cdot\sum_{j\in \widetilde S}v_{\widetilde j}}  \geq (1-\Teps)\cdot \phi_{\widetilde i}(\widetilde S),
$$
where $S = \{j \in U \mid\widetilde j\in \widetilde S\}$. Therefore, the exact same argument as in the first inequality proves that ${\opt^{\dpp}( U) \geq (1-\Teps)\cdot\opt^{\dpp}( \widetilde U)}$.

\subsection{Proof of Lemma \ref{lem:quasi3}}\label{apx:quasi3} 

Let $\widetilde P$ be an optimal $\widetilde U$-policy; in particular, $\obj^{\widetilde P} = \opt^{\dpp}(\widetilde U)$.  Let $\widetilde B$ be the policy obtained by truncating the optimal dynamic policy $\widetilde P$, namely, the one that offers precisely the same assortments as $\widetilde P$ while the maximum load is smaller than $\beta = \thr/\Teps$. Once we hit this threshold, the empty set will be offered to all  remaining customers. Since $\widetilde B$ is a truncated policy and $\widetilde A$ is an optimal truncated policy, we trivially have $\obj^{\widetilde A}\geq \obj^{\widetilde B}$, meaning that it suffices to show that $\obj^{\widetilde B} \geq (1 - \eps) \cdot \obj^{\widetilde P}$.

For this purpose, let $\widetilde M$ be the random variable specifying the maximum load attained by employing the policy $\widetilde P$; in particular, we have $\obj^{\widetilde P} = \E(\widetilde M)$. Also, let $\widetilde M^{\comm} = \min(\beta, \widetilde M)$, noting that the latter random variable is exactly the maximum load 
attained by employing  $\widetilde B$, meaning that $\obj^{\widetilde B} = \E(\widetilde M^{\comm})$. Our analysis  will be based on the following auxiliary claims, whose proofs appear at the end of this section.

\begin{claim}\label{cl:upbd}
$\E(\widetilde M) \leq (\beta+\E(\widetilde M))\cdot \P(\widetilde M\geq \beta)+ \E(\widetilde M\mid \widetilde M< \beta)\cdot \P(\widetilde M< \beta)$.
\end{claim}

\begin{claim}\label{cl:lowbd}
$\E(\widetilde M^{\comm}) \geq \frac{\beta}{\beta+\E(\widetilde M)}\cdot(
(\beta+\E(\widetilde M))\cdot \P(\widetilde M\geq \beta)+ \E(\widetilde M\mid \widetilde M< \beta)\cdot \P(\widetilde M< \beta)
        )$.
\end{claim}

\begin{claim} \label{clm:beta_frac_eps}
$\frac{\beta}{\beta+\E(\widetilde M)} \geq 1-\Teps$.
\end{claim}

We conclude by observing that these claims suffice to show that $\obj^{\widetilde B} \geq (1 - \eps) \cdot \obj^{\widetilde P}$, since
\[ \obj^{\widetilde B} = \E(\widetilde M^{\comm})\geq \frac{\beta}{\beta+\E(\widetilde M)}\cdot  \E(\widetilde M) \geq (1 - \eps) \cdot \E(\widetilde M) = (1 - \eps) \cdot \obj^{\widetilde P}, \]
where the first inequality follows by combining Claims~\ref{cl:upbd} and~\ref{cl:lowbd}, and the second inequality is obtained by plugging in Claim~\ref{clm:beta_frac_eps}.

\paragraph{Proof of Claim~\ref{cl:upbd}.} 
Our proof begins by arguing that 
\begin{equation}\label{eq:claim2}
        \E(\widetilde M\mid \widetilde M\geq \beta)-\beta\leq \E(\widetilde M),
\end{equation}
noting that $\E(\widetilde M\mid \widetilde M\geq \beta)-\beta$ represents the expected increase in the maximum load, starting from the point where the threshold $\beta$ is reached, all with respect to the policy $\widetilde P$. To bound the latter quantity, let ${\cal T}$ be the (random) index of the  customer following the one for which a load of $\beta$ is attained. In other words, ${\cal T}$ is the minimal stage index for which the current maximum load is at least $\beta$. As such, $  \E(\widetilde M\mid \widetilde M\geq \beta, {\cal T})-\beta$ is upper-bounded by the expected maximum load considering only customers ${\cal T}+1,\ldots,T$, which is trivially bounded by the optimal expected maximum load considering customers $1,\ldots,T$, i.e.,
$$
         \E(\widetilde M\mid \widetilde M\geq \beta, {\cal T})-\beta\leq \opt^{\dpp}(\widetilde U).
    $$
    By recalling that $\opt^{\dpp}(\widetilde U) = \E(\widetilde M)$, inequality~\eqref{eq:claim2} follows by introducing the expectation over ${\cal T}$ and using the tower property. Consequently,
    \begin{eqnarray*}
    \E(\widetilde M) & = & \E(\widetilde M\mid \widetilde M\geq \beta)\cdot \P(\widetilde M\geq \beta)+ \E(\widetilde M\mid \widetilde M< \beta)\cdot \P(\widetilde M< \beta) \\
    & \leq & (\beta+\E(\widetilde M))\cdot \P(\widetilde M\geq \beta)+ \E(\widetilde M\mid \widetilde M< \beta)\cdot \P(\widetilde M< \beta) .
    \end{eqnarray*}

\paragraph{Proof of Claim~\ref{cl:lowbd}.} The desired claim is obtained by noting that
\begin{eqnarray*}
\E(\widetilde M^{\comm}) & = & \E(\widetilde M^\comm \mid\widetilde M\geq \beta)\cdot \P(\widetilde M\geq \beta)+ \E(\widetilde M^\comm\mid \widetilde M< \beta)\cdot \P(\widetilde M< \beta) \\
& = & \beta\cdot \P(\widetilde M\geq \beta)+ \E(\widetilde M\mid \widetilde M< \beta)\cdot \P(\widetilde M< \beta)\\
&\geq & \frac{\beta}{\beta+\E(\widetilde M)}\cdot\left(             (\beta+\E(\widetilde M))\cdot \P(\widetilde M\geq \beta)+ \E(\widetilde M\mid \widetilde M< \beta)\cdot \P(\widetilde M< \beta) \right),
\end{eqnarray*}
where the second equality holds since $\E(\widetilde M^\comm \mid\widetilde M\geq \beta) = \beta$ and $\E(\widetilde M^\comm \mid\widetilde M\leq \beta) = \E(\widetilde M \mid\widetilde M\leq \beta)$, by definition of $\widetilde M^\comm$. 

\paragraph{Proof of Claim~\ref{clm:beta_frac_eps}.} We begin by observing that     
\[ \frac{\beta}{\beta+\E(\widetilde M)} = \frac{\thr/\Teps}{\thr/\Teps+\opt^{\dpp}(\widetilde U)} = \frac{1}{1+\Teps\cdot \frac{\opt^{\dpp}(\widetilde U)}{\thr}}, \]
and it therefore remains to show that $\opt^{\dpp}(\widetilde U)\leq \frac{\thr}{1-\Teps}$. For this purpose, note that
$$
        \opt^{\dpp}(\widetilde U)\leq \frac{1}{1-\Teps} \cdot \opt^{\dpp}(U)\leq \frac{1}{1-\Teps} \cdot\opt^{\dpp}(\Nc)\leq \frac{\thr}{1-\Teps},
    $$
where the first and third  inequalities follow from Lemmas~\ref{lem:quasi2} and~\ref{lem:low}, respectively.

\end{document}